\documentclass[10pt,leqno]{amsart}%,draft

\usepackage{amsmath,amsfonts,amssymb,mathtools}
\usepackage{amscd}
\newcommand{\rarrow}[1]{{\buildrel #1 \over \longrightarrow}}

\def\Ker{{\rm Ker }}
\def\Z{{\mathbb Z}}

\def\H{{\mathbb H}}

\def\P{{\rm P}}

\def\mand{\textrm{ \ and \ }}
\def\ext{\mathrm{Ext}}
\def\coe{\mathrm{Coe}}

\makeatletter
  \@addtoreset{equation}{section}
\makeatother
\makeatletter

\newtheorem{thm}{Theorem}[section]

\newtheorem{thm0}[thm]{Theorem}

\newtheorem{prop}[thm]{Proposition}

\newtheorem{lem}[thm]{Lemma}

\newtheorem{rem}[thm]{Remark}

\begin{document}

\title{Relations in the $24$-th homotopy groups of spheres}
\author[Miyauchi]{Toshiyuki Miyauchi}
\author[Mukai]{Juno Mukai}
\email[Miyauchi]{tmiyauchi@fukuoka-u.ac.jp}
\email[Mukai]{jmukai@shinshu-u.ac.jp}
\address[Miyauchi]{Department~of~Applied~Mathematics,
Faculty~of~Science,
Fukuoka~University,
Fukuoka, 814-0180, Japan}
\address[Mukai]{Shinshu University,
Matsumoto, 
Nagano Pref., 390-8621, Japan}
\date{}


\allowdisplaybreaks

\begin{abstract}
The main purpose of this note is to give a proof of the fact that the Toda brackets \ $\langle\bar{\nu},\sigma,\bar{\nu}\rangle$ and $\langle\nu,\eta, \bar{\sigma}\rangle$ are not trivial. This is an affirmative answer of M.~Mahowald's Conjecture (J. Mukai, Determination of the $P$-image by Toda brackets, Geometry and Topology Monographs \textbf{13}(2008), 355--383). 
The second purpose is to determine the relations including
$\bar{\nu}_6\omega_{14}$ in $\pi^6_{30}$
and
$\bar{\nu}_7\omega_{15}$ in $\pi^7_{31}$.
To this end, we provide relations between the Toda bracket and the $J$-homomorphism, 
and between the Toda bracket and the generalized $P$-homomorphism.
\end{abstract}
\maketitle
%\noindent
\section{Introduction}
In this note, we determine some relations in the $24$-th homotopy groups of spheres,
these are necessary to continue works on unstable homotopy groups of spheres 
\cite{Mi, MMO, MT, Od, T}.
The authors used a 2012 preprint version of this note in \cite{MiM2}
 to determine the 2-primary components of the 32-stem homotopy group of spheres. 
Although there have been some changes from the preprint version in the note, 
these do not affect the results in \cite{MiM2}.

We denote by $\pi_k^n$ the direct sum of the torsion-free part and the
$2$-primary component of $\pi_k(S^n)$ and $\pi^S_k=\lim\limits_{n\to\infty}\pi_{n+k}^n$.  
Let
$\iota\in\pi^S_0$, $\eta\in\pi^S_1$, $\nu\in\pi^S_3$, $\sigma\in\pi^S_7$, 
$\bar{\nu}$, $\varepsilon\in\pi^S_8$, $\mu\in\pi^S_9,\ \zeta\in\pi^S_{11}$, 
$\kappa\in\pi^S_{14}$, $\rho\in\pi^S_{15}$, $\omega,\ \eta^*\in\pi^S_{16}$, 
$\bar{\mu}\in\pi^S_{17}$, $\nu^*$, $\xi\in\pi^S_{18}$, $\bar{\zeta}$, 
$\bar{\sigma}\in\pi^S_{19}$, $\bar{\kappa}\in\pi^S_{20}$, 
$\bar{\rho}\in\pi^S_{23}$, $\delta\in\pi^S_{24}$ %, $\mu_{3,*}\in\pi^S_{25}$\ 
be the generators \cite{T,MT,MMO}. 

For  elements  $a_1, \cdots, a_n$ of an abelian group $G$, we denote  by $\{a_1, \cdots, a_n\}$ a group generated by $a_1, \cdots, a_n$. For the group $\Z$ of integers, we set $\Z_n=\Z/n\Z$, and we denote by $(\Z_n)^k$ the direct sum $\Z_n\oplus\cdots\oplus\Z_n$ of $k$-copies of $\Z_n$.
Denote by $\{\alpha,\beta,\gamma\}_n$, $\{\alpha,\beta,\gamma\}=\{\alpha,\beta,\gamma\}_0$ and $\langle\alpha,\beta,\gamma\rangle$ 
the Toda bracket \cite[p.~9,~10,~32]{T}
and
$\mathrm{Ind}\, A$ the indeterminacy of a coset of $A$.
If $\{\alpha,\beta,\gamma\}_n$ consists a single element $\delta$,
that is, $\delta\in\{\alpha,\beta,\gamma\}_n$ and $\mathrm{Ind}\{\alpha,\beta,\gamma\}_n=0$, 
then we denote $\delta=\{\alpha,\beta,\gamma\}_n$. 

We know the following \cite{MMO}: 
$
\pi^S_{24}= \{\bar{\mu}\sigma,\eta\eta^*\sigma\}\cong(\Z_2)^2.
$
The main purpose of this note is to give a proof of the fact that the Toda brackets \ $\langle\bar{\nu},\sigma,\bar{\nu}\rangle$ and $\langle\nu,\eta, \bar{\sigma}\rangle$ are not trivial. 
\begin{thm0} \label{thm}\label{main}
%$\{\varepsilon_9,\sigma_{17},\eta_{24}\sigma_{25}\}_4
%=\eta_9\sigma_{10}\eta^*_{17}+\{\sigma^2_9\eta_{23}\nu_{24}\}$,\ 
%$\{\bar{\nu}_{20},\sigma_{28}, \bar{\nu}_{35}\}
%=\{\nu_{20},\eta_{23}, \bar{\sigma}_{24}\}
%=\eta_{20}\sigma_{21}\eta^*_{28}=\eta_{20}\eta^*_{21}\sigma_{37}$ \ and \
$\{\bar{\nu}_{n},\sigma_{n+8},\bar{\nu}_{n+15}\}
=\{\nu_{n},\eta_{n+3},\bar{\sigma}_{n+4}\}
=\eta_n\sigma_{n+1}\eta^*_{n+8}=\eta_n\eta^*_{n+1}\sigma_{n+17}$
 for $n\ge 20$ and
$\langle\bar{\nu},\sigma, \bar{\nu}\rangle=\langle\nu,\eta, \bar{\sigma}\rangle=\eta\eta^*\sigma$. 
\end{thm0}

This result gives an affirmative answer to \cite[Conjecture 4.8]{Mu2}. 
In the proof of Theorem \ref{main}, our method is to inspect relations in homotopy groups of spheres through those in homotopy groups of rotation groups. 

We know the following \cite[Theorem 12.22]{T} and \cite[Theorem]{Od1}:
\[
\pi^{11}_{29}=\{\xi',\xi'+ \lambda',\eta_{11}\bar{\mu}_{12}\}\cong\Z_8{\oplus}\Z_4{\oplus}\Z_2
%\]
\hspace{3mm}\mbox{and}\hspace{3mm}
%\[
\pi^{13}_{31}=\{\xi_{13},\lambda,\eta_{13}\bar{\mu}_{14}\}
 \cong\Z_8{\oplus}\Z_8{\oplus}\Z_2.
\]
Let be $P: \pi^{2n+1}_{m+2}\to\pi^n_{m}$ the $P$-homomorphism\ ($P=\varDelta$ in \cite{T}) and
$H: \pi^n_{m}\to\pi^{2n-1}_{m}$ the Hopf homomorphism. 
We need 
\begin{lem}\label{HPxl13}
$H(P(\xi_{13}))\equiv\xi'\ \bmod\ 2\lambda',2\xi'$\quad and\quad  
$H(P(\lambda))\equiv\lambda' \bmod 2\lambda',2\xi'$. 
\end{lem}
Notice that Lemma \ref{HPxl13} improves \cite[(3.3)]{MMO}.

Oda \cite[Proposition 2.6 (5)]{Od} obtained the following relation in $\pi^7_{31}$: 
$$
\bar{\nu}_7\omega_{15}\equiv 0\ \bmod\ \nu_7\sigma_{10}\kappa_{17},  \bar{\zeta}'_7,
$$ 
where $\bar{\zeta}'_7=\sigma'\varepsilon_{14}\mu_{22}$  \cite[(5.10)]{MMO}.
The second purpose of this note is to show
\begin{thm0} \label{thm}\label{main2}
$\bar{\nu}_6\omega_{14}\equiv\nu_6\sigma_9\kappa_{16}+P(\xi_{13}+\lambda)\circ\eta_{29}\ \bmod\ 4\bar{\zeta}'_6$\quad 
and \quad $\bar{\nu}_7\omega_{15}=\nu_7\sigma_{10}\kappa_{17}$.
\end{thm0}

In Section 3, we give a relation between the Toda bracket and the $J$-homomorphism, 
and in Section 4, 
a relation between the Toda bracket and the generalized $P$-homomorphism.

\subsection*{Acknowledgement}
The authors wish to thank Professor Nobuyuki Oda for 
kind advices during the preparation of the manuscript.
The second author was supported by JSPS KAKENHI Grant Number JP22K03326.

\section{Recollection of some relations in homotopy groups of spheres}

In this note we use the following formulas without referring to them explicitly.
\begin{align*}
&H(\alpha\circ \Sigma\beta)=H(\alpha)\circ \Sigma\beta \text{ and } 
H(\Sigma\gamma\circ \alpha)=\Sigma(\gamma\wedge\gamma)\circ H(\alpha)\\
&\hspace{15mm}\text{for } 
 \alpha\in\pi_n(S^m),\ \beta\in\pi_q(S^{n-1})\text{ and }\gamma\in\pi_{m-1}(S^{r})
\text{\hspace{2mm}\cite[Proposition 2.2]{T}};\\
&P(\Sigma^{m+2}\alpha)=\pm[\iota_m,\iota_m]\circ \Sigma^m\alpha
=\pm[\iota_m, \Sigma\alpha] \text{\ \ for } \alpha\in\pi^{m-1}_n
\text{\hspace{2mm}\cite[Proposition 2.5]{T}};\\
&\alpha\wedge\beta=(-1)^{(p+k)h}\Sigma^{q}\alpha\circ \Sigma^{p+k}\beta
 =(-1)^{ph}\Sigma^p\beta\circ \Sigma^{q+h}\alpha\\
&\hspace{50mm}\text{for }\alpha\in\pi_{p+k}(S^{p})\text{ and }\beta\in\pi_{q+h}(S^{q})
\text{\hspace{2mm}\cite[Proposition 3.1]{T}};\\
&[\alpha\circ\Sigma\gamma, \beta\circ\Sigma\delta]
=[\alpha,\beta]\circ\Sigma(\gamma\wedge\delta)\\
&\hspace{15mm}\text{for } \alpha\in\pi_{p+1}(S^k),\ 
\beta\in\pi_{q+1}(S^\ell),\ \gamma\in\pi_m(S^p)\text{ and }\delta\in_n(S^q)
\text{\hspace{2mm}\cite[X-(8.18)]{WG78}};\\
&\Sigma[\alpha,\beta]=0
 \text{\ \ for } \alpha\in\pi_p(S^m)\text{ and }\beta\in\pi_q(S^{n})
 \text{\hspace{2mm}\cite[X-(8.20)]{WG78}}.
\end{align*}

We use the result,
the notation of \cite{T} and the properties of Toda brackets freely. 
We denote by $\sharp\alpha$ the order of an element $\alpha$. 
We recall from \cite{T} relations:
\begin{gather}
\label{2eta}% 1
\sharp\eta_n = 2 \text{\ \ for\ \ } n\ge 3
\quad\text{\cite[Proposition 5.1]{T}};
\\
\label{2nu5}% 3
2\nu_5=\Sigma^2 \nu'  
\quad\text{\cite[(5.5)]{T}};
%\\
%\label{4nu5}% 3
%4\nu_5=\eta^3_5  \quad\text{\cite[(5.5)]{T}};
\\
\label{n2}% 4
\eta_4\nu_5=(\Sigma \nu')\eta_7=[\iota_4,\eta_4]  
\quad\text{\cite[(5.9), (5.11)]{T}};
\\
\label{n3}% 4
\nu_5\eta_8=[\iota_5,\iota_5] \quad\text{\cite[(5.10)]{T}};
\\
\label{2nu5nu}
\sharp\nu^2_5=2  \text{\ \ for\ \ } n\ge 5
\quad\text{\cite[Proposition 5.11]{T}};
\\
\label{n4}% 7
2\sigma_9=\Sigma^2 \sigma'
\quad\text{\cite[Lemma 5.14]{T}};
%\\
%\label{et7sg}% 8
%\eta_7\sigma_8=\bar{\nu}_7+\varepsilon_7+\sigma'\eta_{14}
%\quad\text{\cite[(7.4)]{T}};
\\
\label{et9sg}% 8
\eta_9\sigma_{10}=\bar{\nu}_9+\varepsilon_9
 \quad\text{\cite[Lemma 6.4]{T}};
\\
%%%%%%%%%%%%%%%%%%%%% 不要？
%\notag % 8
%\eta_9\sigma_{10}+\sigma_9\eta_{16}=[\iota_9,\iota_9]
% \quad\text{\cite[(7.1)]{T}};
%\\
\label{etbn}% 9
\eta_5\bar{\nu}_6=\nu^3_5 \text{\ \ and\ \ } 
\bar{\nu}_6\eta_{14}=\nu^3_6 \quad\text{\cite[Lemma 6.3]{T}};
\\
%\label{sg'nu}% 10
%\sigma'\nu_{14}=x\nu_7\sigma_{10}\ \text{for some odd\ } x
% \quad\text{\cite[(7.19)]{T}};
%\\
%\label{nu11sg}% 10
%\nu_{11}\sigma_{14}=0
% \quad\text{\cite[(7.20)]{T}};
%\\
\label{sgm11n}% 10
\sigma_{11}\nu_{18}=[\iota_{11},\iota_{11}]
\mand 2\sigma_{11}\nu_{18}=0
 \quad\text{\cite[(7.20), (7.21)]{T}};
\\
%\label{nu6e}% 11
%\nu_6\varepsilon_9=2\bar{\nu}_6\nu_{14}=[\iota_6,\nu^2_6]
% \quad\text{\cite[(7.18)]{T}};
%\\
%\label{bnun}% 11 
%\bar{\nu}_9\nu_{17}=[\iota_9,\nu_9]
% \quad\text{\cite[(7.22)]{T}};
%\\
%\label{4ze5}% 11
%4\zeta_5=\eta^2_5\mu_7
% \quad\text{\cite[(7,14)]{T}};
%\\
%\label{zetaeta}% 11 
%\zeta_7\eta_{18}=\eta_7\zeta_8\in\pi^7_{19}=0
% \quad\text{\cite[Theorem 7.6]{T}};
%\\
\label{bnu6sg}% 15
\varepsilon_3\sigma_{11}=0 \text{\ \ and\ \ }
\bar{\nu}_6\sigma_{14}=0
 \quad\text{\cite[Lemma 10.7]{T}};
\\
\label{bep}% 15
\eta_6\kappa_7=\bar{\varepsilon}_6
\mand
\kappa_9\eta_{23}=\bar{\varepsilon}_9
\quad\text{\cite[(10.23)]{T}};
\\
\label{epep}% 16
\nu_5\sigma_8\nu^2_{15}=\eta_5\bar{\varepsilon}_6
\mand
\varepsilon^2_3
 =\varepsilon_3\bar{\nu}_{11}=\eta_3\bar{\varepsilon}_4=\bar{\varepsilon}_3\eta_{18}
\quad\text{\cite[Lemma 12.10]{T}};
\\
\label{xi12}% 16
\xi_{12}\in\{\sigma_{12},\nu_{19},\sigma_{22}\}_1
\quad\text{\cite[p.~153]{T}};
\\
\label{eta*16}% 16
\eta^*_{16}\in\{\sigma_{16},2\sigma_{23},\eta_{30}\}_1
\quad\text{\cite[p.~153]{T}};
\\
\label{Homega}% 16
H(\omega_{14})=\nu_{27} 
\quad\text{\cite[Lemma 12.15 i)]{T}};
\\
%%%%%%%%%%%%%%%%%%%%% 不要？
%\notag% 16
%[\iota_{17},\iota_{17}]\equiv\eta^*_{17}+\omega_{17}\ \bmod\ \sigma_{17}\mu_{24}
% \quad\text{\cite[Proposition 12.20.ii)]{T}};
%\\
%\label{n*19x}% 18
%[\iota_{19},\iota_{19}]=\nu^*_{19}+\xi_{19}
%\quad\text{\cite[Corollary 12.25]{T}};
%\\
\label{n7}% 18
2\sigma_{11}\zeta_{18}=0
\mand
\sigma_{13}\zeta_{20}=\sigma_{8}\wedge\zeta_5=\zeta_{13}\sigma_{24}=0
\quad\text{\cite[(12.23)]{T}};
\\
\label{2lambda}% 18
\Sigma\lambda''=2\lambda',\ 
\Sigma^2\lambda'=2\lambda \mand 
\Sigma^4\lambda=2\nu^*_{17}
\quad\text{\cite[Lemmas 12.18, 12.19]{T}}; 
\\
\label{2xi}% 18
\Sigma\xi''=2\xi' \mand
\Sigma^2\xi'=2\xi_{13} 
\quad\text{\cite[Lemma 12.19]{T}};
\\
\label{Hxi}
H(\xi')= \bar{\nu}_{21}+\varepsilon_{21}=\eta_{21}\sigma_{22}
\quad\text{\cite[Lemmas 6.4, 12.19]{T}};
\\
\label{4bze5}
4\bar{\zeta}_5=\eta^2_5\bar{\mu}_7
\quad\text{\cite[Lemma 12.4]{T}};
\\
\label{x13et}% 19
[\iota_{13},\sigma_{13}]=(\Sigma\theta)\sigma_{25}
=\xi_{13}\eta_{31}
\quad\text{\cite[p.~166]{T}}.
\end{gather}
%%%%%%%%%%%%%%%%%%%%%%%%%%%%%%%%%%%%%%%%%%%%%%%%%%%%%%%%%%

We also recall from \cite{Og, MMO, Od, Od1, IMM} relations:
%%%%%%%%%%%%%%%%%%%%%%%%%%%%%%%%%%%%%%%%%%%%%%%%%%%%%%%%%%
\begin{gather}
\label{nu10sg}% 10
\nu_{10}\sigma_{13}=2\sigma_{10}\nu_{17}=[\iota_{10},\eta_{10}]
 \quad\text{\cite[(2.18)]{IMM}};
\\
%\label{et4zt}% 12
%\eta_4\zeta_5\equiv\ (\Sigma \nu')\mu_7\ \bmod\ (\Sigma \nu')\eta_7\varepsilon_8
% \quad\text{\cite[Proposition 2.2 (5)]{Og}};
%\\
%\label{nu9m}% 12
%\zeta_6\eta_{17}=\nu_6\mu_9=8([\iota_6,\iota_6]\sigma_{11})
% \quad\text{\cite[(2.20)]{IMM}};
%\\
\label{et9sg2}% 15
\eta_9\sigma^2_{10}=0
\quad\text{\cite[(2.15)]{IMM}};
%\\
%\label{bn^2_6}% 15
%\bar{\nu}^2_6=0
%\quad\text{\cite[Proposition 2.8 (2)]{Og}};
\\
\label{n7k}% 17
\nu_7\kappa_{10}=\kappa_7\nu_{21}
 \quad\text{\cite[Proposition 2.13 (2)]{Og}};
\\
%%%%%%%%%%%%%%%%%%%%% 不要？
%\notag% 17
%\mu_3\varepsilon_{12}\equiv\eta_3\mu_4\sigma_{13}\ \bmod\ 2\bar{\varepsilon}'
%\quad\text{\cite[Proposition 2.13 (7)]{Og}};
%\\
\label{mepbn2}% 17
\mu_5\bar{\nu}_{14}=0 
\mand 
\bar{\nu}_6\mu_{14}=0
\quad\text{\cite[Proposition 2.13 (8)]{Og}};
\\
%\label{2xi"}% 18
%2\xi''\equiv \sigma_{10}\zeta_{17}\ \bmod\ 2\sigma_{10}\zeta_{17}
%\quad\text{\cite[Proposition 1 (4)]{Od1}};
%\\
\label{ze5bnu}% 19
\zeta_5\bar{\nu}_{16}=0
\mand
\zeta_{7}\varepsilon_{18}=0 
\quad\text{\cite[II-Proposition 2.2 (3)]{Od}};
\\
\label{Hlam}% 19
H(\lambda')=\varepsilon_{21}
\quad\text{\cite[Proposition 4 (3)]{Od1}};
\\
\label{n9xi}% 21
\nu_9\xi_{12}=\sigma^3_9 \text{\ \ and\ \ }  \xi_{12}\nu_{30}=\sigma^3_{12}
\quad\text{\cite[II-Proposition 2.1 (2)]{Od}};
\\
\label{k7s}% 21
\kappa_7\sigma_{21}=0
\quad\text{\cite[II-Proposition 2.1 (2)]{Od}};
\\
\label{nulm}% 21
\nu_{10}\lambda=\sigma_{10}\kappa_{17}
\quad\text{\cite[I-Proposition 3.1 (1)]{Od}};
\\
\label{mu3s^2}% 23
\mu_3\sigma^2_{12}=0
\quad\text{\cite[(2.9)]{MMO}};
\\
\label{bnu9omg}% 24
\bar{\nu}_9\omega_{17}=0
\quad\text{\cite[III-Proposition 2.6 (5)]{Od}};
\\
\label{lmdsg}% 26
\lambda\sigma_{31}=0
\quad\text{\cite[III-Proposition 2.2 (2)]{Od}}.
%\\
%\label{n8}% 26
%[\iota_{12},\bar{\varepsilon}_{12}]=(\Sigma\theta')\kappa_{24}
%\quad\text{\cite[I-(2.1), Proposition 6.4 (9)]{Od}}.
\end{gather}

%%%%%%%%%%%%%%%%%%%%% 不要？
%By relations 
%$\nu_7\zeta_{10}=(\Sigma^2\sigma''')\sigma_{14}$ \cite[Lemma 9.2]{T},
%$\Sigma^2\sigma'''=4\sigma'$ \cite[Lemma 5.4]{T}
%and $8\sigma^2_{14}=0$ \cite[Theorem 10.4]{T}, we have 
%\begin{equation}\label{nuzt}% 14
%\nu_7\zeta_{10}=4\sigma'\sigma_{14}\ \mbox{and} \ \nu_{14}\zeta_{17}=0.
%\end{equation}
%%%%%%%%%%%%%%%%%%%%%

%%%%%%%%%%%%% 12
By equations $\mu_5\nu_{14}=(\Sigma^2\nu')\eta_{8}\varepsilon_{9}$
\cite[Proposition 2.2 (4)]{Og}
and $\Sigma^2\nu'=2\nu_5$ \eqref{2nu5}, we have
\begin{equation}
\label{mu7nu}% 16 9-25
\mu_5\nu_{14}=0.
\end{equation}
%%%%%%%%%%%%%%%%%%%%%

%%%%%%%%%%%%%% 14
%By the equation $\Sigma \theta =P(\iota_{27})=[\iota_{13},\iota_{13}]$ \cite[(7.30)]{T},
%we have
%\[
%\begin{split}
%\eta_{12}(\Sigma \theta)\eta_{25}
%&=\eta_{12}\circ[\iota_{13},\iota_{13}]\circ\eta_{25}
%=[\eta_{12},\eta_{12}]\circ\eta_{25}\\
%&=[\iota_{12},\iota_{12}]\circ\Sigma(\eta_{11}\wedge\eta_{11})\circ\eta_{25}
%=[\iota_{12},\iota_{12}]\circ\eta^3_{23}
%=P(\eta^3_{25}).
%\end{split}
%\]
%So, by the equation $P(\eta^3_{25})=0$ \cite[(10.10)]{T},
%we obtain
%\begin{equation}\label{et12Ethet}% 14 12-26
%\eta_{12}(\Sigma \theta)\eta_{25}=0.
%\end{equation}
%%%%%%%%%%%%%%%%%%%%%%

%%%%%%%%%%%%% 16
By relations $\eta_9\bar{\varepsilon}_{10}=\nu_{9}\sigma_{12}\nu^2_{19}$ \eqref{epep},
and $\sigma_{12}\nu_{19}=0$  \eqref{sgm11n}, we have
\begin{equation}
\label{et9be}% 16 9-25
\eta_9\bar{\varepsilon}_{10}=0.
\end{equation}
By the fact $\Ker\{\Sigma:\pi^{14}_{30}\to\pi^{15}_{31}\}=\{2\omega_{14}\}$
\cite[Theorem 12.16]{T} and
the relation $2\iota_{15}\circ\Sigma\omega_{14}=2\Sigma\omega_{14}$, 
we have 
\[
2\iota_{14}\circ\omega_{14}\equiv 2\omega_{14}\ \bmod\ 2\omega_{14}.
\]
On the other hand, by the relation $H(\omega_{14})=\nu_{27}$ \eqref{Homega},
we have
\[
H(2\iota_{14}\circ\omega_{14})=\Sigma(2\iota_{13}\wedge 2\iota_{13})\circ H(\omega_{14})
=4\nu_{27}
\mand
H(2\omega_{14})=2\nu_{27}.
\]
Hence, we obtain
\begin{equation}\label{4ome14}% 16 14-30
2\iota_{14}\circ\omega_{14} =4\omega_{14}.
\end{equation}
%%%%%%%%%%%%%

%%%%%%%%%%%%% 18
By the relations $\bar{\varepsilon}_6=\eta_6\kappa_7$ \eqref{bep}, 
$\kappa_7\nu_{21}=\nu_7\kappa_{10}$ \eqref{n7k} 
and $\eta_6\nu_7=0$ \eqref{n2}, we have
\begin{equation}\label{be6n}% 18 6-24
\bar{\varepsilon}_6\nu_{21}=\eta_6\kappa_{7}\nu_{21}=\eta_6\nu_7\kappa_{10}=0.
\end{equation}
%%%%%%%%%%%%%

%%%%%%%%%%%%%%%%%%%%% 不要？
%By  (\ref{n7}), \cite[Theorem 12.8, (12.25), p. 166]{T} and using the EHP-sequence, we obtain 
%%We recall \cite[Lemmas 12.14, 12.19, Theorem 12.22]{T}
%\begin{equation}\label{sgm12zt}
%4\sigma_9\zeta_{16}=[\iota_9,\eta_9\mu_{10}],\ 2\sigma_{10}\zeta_{17}=[\iota_{10},\mu_{10}] \mand\sigma_{12}\zeta_{19}=8[\iota_{12},\sigma_{12}].
%\end{equation}
%
%$4\sigma_9\zeta_{16}=[\iota_9,\eta_9\mu_{10}]$ \cite[(6.7)]{GM}
%
%$2\sigma_{10}\zeta_{17}=[\iota_{10},\mu_{10}]$ \cite[(12.25)]{T}
%
%$\sigma_{12}\zeta_{19}=8[\iota_{12},\sigma_{12}]$ \cite[p.~166]{T}
%%%%%%%%%%%%%%%%%%%%%

%%%%%%%%%%%%%%%%%%%%% 18
By relations $H(\rho_{13})=\eta^3_{25}$ \cite[(10.22)]{T}
and $\nu_{22}\eta_{25}=0$ \eqref{n3},
we have
\[
H(\nu_{10}\rho_{13})=\Sigma(\nu_{9}\wedge\nu_{9})H(\rho_{13})
=\nu^2_{19}\eta^3_{25}=0.
\]
So, by the EHP-sequence, we have
$\nu_{10}\rho_{13}\in\Sigma\pi^9_{27}$.
Moreover, by the fact that
$\Sigma\pi^9_{27}
=\{\sigma_{10}\zeta_{17},\eta_{10}\bar{\mu}_{11}\}
\cong\Z_4\oplus\Z_2$ \cite[p. 164]{T}
and relations 
$\nu\rho=0$ \cite[Lemma 12.24]{T}
and $\eta\bar{\mu}\neq 0$ \cite[Theorem 12.22]{T},
we have $\nu_{10}\rho_{13}\in\{\sigma_{10}\zeta_{17}\}\cong\Z_4$.
Furthermore, by the relation
\[
2(\nu_{10}\rho_{13})=\nu_{10}(2\rho_{13})
=\nu_{10}\Sigma^4\rho'=\Sigma^4(\nu_6\rho')=0
\]
from $2\rho_{13}=\Sigma^4\rho'$ \cite[Lemma 10.9]{T}
and $\nu_6\rho'=0$ \cite[Lemma 6]{Os82}, we have
$\nu_{10}\rho_{13}\in\{2\sigma_{10}\zeta_{17}\}$.
In addition, the equation 
$2\sigma_{11}\zeta_{18}=0$ \eqref{n7} implies 
$\nu_{11}\rho_{14}=0$. Thus, we obtain
\begin{equation}\label{nu11ro}% 18 11-29
\nu_{10}\rho_{13}\in\{2\sigma_{10}\zeta_{17}\}\cong\Z_2
\mand
\nu_{11}\rho_{14}=0. 
\end{equation}
%%%%%%%%%%%%%

%%%%%%%%%%%%% 19
By relations
$\mu^2_3\equiv \eta_3\bar{\mu}_4\bmod 2\mu'\sigma_{14}$
\cite[Proposition 2.17 (9)]{Og},
$2\eta_3=0$ \eqref{2eta} and 
$\eta^2_5\bar{\mu}_7=4\bar{\zeta}_5$ \eqref{4bze5}, we have
\begin{equation}\label{etmumu}% 19 3-22
\eta_3\mu^2_4 =\eta^2_3\bar{\mu}_5
\mand \eta_5\mu^2_6=4\bar{\zeta}_5.
\end{equation}
%%%%%%%%%%%%%

%%%%%%%%%%%%% 19
By the relations 
$\nu_5\sigma_8\nu^2_{15}=\eta_5\bar{\varepsilon}_6$ \eqref{epep} and 
$\bar{\varepsilon}_6\nu_{21}=0$ \eqref{be6n}, we have
$\nu_5\sigma_8\nu^3_{15}=\eta_5\bar{\varepsilon}_6\nu_{21}=0$.
By $4\zeta_5=\eta^2_5\mu_7$ \cite[(7,14)]{T} and 
$2\varepsilon_{16}=0$ \cite[Theorem 7.1]{T}, we have
$\eta^2_5\mu_7\varepsilon_{16}=0$.
Moreover,
by $\nu_5(\Sigma\sigma')=2\nu_5\sigma_8$ \cite[(7.16)]{T}
and $2\mu_{15}=0$ \cite[Lemma 6.5]{T},
we have
$\nu_5(\Sigma\sigma')\mu_{15}=\nu_5(\Sigma\sigma')\eta_{15}\varepsilon_{16}=0$.
Hence, we obtain
\begin{equation}\label{n5sn^3}% 19 5-24
\nu_5\sigma_8\nu^3_{15}=\eta_5\bar{\varepsilon}_6\nu_{21}
=\eta^2_5\mu_7\varepsilon_{16}
=\nu_5(\Sigma\sigma')\mu_{15}=\nu_5(\Sigma\sigma')\eta_{15}\varepsilon_{16}=0.
\end{equation}
%%%%%%%%%%%%%

%%%%%%%%%%%%%%%%%%%%% 不要？
%By the relations 
%$\bar{\zeta}_5\eta_{24}\equiv \zeta_5\mu_{16}\equiv \nu_5\bar{\mu}_8\
%	\bmod\ \nu_5\eta_8\mu_9\sigma_{18}$
%\cite[II-Proposition 2.2(1);(2)]{Od},
%$\nu_6\bar{\mu}_9=16P(\rho_{13})$
%\cite[(16.6)]{MT} and $\nu_6\eta_9=0$ \eqref{n3}, we have
%\begin{equation}\label{nu6bm}% 19 6-25
%\bar{\zeta}_6\eta_{25}=\zeta_6\mu_{17}=\nu_6\bar{\mu}_9=16P(\rho_{13}).
%\end{equation}
%%%%%%%%%%%%%

%%%%%%%%%%%%% 19
By the relations $\zeta_7\bar{\nu}_{18}=\zeta_7\varepsilon_{18}=0$
\eqref{ze5bnu} and
$\bar{\nu}_{18}+\varepsilon_{18}=\eta_{18}\sigma_{19}$
\eqref{et9sg}, we have
\begin{equation}\label{ze7etsg}% 19 7-26
\zeta_7\eta_{18}\sigma_{19}=0.
\end{equation}
%%%%%%%%%%%%%

%%%%%%%%%%%%% 19
%By the relation $\bar{\nu}\zeta=0$ \cite[Theorem 14.1 iv)]{T}
%and the fact 
%$\pi^{7}_{26}\cong\pi^S_{19}$ \cite[Theorems 12.9, 12.23]{T},
%we have
%\begin{equation}\label{bn7z}% 19 7-26
%\bar{\nu}_7\zeta_{15}=0.
%\end{equation}
%%%%%%%%%%%%%

%%%%%%%%%%%%% 19
By relations %$\Sigma\lambda''=2\lambda'$,
%$\Sigma^2\lambda'=2\lambda$,
%$\Sigma^4\lambda=2\nu^*_{17}$
\eqref{2lambda}, 
%$\Sigma\xi''=2\xi'$, 
%$\Sigma^2\xi'=2\xi_{13}$
\eqref{2xi} and $2\eta_{3}=0$ \eqref{2eta},
we obtain
\begin{equation}\label{Elameta}% 19 11-30
\Sigma(\lambda''\eta_{28})=\Sigma(\xi''\eta_{28})=0,\ 
\Sigma^2(\lambda'\eta_{29})=\Sigma^2(\xi'\eta_{29})=0 
\mand
\Sigma^4(\lambda\eta_{31})=0.
\end{equation}

%%%%%%%%%%%%% 19
By the relation$H(\omega_{14})=\nu_{27}$ \eqref{Homega},
we have
\[
H(\nu_{11}\omega_{14})=\Sigma(\nu_{10}\wedge\nu_{10})\circ H(\omega_{14})
=\nu^2_{21}\circ\nu_{27}=\nu^3_{21}.
\]
By 
$H(\lambda')=\varepsilon_{21}$ \eqref{Hlam},
$H(\xi')=\varepsilon_{21}+\bar{\nu}_{21}$ \eqref{Hxi}, 
$2\eta_{29}=0$ \eqref{2eta}
and $\bar{\nu}_{21}\eta_{29}=\nu^3_{21}$ \eqref{etbn},
we have
\[
H((\lambda'+\xi')\eta_{29})
=H(\lambda')\eta_{29}+H(\xi')\eta_{29}
=2\varepsilon_{21}\eta_{29}+\bar{\nu}_{21}\eta_{29}
=\nu^3_{21}.
\]
So, by using the EHP-sequence and \cite[Theorem 12.23, (12.22)]{T}, we have
\[
\nu_{11}\omega_{14}\equiv (\lambda'+\xi')\eta_{29}
\bmod\ \Sigma\pi^{10}_{29}=\{\bar{\sigma}_{11},\bar{\zeta}_{11}\}\cong\Z_2\oplus\Z_8,
\]
and $\pi^S_{19}=\{\bar{\sigma},\bar{\zeta}\}\cong\Z_2\oplus\Z_8$.
Since $\nu\omega=0$ by $P(\nu_{35})=\omega_{17}\nu_{33}$ \cite[p.~170]{T}
and $\Sigma^2((\lambda'+\xi')\eta_{29})=0$
by \eqref{Elameta},
we obtain 
\begin{equation}\label{n11om}% 19 11-30, 13-32
\nu_{11}\omega_{14}=(\lambda'+\xi')\eta_{29}\text{\ \ and\ \ } 
\nu_{13}\omega_{16}=0.
\end{equation}
%%%%%%%%%%%%% 

%%%%%%%%%%%%% 19
By relations 
$(2\nu_{11})\omega_{14}
=\nu_{11}\circ 2\iota_{14}\circ \omega_{14}
=4\nu_{11}\omega_{14}$ \eqref{4ome14},
\eqref{n11om}
and $2\eta_{29}=0$ \eqref{2eta}, we have
\begin{equation}\label{2nu11ome}% 19 11-30
(2\nu_{11})\omega_{14}=0.
\end{equation}
%%%%%%%%%%%%% 

%%%%%%%%%%%%% 19
We recall from \cite[I-Proposition 3.4 (3)]{Od} the relation:
\[
\{\nu_{11},\sigma_{14},\bar{\nu}_{21}\}
\ni \bar{\sigma}_{11}+\alpha
\text{\ \ for some }
\alpha\in\{\xi'\eta_{29}, \lambda'\eta_{29}\}.
\]
By \eqref{Elameta}, we have 
$\bar{\sigma}_{13}\in\{\nu_{13},\sigma_{16},\bar{\nu}_{23}\}$.
For $n\ge 15$, by \cite[Theorems 7.4, 12.16]{T}, we have
\[
\mathrm{Ind}\{\nu_{n},\sigma_{n+3},\bar{\nu}_{n+10}\}
=\nu_{n}\circ\pi^{n+3}_{n+19}+\pi^{n}_{n+11}\circ\bar{\nu}_{n+11}
=\{\nu_{n}\omega_{n+3}, \nu_{n}\sigma_{n+3}\mu_{n+10},
\zeta_{n}\bar{\nu}_{n+11}\}.
\]
So, by relations $\nu_{n}\omega_{n+3}=0$ \eqref{n11om},
$\nu_{n}\sigma_{n+3}=0$ \eqref{nu10sg}
and
$\zeta_{n}\bar{\nu}_{n+11}=0$ \eqref{ze5bnu},
we have
\begin{equation}\label{bs15}% 19 13-32
\bar{\sigma}_{n}=\{\nu_{n},\sigma_{n+3},\bar{\nu}_{n+10}\}
\text{\ \ for\ } n\geq 15.
\end{equation}
%%%%%%%%%%%%%%%%%%%%%

%We recall the element $\omega'\in\pi^{12}_{31}$ \cite[Lemma 12.21, (12.27), p.~166]{T}: 
%\begin{equation}\label{e2omg}
%\Sigma^2\omega'=2\omega_{14}\nu_{30}=[\iota_{14},\nu^2_{14}] \text{\ and\ }
%H(\omega')\equiv\varepsilon_{23}\ 
%\bmod\ \varepsilon_{23}+\bar{\nu}_{23}.
%\end{equation}
%%%%%%%%%%%%%%%%%%%%% 不要？
%By relations 
%$\nu_{13}\eta^*_{16}\equiv E\omega'\bmod\ \xi_{13}\eta_{31}$
%\cite[Proposition 2.20(8)]{Og},
%$\Sigma^2\omega'=[\iota_{14},\nu^2_{14}]$ \eqref{e2omg} and 
%$\xi_{14}\eta_{32}=0$ \eqref{x13et}, 
%we have 
%\begin{equation}
%\label{n14et*}% 19
%\nu_{14}\eta^*_{17}=[\iota_{14},\nu^2_{14}].
%\end{equation}
%%%%%%%%%%%%%%%%%%%%%

%%%%%%%%%%%%%%%%%%%%%%%%%%%%% 21
%By $\nu_6\zeta_{9}=\zeta_6\nu_{17}$ \cite[Proposition 2.4(2)]{Og} and
%$\nu_{17}\sigma_{20}=0$ \eqref{nu10sg},
%\begin{equation}\label{nu6zs}% 21 6-27
%\nu_6\zeta_9\sigma_{20}=0.
%\end{equation}
%%%%%%%%%%%%%%%%%%%%%%%

%%%%%%%%%%%%%%%%%%%%%%%%%% 23
By relations 
$\bar{\varepsilon}_3\eta_{18}=\varepsilon_3\bar{\nu}_{11}$ \eqref{epep}
and $\bar{\nu}_{11}\sigma_{19}=0$ \eqref{bnu6sg}, 
we have
\begin{equation}\label{bep3etas}% 23 3-26
\bar{\varepsilon}_3\eta_{18}\sigma_{19}=0.
\end{equation}
%%%%%%%%%%%%%%%

%%%%%%%%%%%%%%%%%%%%%%%%%% 23
%By relations
%$\bar{\varepsilon}_{14}=\eta_{14}\kappa_{15}$ \eqref{bep},
%$\bar{\nu}_6\eta_{14}=\nu^3_6$ \eqref{etbn},
%$\nu^2_{9}\kappa_{15}=4\bar{\kappa}_{9}$ \cite[Theorem 15.4]{MT},
%we have
%\begin{equation}\label{bn8bep}% 23  6-29
%\bar{\nu}_6\bar{\varepsilon}_{14}
%%=\bar{\nu}_6\eta_{14}\kappa_{15}
%%=\nu^3_6\kappa_{15}=\nu_6\nu^2_9\kappa_{15}
%%=\nu_6(4\bar{\kappa}_{9})
%=4\nu_6\bar{\kappa}_{9}.
%\end{equation}
%%%%%%%%%%%%%%%

%%%%%%%%%%%%%%%%%%%%%%%%%% 23
By relations
$\rho_{13}\eta_{28}=\mu_{13}\sigma_{22}$ \cite[Propsition 12.20 i)]{T}
and  $\mu_{13}\sigma^2_{22}=0$ \eqref{mu3s^2},
we have
\begin{equation}\label{r13etas}% 23  13-36
\rho_{13}\eta_{28}\sigma_{29}=0.
\end{equation}
%%%%%%%%%%%%%%%

%%%%%%%%%%%%%%%%%%%%%%%%% 24
By relations
$\Sigma\bar{\varepsilon}'=(\Sigma\nu')\kappa_7$ \cite[Lemma 12.3]{T}
and
$\kappa_7\sigma_{21}=0$ \eqref{k7s},
we have
$\Sigma(\bar{\varepsilon}'\sigma_{20})=(\Sigma\nu')\kappa_7\sigma_{21}=0$.
Since $\Sigma:\pi^3_{26}\to\pi^4_{28}$ is a monomorphism
\cite[Theorem 1.1 (b)]{MMO},
we have
\begin{equation}\label{bep's}% 24  3-27
\bar{\varepsilon}'\sigma_{20}=0.
\end{equation}
%%%%%%%%%%%%%%%%%%%%

%%%%%%%%%%%%%%%%%%%% 24
By relations 
$\zeta'\varepsilon_{22}=\zeta'\bar{\nu}_{22}=0$
\cite[I-Proposition 3.1 (2)]{Od} and 
$\eta_{22}\sigma_{23}=\bar{\nu}_{22}+\varepsilon_{22}$ \eqref{et9sg}, 
we have
\begin{equation}\label{z'etas}% 24 6-30 
\zeta'\eta_{22}\sigma_{23}=0.
\end{equation}
%%%%%%%%%%%%%%%%%%%%%

%%%%%%%%%%%%%%%%%%%% 24
By relations 
$\eta_{14}\sigma_{15}=\bar{\nu}_{14}+\varepsilon_{14}$ \eqref{et9sg},
$\bar{\nu}_{14}\mu_{22}=0$ \eqref{mepbn2} 
and
$\sigma'\varepsilon_{14}\mu_{22}=\bar{\zeta}'_7$ \cite[(5.10)]{MMO}, 
we have
$\sigma'\eta_{14}\mu_{15}\sigma_{24}
=\sigma'\eta_{14}\sigma_{15}\mu_{22}
=\sigma'(\bar{\nu}_{14}+\varepsilon_{14})\mu_{22}
=\bar{\zeta}'_7$ and
\begin{equation}\label{bz'7}% 24 7-31
\sigma'\eta_{14}\mu_{15}\sigma_{24}=\bar{\zeta}'_7.
\end{equation}
%%%%%%%%%%%%%%%%%%%%%

%%%%%%%%%%%%%%%%%%%% 24
By relations 
$\sigma'\nu_{14}=x\nu_7\sigma_{10}$ for some odd $x$
\cite[(7.19)]{T}, $2\kappa_{19}=0$ \cite[Theorem 10.3]{T}
and $\Sigma^2\sigma'=2\sigma_9$ \eqref{n4}, we have
$\nu_9\sigma_{12}\kappa_{19}
=(\Sigma^2\sigma')\nu_{16}\kappa_{19}=2\sigma_9\nu_{16}\kappa_{19}=0$ and
\begin{equation}\label{nu9sgka}% 24 9-33
\nu_9\sigma_{12}\kappa_{19}=0.
\end{equation}
%%%%%%%%%%%%%%%%%%%%%

%%%%%%%%%%%%%%%%%%%% 24
By relations 
$\phi_9\equiv \sigma_9\eta^*_{16}
\bmod\ \sigma^2_9\mu_{23},\ 4\nu_{9}\bar{\kappa}_{12}$
\cite[I-Propositions 3.4 (7)]{Od},
$\eta_9\sigma^2_{10}\mu_{24}=0$ \eqref{et9sg2}
and $\eta_9\nu_{10}=0$ \eqref{n2},
we have
\begin{equation}\label{eta9phi}% 24 9-33
\eta_9\phi_{10}=\eta_9\sigma_{10}\eta^*_{17}.
\end{equation}
%%%%%%%%%%%%%%%%%%%%%

%%%%%%%%%%%%%%%%%%%%% 24
By $\nu_{10}\lambda=\sigma_{10}\kappa_{17}$ \eqref{nulm} ,
$\kappa_{17}\nu_{31}=\nu_{17}\kappa_{20}$ \eqref{n7k}
and
$\sigma_{11}\nu_{18}=[\iota_{11},\iota_{11}]$ \eqref{sgm11n}, 
we obtain 
\begin{equation}\label{nulmn}% 24  10-34
\nu_{10}\lambda\nu_{31}=\sigma_{10}\nu_{17}\kappa_{20}
\mand
\Sigma(\nu_{10}\lambda\nu_{31})=[\iota_{11},\kappa_{11}].
\end{equation}
%%%%%%%%%%%%%%%%%%%%%

%%%%%%%%%%%%%%%%%%%%%%% 25
By relations 
$2\xi_{12}\sigma_{30}=\sigma_{12}\xi_{19}$ 
\cite[I-Proposition 3.5 (3)]{Od},
$2(\sigma_{12}\xi_{19}+\sigma_{12}\nu^*_{19})=0$
and
$8\xi_{12}\sigma_{30}=0$ 
\cite[Theorem 1 (a)]{Od}, we have
\begin{equation}\label{4xi12sg}% 25  12-37
4\xi_{12}\sigma_{30}=2\sigma_{12}\xi_{19}=2\sigma_{12}\nu^*_{19}.
\end{equation}
%%%%%%%%%%%%%%%%%%%%%

%%%%%%%%%%%%%%%%%%%%%%% 26 不要？
%We recall from \cite[p.~75]{Od} the relation:
%\[
%2\tau^{IV}\equiv \Sigma\tau'''+zP(\rho_{25})
%\bmod\ \sigma_{12}\bar{\sigma}_{19},\ 4\Sigma\tau'''
%\text{\ \ for some integer $z$}.
%\]
%By relations $\Sigma P(\rho_{25})=0$,
%$2\sigma_{13}\bar{\sigma}_{20}=0$
%and $8\tau'''=0$ \cite[I-Theorem 1(b)]{Od},
%we have
%\begin{equation}\label{2E2tau'''}% 26 13-39
%2\Sigma^2\tau'''=4\Sigma\tau^{IV}.
%\end{equation}
%%%%%%%%%%%%%%%%%%%%%%%%%%%

We need the following lemmas.
%%%%%%%%%%%%% 12
\begin{lem}\label{etnusg}% 12 10-22
\begin{enumerate}
\item
$\{\eta_{10},\nu_{11},\sigma_{14}\}=2[\iota_{10},\nu_{10}]$. 
\item
$\{\eta_n,\nu_{n+1},\sigma_{n+4}\}=0$ for $n\ge 11$. 
\end{enumerate}
\end{lem}
\begin{proof}
By relations $\eta_5\nu_6=0$ \eqref{n2} and 
$\nu_{11}\sigma_{12}=0$ \eqref{nu10sg},
the Toda bracket 
$\{\eta_n,\nu_{n+1},\sigma_{n+4}\}$ 
is well-defined for $n\ge 10$.
From the facts $\pi^{12}_{23}=\{[\iota_{12},\iota_{12}],\zeta_{12}\}$, $\pi^{k}_{k+11}=\{\zeta_{k}\}$ ($k=10$, $11$ and $k\ge 13$) and $\pi^\ell_{\ell+5}=0$ ($\ell\ge 7$) \cite[Proposition 5.9, Theorem 7.4]{T} and
relations $\eta_5\zeta_6=0$ \cite[Proposition 2.2 (5)]{Og} and
\[
\begin{split}
\eta_{11}\circ[\iota_{12},\iota_{12}]&=[\eta_{11},\eta_{11}]
=[\iota_{11},\iota_{11}]\circ\Sigma(\eta_{10}\wedge\eta_{10})
=[\iota_{11},\iota_{11}]\circ\eta^2_{21}\\
&=P(\eta_{23})\circ\eta_{22}=P(H(\theta))\circ\eta_{22}=0
\end{split}
\]
from $H(\theta)=\eta_{23}$ \cite[Lemma 7.5]{T}, 
we have 
\[
\mathrm{Ind}\{\eta_n,\nu_{n+1},\sigma_{n+4}\}
=\eta_n\circ\pi^{n+1}_{n+12}+\pi^n_{n+5}\circ\sigma_{n+5}
=0 
\text{\ \ for $n\ge 10$}.
\]
Hence $\{\eta_n,\nu_{n+1},\sigma_{n+4}\}$ consists of a single element for $n\ge 10$.
%Hence, to show (2), it suffices to show that 
%$0\in\{\eta_{11},\nu_{12},\sigma_{15}\}$. % by using \cite[Proposition 1.3]{T}. 
%From the fact that  $\{\eta_{10},\nu_{11},\sigma_{14}\}\subset\pi^{10}_{22}=\{[\iota_{10},\nu_{10}]\}$ \cite[Theorem 7.6]{T}, we have 
%$\{\eta_{11},\nu_{12},\sigma_{15}\}\supset \Sigma\{\eta_{10},\nu_{11},\sigma_{14}\}=0$. 
%Thus we obtain (2).

Next we show (1). %We use the relation $\nu_{10}\sigma_{13}=[\iota_{10},\eta_{10}]$ (\ref{nu10sg}).  
Let $\H P^2$ be the quaternionic projective plane,
$i_\H: S^4=\H P^1 \hookrightarrow\H P^2$ the inclusion
and $p_\H:\H P^2\to\H P^2/\H P^1=S^{8}$ the collapsing map. 
We know that $\H P^2$ is $2$-local homotopy equivalent to $S^4\cup_{\nu_4}e^8$.
For a CW-pair $(X,Y)$ and a CW-complex $Z$, we denote by 
$\pi_{k}(X:2)$, $\pi_{k}(X,Y:2)$ and $[X,Z:2]$
the direct sum of the torsion-free part and the $2$-primary component of 
$\pi_{k}(X)$, $\pi_{k}(X,Y)$ and $[X,Z]$, respectively. 
By \cite[Proposition 1.7]{T} and $\mathrm{Ind}\{\eta_{10},\nu_{11},\sigma_{14}\}=0$, there exist an extension $\ext({\eta}_{10})\in[\Sigma^7\H P^2,S^{10}:2]$ of $\eta_{10}$ 
and a coextension $\coe({\sigma}_{14})\in\pi_{22}(\Sigma^7\H P^2:2)$ of $\sigma_{14}$ such that
$\{\eta_{10},\nu_{11},\sigma_{14}\}=\ext({\eta}_{10})\coe({\sigma}_{14})$. 
We recall from \cite[Theorem 6.7.4]{YMW} that:
\[\pi_{22}(\Sigma^7\H P^2:2)=\{\overline{\sigma_{15}}\}\cong\Z_{128},
\text{\ where\ \ } (\Sigma^7p_\H)_*(\overline{\sigma_{15}})=\sigma_{15}.
\]
Since $\sharp\sigma_{15}=16$ \cite[Proposition 5.5]{T} and 
$(\Sigma^7p_\H)_*(\coe({\sigma}_{14}))=\Sigma\sigma_{14}=\sigma_{15}$,
we can set $\coe({\sigma}_{14})=(16k+1)\overline{\sigma_{15}}$ for some integer $k$.

We examine the EHP-sequence (see \cite[Theorem 4.9]{T1})
\[
\pi_{22}(\Sigma^7\H P^2:2)\rarrow{H}\pi_{22}(\Sigma(\Sigma^6\H P^2\wedge \Sigma^6\H P^2):2)\rarrow{P}\pi_{20}(\Sigma^6\H P^2:2). 
\]
By using the homotopy exact sequence of a pair 
$(\Sigma(\Sigma^6\H P^2\wedge \Sigma^6\H P^2),S^{21})$ and the facts that
$\H P^2 \wedge \H P^2$ has a cell structure 
$S^8 \cup e^{12}\cup e^{12}\cup e^{16}$,
$\pi^{21}_{22}=\{\eta_{21}\}\cong\Z_2$ \cite[Proposition 5.1]{T}
and
$\pi_{k}(\Sigma(\Sigma^6\H P^2\wedge \Sigma^6\H P^2),S^{21}:2)
\cong\pi_{k}((S^{25}\vee S^{25})\cup e^{29}:2)=0$ for $k=22$, $23$
from Blakers–Massey theorem \cite[VII-(7.12)]{WG78},
we have
\[
\pi_{22}(\Sigma(\Sigma^6\H P^2\wedge \Sigma^6\H P^2):2)
=\{\Sigma(\Sigma^6i_\H\wedge \Sigma^6i_\H)\circ\eta_{21}\}\cong\Z_2.
\]
By the naturally of the $P$-homomorphism:
\[
\begin{CD}
\pi^{21}_{22} @>{P}>> \pi^{10}_{20}\\
@V{(\Sigma(\Sigma^6i_\H\wedge \Sigma^6i_\H))_*}VV @V{(\Sigma^6i_\H)_*}VV \\
\pi_{22}(\Sigma(\Sigma^6\H P^2\wedge \Sigma^6\H P^2):2) @>{P}>> \pi_{20}(\Sigma^6\H P^2:2), 
\end{CD}
\]
and the relation $P(\eta_{21})=[\iota_{10},\eta_{10}]=\nu_{10}\sigma_{13}$ \eqref{nu10sg},
we have
\[
\begin{split}
P(\Sigma(\Sigma^6i_\H\wedge \Sigma^6i_\H)\circ\eta_{21})
=(\Sigma^6 i_\H)\circ P(\eta_{21}) = (\Sigma^6 i_\H)\circ\nu_{10}\sigma_{13}.
\end{split}
\]
Since $S^{13}\rarrow{\nu_{10}}S^{10}\rarrow{\Sigma^6 i_\H}\Sigma^6\H P^2$
is a cofibration in $2$-local,
we have
$P(\Sigma(\Sigma^6i_\H\wedge \Sigma^6i_\H)\circ\eta_{21})=0$
and $H$ is an epimorphism.
So, we obtain $H(\overline{\sigma_{15}})=\Sigma(\Sigma^6i_\H\wedge \Sigma^6i_\H)\circ\eta_{21}$
and
\[
H(\coe({\sigma}_{14}))
=H((16k+1)\overline{\sigma_{15}})=\Sigma(\Sigma^6i_\H\wedge \Sigma^6i_\H)\circ\eta_{21}.
\]
Since $\Sigma:[\Sigma^6\H P^2,S^9:2]\to [\Sigma^7\H P^2,S^{10}:2]$ is an isomorphism,
we can set $\ext({\eta}_{10})=\Sigma\ext({\eta}_{9})$,
where $\ext({\eta}_{9})\in[\Sigma^6\H P^2,S^9:2]$ is an extension of $\eta_9$. 
Thus, by $\eta^3_{19}=4\nu_{19}$ \cite[(5.5)]{T}, we have
\[
\begin{split}
H(\ext({\eta}_{10})\coe({\sigma}_{14}))
&=H(\Sigma\ext({\eta}_{9})\circ\coe({\sigma}_{14}))
=\Sigma(\ext({\eta}_{9})\wedge \ext({\eta}_{9}))\circ H(\coe({\sigma}_{14}))\\
&=\Sigma(\ext({\eta}_{9})\wedge \ext({\eta}_{9}))
\circ \Sigma(\Sigma^6i_\H\wedge \Sigma^6i_\H)
\circ\eta_{21}\\
&=\Sigma(\eta_{9}\wedge\eta_{9})\circ\eta_{21}=\eta^3_{19}=4\nu_{19}.
\end{split}
\]
Hence, by the relation $H(2[\iota_{10},\nu_{10}])=H([\iota_{10},\iota_{10}])\circ 2\nu_{19}=(\pm2\iota_{19})\circ 2\nu_{19}=4\nu_{19}$ from \cite[Proposition 2.7]{T}
and the fact that $H:\pi^{10}_{22}\to\pi^{19}_{22}$ is a monomorphism \cite[p.~79]{T}, we obtain $\ext({\eta}_{10})\coe({\sigma}_{14})=2[\iota_{10},\nu_{10}]$ and (1).

By using \cite[Proposition 1.3]{T} and the relation (1), we obtain
\[
\{\eta_n,\nu_{n+1},\sigma_{n+4}\}
\supset (-1)^{n-10}\Sigma^{n-10}\{\eta_{10},\nu_{11},\sigma_{14}\}
=(-1)^{n-10}\Sigma^{n-10}(2[\iota_{10},\nu_{10}])=0
\]
and (2). This completes the proof. 
\end{proof}
%%%%%%%%%%%%%%%%%%%%%

%%%%%%%%%%%%%%%%%%%%%%%% 16, 18, 21
\begin{lem}\label{a0}
\begin{enumerate}
\item
$\{\eta_{15},2\sigma_{16},\sigma_{23}\}=\omega_{15}+\eta^{*\prime}+\{\sigma_{15}\mu_{22}\}$.
\item
$\nu_{13}\circ\pi^{16}_{31}=0$.
\item
$\{\nu_{13},2\sigma_{16},\sigma_{23}\}=\{\nu_{13},\sigma_{16},2\sigma_{23}\}=\xi_{13}+x(\lambda+2\xi_{13})$
for some odd $x$.
\item
$\{\nu^2_{10},2\sigma_{16},\sigma_{23}\}
=\sigma_{10}\kappa_{17}+\{\sigma^3_{10}\}$.
\end{enumerate}
\end{lem}
\begin{proof}
We recall from \cite[I-Proposition 3.4 (6)]{Od} the relation
\[
\{\eta_{15},2\sigma_{16},\sigma_{23}\}\ni\omega_{15}+\eta^{*\prime}.
\]
By \cite[Theorems 7.4, 10.10]{T}, we have
\[
\begin{split}
\mathrm{Ind}\{\eta_{15},2\sigma_{16},\sigma_{23}\}
&=\eta_{15}\circ\pi^{16}_{31}+\pi^{15}_{24}\circ\sigma_{24}\\
&=\eta_{15}\circ\{\rho_{16},\bar{\varepsilon}_{16},[\iota_{16},\iota_{16}]\}
+\{\nu_{15}^3, \mu_{15}, \eta_{15}\varepsilon_{16}\}\circ\sigma_{24}.
\end{split}
\]
By relations
$\eta_{15}\rho_{16}=\mu_{15}\sigma_{24}=\sigma_{15}\mu_{22}$
\cite[Proposition 12.20 i)]{T},
$\eta_{15}\bar{\varepsilon}_{16}=0$
\eqref{et9be},
\[
\eta_{15}\circ[\iota_{16},\iota_{16}]
=[\eta_{15},\eta_{15}]
=[\iota_{15},\iota_{15}]\circ \Sigma(\eta_{14}\wedge\eta_{14})
=[\iota_{15},\iota_{15}]\circ\eta^2_{29}
=P(\eta^2_{31})=0
\]
from \cite[p.~160]{T},
$\nu_{21}\sigma_{24}=0$ \eqref{nu10sg}
and
$\varepsilon_{16}\sigma_{24}=0$ \eqref{bnu6sg}, 
we obtain 
$\mathrm{Ind}\{\eta_{15},2\sigma_{16},\sigma_{23}\}
=\{\sigma_{15}\mu_{22}\}$ 
and (1).

By \cite[Theorem 10.10]{T}, we have
\[
\nu_{13}\circ\pi^{16}_{31}
=\{\nu_{13}\rho_{16}, \nu_{13}\bar{\varepsilon}_{16}, 
\nu_{13}[\iota_{16},\iota_{16}]\}.
\]
So, by  relations
$\nu_{13}\rho_{16}=0$ \eqref{nu11ro}, 
$\nu_{13}\bar{\varepsilon}_{16}=0$ \eqref{bep},
\[
\nu_{13}\circ[\iota_{16},\iota_{16}]
=[\nu_{13},\nu_{13}]=[\iota_{13},\iota_{13}]\circ \Sigma(\nu_{12}\wedge\nu_{12})
=[\iota_{13},\iota_{13}]\circ\nu^2_{25}
=[\iota_{13},\nu^2_{13}]=0
\]
from \cite[(3.6)]{GM}, we obtain (2).

We recall from \cite[I-Proposition 3.4(8)]{Od} the relation
\[
\{\nu_{13},\sigma_{16},2\sigma_{23}\}
=\xi_{13}+x(\lambda+2\xi_{13})
\text{\ \ for some odd } x.
\]
By \cite[Proposition 1.2 ii)]{T},
we have $\{\nu_{13},\sigma_{16},2\sigma_{23}\}\subset\{\nu_{13},2\sigma_{16},\sigma_{23}\}$.
By \cite[Theorem 7.4]{T}, $\zeta_{13}\sigma_{24}=0$ \eqref{n7} and  (2), we have
\[
\mathrm{Ind}\{\nu_{13},2\sigma_{16},\sigma_{23}\}
=\nu_{13}\circ\pi^{16}_{31}+\pi^{13}_{24}\circ\sigma_{24}
=\{\zeta_{13}\sigma_{24}\}=0.
\]
Hence we obtain (3).

By using \cite[Proposition 1.2 iv)]{T}
and relations (3),  
$\nu_{10}\xi_{13}=\sigma^3_{10}$ \eqref{n9xi}, 
$\nu_{10}\lambda=\sigma_{10}\kappa_{17}$ \eqref{nulm}
and
$2\kappa_{17}=2\sigma_{17}^2=0$ \cite[Theorem 10.3]{T}, 
we obtain 
\[
\{\nu^2_{10},2\sigma_{16},\sigma_{23}\}
\supset\nu_{10}\circ\{\nu_{13},2\sigma_{16},\sigma_{23}\}=
\nu_{10}\xi_{13}+x\nu_{10}\lambda+2x\nu_{10}\xi_{13}
=\sigma^3_{10}+\sigma_{10}\kappa_{17}.
\]
By \cite[Theorem 10.3]{T}, (2) and the relation 
$\kappa_{10}\sigma_{24}=0$ \eqref{k7s}, 
we have
\[
\mathrm{Ind}\{\nu^2_{10},2\sigma_{16},\sigma_{23}\}
=\nu^2_{10}\circ\pi_{31}^{16}+\pi_{24}^{10}\circ\sigma_{24}
=\nu_{10}\circ\nu_{13}\circ\pi_{31}^{16}
+\{\sigma^3_{10}, \kappa_{10}\sigma_{24}\}=\{\sigma^3_{10}\}.
\]
Hence we obtain (4). This completes the proof.
\end{proof}
%%%%%%%%%%%%%%%%%%%%%%%% 16, 18, 21

%%%%%%%%%%%%%%%%%%%%%%%% 18, 19
\begin{lem}\label{sgnep}
\begin{enumerate}
\item $\langle\sigma,\nu,\sigma\rangle=\xi \in\pi^S_{18}$.
\item $\langle\sigma,\nu,\eta\sigma\rangle=
\langle\sigma,\nu,\varepsilon\rangle
=\langle2\sigma,\nu,\varepsilon\rangle
=\langle\nu,\sigma,\varepsilon\rangle
=\langle\nu,\varepsilon,\sigma\rangle=0\in\pi^S_{19}$.
\end{enumerate}
\end{lem}
\begin{proof}
Firstly, by relations $\sigma\nu=\nu\sigma=0$ \eqref{sgm11n},
$\nu\eta=0$ \eqref{n3},
$\nu\varepsilon=0$ \cite[(7.18)]{T}
and
$\sigma\varepsilon=\varepsilon\sigma=0$ \eqref{bnu6sg},
all Toada brackets in this lemma are well-defined.

Secondly, we recall that 
$\pi_{11}^S=\{\zeta\}$, $\pi_{12}^S=0$, $\pi_{16}^S=\{\eta\rho, \omega\}$ 
\cite[Theorems 7.4, 7.6, 12.16]{T} and relations 
%$\eta\nu=0$ \eqref{n2}, 
$\sigma\zeta=0$ \eqref{n7}, 
$\zeta\eta\sigma=0$ \eqref{ze7etsg} 
$\zeta\varepsilon=0$ \eqref{ze5bnu}, 
$\nu\eta=0$ \eqref{n3}
and 
$\nu\omega=0$ \eqref{n11om}.
%\cite[Theorem 14.1]{T}.
The indeterminacy of all Toada brackets in this lemma are $0$, because
\[
\sigma\circ\pi^S_{11}=\pi^S_{11}\circ\sigma=0
\mand
\sigma\circ\pi^S_{12}=\pi^S_{12}\circ\sigma=\pi^S_{11}\circ\eta\sigma=\pi^S_{11}\circ\varepsilon
=\nu\circ\pi^S_{16}=0.
\]
Hence,
the relation (1) follows directly from 
$\xi_{12}\in\{\sigma_{12},\nu_{19},\sigma_{22}\}$ \eqref{xi12}. 

Thirdly, by using \cite[(3.5) i)]{T}
and the fact $\langle\sigma,\nu,\eta\rangle\subset \pi^S_{12}=0$,
we have 
$
\langle\sigma,\nu,\eta\sigma\rangle\supset\langle\sigma,\nu,\eta\rangle\circ\sigma=0
%\ \bmod\ 0.%\sigma\circ\pi^S_{12}+\pi^S_{11}\circ\eta\sigma=0.
$
and $\langle\sigma,\nu,\eta\sigma\rangle=0$.

By $\varepsilon\in\langle\eta,2\iota,\nu^2\rangle$ \cite[(6.1)]{T}
and $\mathrm{Ind}\langle\eta,2\iota,\nu^2\rangle
=\eta\circ\pi^S_{7}+\pi^S_1\circ\nu^2=\{\eta\sigma\}$,
we have
$\langle\eta,2\iota,\nu^2\rangle=\varepsilon+\{\eta\sigma\}$.
Using \cite[(3.7), (3.8)]{T},
by  relations 
$\langle\sigma,\nu,\eta\rangle=0$,
$\langle\nu,\eta,2\iota\rangle\subset\pi_5^S=0$ \cite[Proposition 5.9]{T},
$\langle\sigma,\nu,\eta\sigma\rangle=0$,
$\pi_{12}^S=0$
and $\pi_{13}^S=0$ \cite[Theorem 7.7]{T},
for some $x\in\{0,1\}$,
we have
\[
\begin{split}
0&\in \langle\langle\sigma,\nu,\eta\rangle,2\iota,\nu^2\rangle
- \langle\sigma,\langle\nu,\eta,2\iota\rangle,\nu^2\rangle
+ \langle\sigma,\nu,\varepsilon+x\eta\sigma\rangle\\
&\subset\langle 0,2\iota,\nu^2\rangle - \langle\sigma,0,\nu^2\rangle
+\langle\sigma,\nu,\varepsilon\rangle
+\langle\sigma,\nu,x\eta\sigma\rangle\\
&=\pi^S_{13}\circ\nu^2+ \sigma\circ\pi^S_{12}+\langle\sigma,\nu,\varepsilon\rangle
+x\langle\sigma,\nu,\eta\sigma\rangle
=\langle\sigma,\nu,\varepsilon\rangle
\end{split}
\]
and $\langle\sigma,\nu,\varepsilon\rangle=0$.
We also have $\langle2\sigma,\nu,\varepsilon\rangle=0$ by 
$\langle2\sigma,\nu,\varepsilon\rangle\supset 2\iota\circ\langle\sigma,\nu,\varepsilon\rangle=0$.

By using \cite[(3.8), (3.9) i)]{T} and relations 
$\varepsilon=\bar{\nu}+\eta\sigma$ \eqref{et9sg},
$\langle\nu,\sigma,\bar{\nu}\rangle=\langle\eta\sigma,\sigma,\nu\rangle
 =\bar{\sigma}$
\cite[I-Proposition 3.3 (4)]{Od}
and $2\bar{\sigma}=0$ \cite[(12.17)]{T},
we have 
\[\begin{split}
\langle\nu,\sigma,\varepsilon\rangle=\langle\nu,\sigma,\bar{\nu}+\eta\sigma\rangle
 &\subset \langle\nu,\sigma,\bar{\nu}\rangle+ \langle\nu,\sigma,\eta\sigma\rangle
 =\bar{\sigma}+ \langle\eta\sigma,\sigma,\nu\rangle
 =\bar{\sigma}+\bar{\sigma}=2\bar{\sigma}=0.
\end{split}
\]

By the equation $\langle\sigma,\nu,\varepsilon\rangle
=\langle\nu,\sigma,\varepsilon\rangle=0$ and use of \cite[(3.9).ii), (3.10)]{T}, 
we have
\[
\begin{split}
0&\in
\langle\nu,\varepsilon,\sigma\rangle
-\langle\varepsilon,\sigma,\nu\rangle
+\langle\sigma,\nu,\varepsilon\rangle
=\langle\nu,\varepsilon,\sigma\rangle
-\langle\nu,\sigma,\varepsilon\rangle
=\langle\nu,\varepsilon,\sigma\rangle
\end{split}
\]  
and $\langle\nu,\varepsilon,\sigma\rangle=0$.
This completes the proof.
\end{proof}
%%%%%%%%%%%%%%%%%%%%%%%% 18, 19

%%%%%%%%%%%%% 19
\begin{lem}\label{n15sgep}% 19 13-32
\begin{enumerate}
\item
$\{\nu_{11},\sigma_{14},\varepsilon_{21}\}\subset\{\lambda'\eta_{29}, \xi'\eta_{29}\}$
\ and\ \ 
$\Sigma^2\{\nu_{11},\sigma_{14},\varepsilon_{21}\}=0$.
\item
$\{\nu_{13},\sigma_{16},\varepsilon_{23}\}_2=0$.
\item
$\{\nu_{n},\sigma_{n+3},\varepsilon_{n+10}\}=0$ for $n\ge 15$.
\end{enumerate}
\end{lem}
\begin{proof}
By relations $\nu_{11}\sigma_{14}=0$ \eqref{nu10sg} and
$\sigma_{14}\varepsilon_{21}=0$ \eqref{bnu6sg},
all Toada brackets in this lemma are well-defined.

We recall 
$
\pi^{11}_{30}=\{\lambda'\eta_{29}, \xi'\eta_{29}, 
\bar{\sigma}_{11}, \bar{\zeta}_{11}\}\cong(\Z_2)^3\oplus\Z_8$ 
and 
$\pi^S_{19}\cong\{\bar{\sigma},\bar{\zeta}\}\cong\Z_2\oplus\Z_8$ 
\cite[Theorem 12.23]{T}.
Hence, by equations $\langle\nu,\sigma,\varepsilon\rangle=0$ (Lemma \ref{sgnep} (2))
 and
$\Sigma^2(\lambda'\eta_{29})=\Sigma^2(\xi'\eta_{29})=0$ \eqref{Elameta},
we obtain
$\{\nu_{11},\sigma_{14},\varepsilon_{21}\}\subset\{\lambda'\eta_{29}, \xi'\eta_{29}\}$
and
$\Sigma^2\{\nu_{11},\sigma_{14},\varepsilon_{21}\}=0$.

By the facts 
$\pi^{\ell}_{\ell+11}=\{\zeta_{\ell}\}$ for $\ell\ge 13$,
$\pi^{14}_{30}=\{\omega_{14},\sigma_{14}\mu_{21}\}$ and
$\pi^{k}_{k+16}=\{\omega_{k},\sigma_{k}\mu_{k+7}\}$ for $k\ge 18$
 \cite[Theorems 7.4, 12.16]{T} and relations
$\nu_{13}\omega_{16}=0$ \eqref{n11om}, 
$\nu_{13}\sigma_{16}=0$ \eqref{nu10sg} and 
$\zeta_{13}\varepsilon_{24}=0$ \eqref{ze5bnu}, we have 
\begin{align*}
\mathrm{Ind}\{\nu_{13},\sigma_{16},\varepsilon_{23}\}_2
&=\nu_{13}\circ \Sigma^2\pi^{14}_{30}+\pi^{13}_{24}\circ\varepsilon_{24}=0
\shortintertext{and}
\mathrm{Ind}\{\nu_{n},\sigma_{n+3},\varepsilon_{n+10}\}
& =\nu_{n}\circ\pi^{n+3}_{n+19}+\pi^{n}_{n+11}\circ\varepsilon_{n+11}=0
\end{align*}
for $n\ge 15$.
By using \cite[Proposition 1.3]{T} and the second relation of (1), we have 
$
0=\Sigma^2\{\nu_{11}, \sigma_{14}, \varepsilon_{21}\}
 \subset \{\nu_{13},\sigma_{16},\varepsilon_{23}\}_2.
$
This leads to (2). Moreover, we have
\[
0=\Sigma^{n-13}\{\nu_{13}, \sigma_{16}, \varepsilon_{23}\}_2
 \subset \Sigma^{n-15}\{\nu_{15},\sigma_{18},\varepsilon_{25}\}
 \subset (-1)^{n-15}\{\nu_{n},\sigma_{n+3},\varepsilon_{n+10}\}
\]
for $n\ge 15$. This leads to (3) and completes the proof. 
\end{proof}
%%%%%%%%%%%%%%%%%%%%%

%%%%%%%%%%%%%%%%%%%%%%%%%%% 26
\begin{lem}\label{nu^2bs}% 26
\begin{enumerate}
%\item
%$\nu^2\bar{\sigma}=\eta\bar{\sigma}'$.
%\item
%$\nu^2_5\bar{\sigma}_{11}=\eta_5\bar{\sigma}'_6$.
\item
$\nu^2_9\bar{\sigma}_{15}\equiv 2(\sigma_9\nu^*_{16})\ 
\bmod\ 4\sigma_9\xi_{16}$
\ \ and\ \ \ 
$\nu^2_{10}\bar{\sigma}_{16} = 2\sigma_{10}\nu^*_{17}$.
%=2\sigma_{10}\xi_{17}
\item 
$\nu^2_{12}\bar{\sigma}_{18}=4\xi_{12}\sigma_{30}$
\ \ and\ \ \
$\nu^2_{17}\bar{\sigma}_{23}=4\xi_{17}\sigma_{35}
=[\iota_{17},\eta^2_{17}\sigma_{19}]\ne 0$.
\end{enumerate}
\end{lem}
\begin{proof}
By relations
$2(\sigma_9\nu^*_{16})=2E^2\phi''' $ 
\cite[I-Proposition 5.1 (2)]{Od}
and
$2E^2\phi'''\equiv\nu^2_9\bar{\sigma}_{15}
\bmod\ 4\sigma_9\xi_{16}$
\cite[II-Proposition 2.3 (2)]{Od},
we have the first of (1).
Moreover, by relations $4\sigma_{10}\xi_{17}=0$ from \cite[I-Proposition 6.3 (6)]{Od},
we have the second of (1).

By relations (1) and
$2\sigma_{12}\nu^*_{19}=4\xi_{12}\sigma_{30}$
\eqref{4xi12sg}, we have the first of (2).
Furthermore, by the relation
$P(\eta^2_{35}\sigma_{37})=4\xi_{17}\sigma_{35}$
\cite[I-Proposition 6.3 (11)]{Od},
we have  the second of (2).
This  completes the proof. 
\end{proof}
%%%%%%%%%%%%%%%%%%%%%

%
Finally, we show the following relations in the $24$-th homotopy groups of spheres.
\begin{prop}
\begin{enumerate}
\item
$\eta^*_{16}\bar{\nu}_{32}\equiv \nu^*_{16}\nu^2_{34}\ \bmod\ [\iota_{16},\nu^3_{16}]$.
\item
$\eta^*_{19}\bar{\nu}_{35}=\nu^*_{19}\nu^2_{37}
=\omega_{19}\bar{\nu}_{35}=[\iota_{19},\nu^2_{19}]$.
\end{enumerate}
\end{prop}
\begin{proof}
By using \cite[Propositions 1.4, 1.2 iv)]{T},
and relations
$\eta^*_{16}\in\{\sigma_{16},2\sigma_{23},\eta_{30}\}_1$ \eqref{eta*16},
$\eta_{30}\nu_{31}=0$ \eqref{n2}
and 
$\{2\iota_{30},\eta_{30},\nu_{31}\}\subset\pi^{30}_{35}=0$
\cite[Proposition 5.9]{T},
we have
\[\begin{split}
\eta^*_{16}\nu_{32}&\in\{\sigma_{16},2\sigma_{23},\eta_{30}\}\circ\nu_{32}
=-(\sigma_{16}\circ\{2\sigma_{23},\eta_{30},\nu_{31}\})
\\
&\supset -(\sigma^2_{16}\circ\{2\iota_{30},\eta_{30},\nu_{31}\})=0
\end{split}\]
with the indeterminacy
$\mathrm{Ind}(\sigma_{16}\circ\{2\sigma_{23},\eta_{30},\nu_{31}\})
=\sigma_{16}\circ\pi^{23}_{32}\circ\nu_{32}$.
By the fact $\pi^{23}_{32}=\{\nu^3_{23},\nu_{23},\eta_{23}\varepsilon_{24}\}$
\cite[Theorem 7.2]{T}
and the relation
$\sigma_{16}\nu_{23}=0$ \eqref{sgm11n}, we obtain
\[
\sigma_{16}\circ\pi^{23}_{32}\circ\nu_{32}
=\{\sigma_{16}\nu^4_{23}, \sigma_{16}\mu_{23}\nu_{32}, 
\sigma_{16}\eta_{23}\varepsilon_{24}\nu_{32}\}=0.
\]
and $\eta^*_{16}\nu_{32}=0$.
Hence, by $\nu_{32}\eta_{35}=0$ \eqref{n3},
the Toda bracket $\{\eta^*_{16},\nu_{32},\eta_{35}\}$ is well-defined.
Using \cite[Proposition 1.4]{T} and \cite[Theorem A]{Mi},
by relations
$\bar{\nu}_{32}=\{\nu_{32},\eta_{35},\nu_{36}\}$ \cite[Lemma 6.2]{T}, 
$\eta_{16}\bar{\kappa}_{17}\nu_{37}=\bar{\kappa}_{16}\eta_{36}\nu_{37}=0$,
$\sigma_{30}\nu_{37}=0$ and
$(\Sigma^3\lambda)\nu_{34}=[\iota_{16},\nu^2_{16}]$ \cite[(7.10)]{Mi},
we obtain
\[\begin{split}
\eta^*_{16}\bar{\nu}_{32}
&\in\eta^*_{16}\circ\{\nu_{32},\eta_{35},\nu_{36}\}
=\{\eta^*_{16},\nu_{32},\eta_{35}\}\circ\nu_{37}\\
&\subset\pi^{16}_{37}\circ\nu_{37}
=\{\eta_{16}\bar{\kappa}_{17},\sigma^3_{16},
(\Sigma^3\lambda)\nu_{34},\nu^*_{16}\nu_{34}\}\circ\nu_{37}
=\{[\iota_{16},\nu^3_{16}], \nu^*_{16}\nu^2_{34}\}.
\end{split}\]
On the other hand, by relations
$H(\eta^*_{16})=\eta_{31}$,
$H(\nu^*_{16})\equiv \nu_{31}\bmod 2\nu_{31}$ \cite[Lemma 12.14]{T},
$\eta_{31}\bar{\nu}_{32}=\nu^3_{31}$ \eqref{etbn},
$H([\iota_{16},\iota_{16}])=\pm 2\iota_{31}$ \cite[p.~24]{T}
and $2\nu^2_{31}=0$ \eqref{2nu5nu}, we have
\[
H(\eta^*_{16}\bar{\nu}_{32})=\eta_{31}\bar{\nu}_{32}
=\nu^3_{31}=H(\nu^*_{16}\nu^2_{34})
\mand H([\iota_{16},\nu^3_{16}])=\pm 2\iota_{31}\circ\nu^2_{31}
=\pm 2\nu^2_{31}=0.\]
This leads to (1).

By (1), we have $\eta^*_{19}\bar{\nu}_{35}=\nu^*_{19}\nu^2_{37}$.
By relations
$[\iota_{19},\iota_{19}]=\nu^*_{19}+\xi_{19}$ \cite[Corollary 12.25]{T}, 
$\xi_{19}\nu_{37}=\sigma^3_{19}$ \eqref{n9xi}
and $\sigma_{33}\nu_{40}=0$, we have
\[
[\iota_{19},\nu^2_{19}]
=[\iota_{19},\iota_{19}]\nu^2_{37}
=(\nu^*_{19}+\xi_{19})\nu^2_{37}
=\nu^*_{19}\nu^2_{37}+\sigma^3_{19}\nu_{40}
=\nu^*_{19}\nu^2_{37}.
\]
The last equality of (2) is just \cite[(5.34)]{MMO}.
This leads to (2) and completes the proof.
\end{proof}

%%%%%%%%%%%%%%%%%%%%%%%%%%%
\section{Some relations in homotopy groups of spheres and rotation groups}
%%%%%%%%%%%%%%%%%%%%%%%%%%%

Let be $SO(n)$ the $n$-th rotation group, 
$i_{m,n}: SO(m)\hookrightarrow SO(n)\ (m< n)$ the inclusion and
$i_n=i_{n-1,n}$.
% and $p_n: SO(n)\to SO(n)/SO(n-1)=S^{n-1}$ the projection. 
Let $\varDelta: \pi_k(S^n)\to\pi_{k-1}(SO(n))$ be the connecting map associated with the bundle $p_{n+1}: SO(n+1)\to S^n$.
Suppose that there exist elements $\alpha\in\pi_k(S^n)$ and $\beta\in\pi_k(SO(n+1))$ satisfying the relation 
$
{p_{n+1}}_*(\beta)=\alpha.
$
Then, $\beta$ is called a lift of $\alpha$ and is written $[\alpha]$. 
For a lift $[\alpha]\in\pi_k(SO(m))$, we write $[\alpha]_n={i_{m,n}}_*[\alpha]\in\pi_k(SO(n))$ for $m\leq n$. 

Let $J: \pi_k(SO(n))\to\pi_{k+n}(S^n)$ be the $J$-homomorphism
(\cite[p.~214]{WG}, \cite[(11.2)]{T}). 
We recall \cite[(11.2)]{T} that the $J$-homomorphism 
is defined by the composition
\[
J={G_n}_*\circ \Sigma^n : \pi_{k}(SO(n)) \to \pi_{k+n}(\Sigma^nSO(n)) \to \pi_{k+n}(S^n).
\]
Here $G_n:\Sigma^nSO(n)\to S^n$ is the Hopf construction obtained from the action of $SO(n)$
 as the rotations of $S^{n-1}$. 
For $\alpha\in\pi_k(SO(n))$ and $\beta\in\pi_\ell(S^k)$, we know 
\begin{equation}\label{Js}
J(\alpha\circ\beta)
={G_n}_*(\Sigma^n(\alpha\circ\beta))
={G_n}\circ\Sigma^n(\alpha\circ\beta)
={G_n}\circ\Sigma^n\alpha\circ \Sigma^n\beta
 =J\alpha \circ \Sigma^n\beta.
\end{equation}

We use well-known relations, for $\alpha\in\pi_k(SO(n))$ and $\beta\in\pi_k(S^n)$,
\begin{gather}\label{Ji}
J({i_{m,n}}_*(\alpha))=(-1)^{n-m}\Sigma^{n-m}J(\alpha) 
\text{\ \ \cite[(2.1)]{WJ}}, 
\\
\label{Jd}
%J(\varDelta(\beta))=[\beta,\iota_n]=(-1)^{kn}[\iota_n,\beta]
%\text{\ \ \cite[(3.6)]{WJ} and \cite[(7.5)]{WJ}},
J(\varDelta(\beta))=[\beta,\iota_n]
=(-1)^{(k+1)(n+1)}[\iota_n,\beta]
\text{\ \ \cite[(3.6)]{WJ} and \cite[X-(7.5)]{WG78}},
\\
\label{HJ}
HJ(\alpha) =(-1)^n \Sigma^n({p_n}_*(\alpha)) 
\text{\ \ \cite[p.~215]{WG}}.
\end{gather}

To obtain some relations between elements of homotopy groups of rotation groups and spheres,
we need the following theorem and lemma:
\begin{thm}\label{Ji0}
Assume that elements $\alpha \in\pi_{h+n}(SO(m))$,
$\beta\in \pi_{k}(S^h)$ and $\gamma\in \pi_{\ell}(S^{k})$ satisfy the conditions
$\alpha\circ\Sigma^n\beta=0$ and $\beta\circ\gamma=0$ for $n\geq 0$. 
Then the Toda bracket $\{J(\alpha),\Sigma^{m+n}\beta,\Sigma^{m+n}\gamma\}_{m+n}$ is well-defined and 
\[
J\{\alpha,\Sigma^n\beta,\Sigma^n\gamma\}_n 
\subset (-1)^m\{J(\alpha),\Sigma^{m+n}\beta,\Sigma^{m+n}\gamma\}_{m+n}.
\]
\end{thm}
\begin{proof}
By the relation (\ref{Js}),
we have $J(\alpha)\circ \Sigma^{m+n}\beta=J(\alpha\circ\Sigma^{n}\beta)=0$.
Hence the Toda bracket 
$\{J(\alpha),\Sigma^{m+n}\beta,\Sigma^{m+n}\gamma\}_{m+n}$ 
is well-defined.
By the definition of the $J$-homomorphism and \cite[Propositions 1.3, 1.2 iv)]{T}, we have
\[\begin{split}
J\{\alpha,\Sigma^n\beta,\Sigma^n\gamma\}_n &
= {G_m}_*(\Sigma^m\{\alpha,\Sigma^n\beta,\Sigma^n\gamma\}_n)
=G_m\circ \Sigma^m\{\alpha,\Sigma^n\beta,\Sigma^n\gamma\}_n\\
&\subset {G_m}\circ (-1)^m
\{\Sigma^m\alpha,\Sigma^{m+n}\beta,\Sigma^{m+n}\gamma\}_{m+n}\\
&\subset (-1)^m\{{G_m}\circ \Sigma^m\alpha,
\Sigma^{m+n}\beta,\Sigma^{m+n}\gamma\}_{m+n}\\
&= (-1)^m\{J(\alpha),\Sigma^{m+n}\beta,\Sigma^{m+n}\gamma\}_{m+n}.
\end{split}\]
This completes the proof. 
\end{proof}
\begin{lem}\label{id}
If $n\neq 1$, $3$ or $7$, then
$\{p_{n+1},i_{n+1},\varDelta(\iota_n)\}\ni (2k+1)\iota_n$ for some integer $k$.
\end{lem}
\begin{proof}
Let $P^n$ be the real $n$ dimensional projective space, $\gamma_n: S^n\to P^n$ be the projection
and $g_{n}: \P^{n}\to SO(n+1)$ be the canonical inclusion \cite[p.~202]{WG78}.  
We know 
\[
g_{n-1}\gamma_{n-1}=(-1)^{n+1}\varDelta(\iota_n)
\text{ \cite[IV-(10.5)]{WG78}},
\ i_{n+1} g_{n-1}=g_{n} i'_n,\ p_{n+1} g_{n}=p'_{n}
\text{ \cite[IV-(10.11)]{WG78}},
\]
where $i'_n: \P^{n-1}\to\P^n$ is the inclusion and $p'_n: \P^n\to S^n$ the collapsing map. 
By using \cite[Proposition 1.2 i)]{T} and the relation
$\varDelta(\iota_n)=(-1)^{n+1}g_{n-1}\gamma_{n-1}$, we have
\[\begin{split}
\{p_{n+1},i_{n+1},\varDelta(\iota_n)\}
&=\{p_{n+1},i_{n+1},(-1)^{n+1}g_{n-1}\gamma_{n-1}\}\\
&\supset  \{p_{n+1},i_{n+1},g_{n-1}\gamma_{n-1}\}\circ (-1)^{n+1}\iota_{n}
=  (-1)^{n+1}\{p_{n+1},i_{n+1},g_{n-1}\gamma_{n-1}\}.
\end{split}\]
By the equation
$\mathrm{Ind}\{p_{n+1},i_{n+1},(-1)^{n+1}g_{n-1}\gamma_{n-1}\}
=\mathrm{Ind}(-1)^{n+1}\{p_{n+1},i_{n+1},g_{n-1}\gamma_{n-1}\}$,
we have
\[
\{p_{n+1},i_{n+1},\varDelta(\iota_n)\}=
(-1)^{n+1}\{p_{n+1},i_{n+1},g_{n-1}\gamma_{n-1}\}.
\]
Moreover, by \cite[Proposition 1.2 ii), iii)]{T},
$i_{n+1} g_{n-1}=g_{n} i'_n$ and
$\{p'_n,i'_n,\gamma_{n-1}\}\ni \iota_n$ from \cite[(3.4)]{Sp62},
we have
\[\begin{split}
\{p_{n+1},i_{n+1},\varDelta(\iota_n)\}
&=(-1)^{n+1}\{p_{n+1},i_{n+1},g_{n-1}\gamma_{n-1}\}
\subset (-1)^{n+1}\{p_{n+1},i_{n+1} g_{n-1},\gamma_{n-1}\}\\
&=(-1)^{n+1}\{p_{n+1},g_{n} i'_n,\gamma_{n-1}\}
\supset (-1)^{n+1}\{p_{n+1} g_{n},i'_n,\gamma_{n-1}\}\\
&= (-1)^{n+1}\{p'_n,i'_n,\gamma_{n-1}\}\ni (-1)^{n+1}\iota_n.
\end{split}\]
So, we have the relation
\[
(-1)^{n+1}\iota_n+\alpha \in \{p_{n+1},i_{n+1},\varDelta(\iota_n)\}
\]
for some $\alpha\in \mathrm{Ind}\{p_{n+1},i_{n+1}g_{n-1},\gamma_{n-1}\}
= p_{n+1}\circ\pi_n(SO(n+1))+[\Sigma\P^{n-1},S^n]\circ \Sigma\gamma_{n-1}$.
By the facts ${p_{n+1}} \circ \pi_n(SO(n+1))=\Ker\{\varDelta:\pi_n(S^n)\to \pi_{n-1}(SO(n))\}$
and
\cite[IV-(10.6), (10.8)]{WG78}, we have
\[
{p_{n+1}} \circ \pi_n(SO(n+1))=
\{{(1+(-1)^{n-1})}\iota_n\}  \text{\ \ for $n\neq 1$, $3$ or $7$}.
\]
By exact sequences
\[
\pi_{k+1}(S^n)\rarrow{{\Sigma p'_k}^*}
[\Sigma\P^{k},S^n]\rarrow{{\Sigma i'_k}^*}
[\Sigma\P^{k-1},S^n]\rarrow{{\Sigma\gamma_{k-1}}^*}
\pi_{k}(S^n)
\]
for $k=2,\dots,\ n-1$,
we have $[\Sigma\P^{n-1},S^n]=\{\Sigma p'_{n-1}\}$.
By the equation $p'_{n-1}\circ\gamma_{n-1}=(1+(-1)^n)\iota_{n-1}$
from the cell structure of $P^{n-1}$, we have
\[
[\Sigma\P^{n-1},S^n]
\circ \Sigma\gamma_{n-1}=\{\Sigma(p'_{n-1}\circ\gamma_{n-1})\}
=\{(1+(-1)^n)\iota_{n}\}.
\]
Hence $\alpha\in\{2\iota_n\}$ for $n\neq 1$, $3$ or $7$.
This completes the proof. 
\end{proof}

%%%%%%%%%%%%%%%%%%%%%%%%%%%%%%%%%%%%%%%%%%%%%%%%%%%%%%%%%%%%%%
We denote by $R^n_k$ the direct sum of the torsion-free part and 
the $2$-primary component of $\pi_k(SO(n))$. 
We use the exact sequence induced from the fibration $p_n$: 
\[
({\mathcal{R}}^n_k) \hspace{0.5cm}
\cdots\rarrow{}\pi^{n}_{k+1}\rarrow{\varDelta}R^n_k\rarrow{{i_{n+1}}_*} R^{n+1}_k\rarrow{{p_{n+1}}_*}\pi^n_k\rarrow{\varDelta}R^n_{k-1} \to \cdots. 
\]
Denote by 
$M^n=S^{n-1}\cup_{2\iota_{n-1}}e^n$ the $\Z_2$-Moore space,
$\iota''_n$ the identity class of $M^n$, 
$i''_n: S^{n-1}\to M^n$ the inclusion map and
$p''_n:M^n\to S^n$ the collapsing map. 
Note that $M^n$ is a $2$-local space.
Let $\ext(\eta_{n})\in [M^{n+2},S^n]$ and $\coe(\eta_{n})\in\pi_{n+2}(M^{n+1})$ 
for $n\ge 3$ be an extension and a coextension of $\eta_n$, respectively.
%Let $\ext(\eta_{3})\in [M^{5},S^3]$ and $\coe(\eta_{3})\in\pi_{5}(M^{4})$ 
%be an extension and a coextension of $\eta_3$, respectively.
%By the definition of extension and coextension \cite[p.~13]{T},
%we put $\ext(\eta_{n})=\Sigma^{n-3}\ext(\eta_{3})$
%and $\coe(\eta_{n})=\Sigma^{n-3}\coe(\eta_{3})$ for $n\ge 3$.
We recall from \cite[p.~412]{GM} that:
\begin{gather}\label{2iota_M}
2\iota''_n =i''_n\eta_{n-1}p''_n \quad\text{for \ } n\geq 3,
%\\
%\label{2coe}
%2\coe(\eta_{n})=i''_{n+1}\eta^2_{n}
%\quad\text{for \ } n\geq 3,
\\
\notag%
\ext(\eta_{n})\coe(\eta_{n+1})=\pm \Sigma^{n-3}\nu'
\quad\text{for \ } n\geq 3,
\\
\label{ep_GM}
\varepsilon_n=\{\eta_n \ext(\eta_{n+1}),\coe(\eta_{n+2}),\nu_{n+4}\}_{n-5}
\quad (\coe(\eta_{7})\circ\nu_{9}=0)
\quad\text{for \ } n\geq 5.
\end{gather}
By the relation $2\nu_{5}=\Sigma^{3}\nu'$ \eqref{2nu5},
we have
\begin{equation}\label{2nu_n}
\ext(\eta_{n})\coe(\eta_{n+1})=\pm 2\nu_{n}
\quad\text{for \ } n\geq 5.
\end{equation}

Recall from \cite[Table 2]{K} a generator $[\eta_{11}]\in R^{11}_{12}$
such that $\sharp[\eta_{11}]=2$.
We consider the Toda bracket 
$\{[\eta_{11}]\ext(\eta_{12}),\coe(\eta_{13}),\nu_{15}\}_{6}\subset R^{12}_{19}$.
By the relation \eqref{2nu_n},
we have 
\[
[\eta_{11}]\ext(\eta_{12})\circ\coe(\eta_{13})
=[\eta_{11}]\circ (\pm 2\nu_{12}) = \pm 2[\eta_{11}]\circ \nu_{12}= 0.
\]
So, by the relation $\coe(\eta_{7})\circ \nu_{9}=0$ \eqref{ep_GM},
the Toda bracket $\{[\eta_{11}]\ext(\eta_{12}),\coe(\eta_{13}),\nu_{15}\}_6$
is well-defined.
By  \cite[Proposition 1.2 iv)]{T} and
%$\varepsilon_{11}\in\{\eta_{11}\ext(\eta_{12}),\coe(\eta_{13}),\nu_{15}\}_6$
\eqref{ep_GM}, we have
\[
\begin{split}
{p_{12}}_*\{[\eta_{11}]\ext(\eta_{12}),\coe(\eta_{13}),\nu_{15}\}_6
&\subset\{{p_{12}}_*([\eta_{11}])\ext(\eta_{12}),\coe(\eta_{13}),\nu_{15}\}_6\\
&=\{\eta_{11}\ext(\eta_{12}),\coe(\eta_{13}),\nu_{15}\}_6=\varepsilon_{11}.
\end{split}
\]
Hence, we can take a lift of $\varepsilon_{11}$ as
\begin{equation}\label{[ep11]}
[\varepsilon_{11}]
\in\{[\eta_{11}]\ext(\eta_{12}),\coe(\eta_{13}),\nu_{15}\}_6
\subset R^{12}_{19}.
\end{equation}
By using \cite[Proposition 1.2 iv)]{T} and relations 
$[\eta_{11}]_{13}=\varDelta(\iota_{13})$ \cite[Table 3]{K}
and \eqref{2nu_n}, we have
\[
\begin{split}
[\varepsilon_{11}]_{13}
&\in i_{13}\circ\{[\eta_{11}]\ext(\eta_{12}),\coe(\eta_{13}),\nu_{15}\}_6
\subset\{i_{13}\circ[\eta_{11}]\ext(\eta_{12}),\coe(\eta_{13}),\nu_{15}\}_6\\
&=\{[\eta_{11}]_{13}\ext(\eta_{12}),\coe(\eta_{13}),\nu_{15}\}_6
\subset\{\varDelta(\iota_{13}),\pm 2\nu_{12},\nu_{15}\}_6.
\end{split}
\]
Thus, by \cite[Theorem 2.1]{MT2} and the relation
$\iota_{13}\circ \Sigma(\pm 2\nu_{12})=\pm 2\nu_{13}$, 
there exists a lift $[\pm2\nu_{13}]\in R^{14}_{16}$ 
of $\pm 2\nu_{13}$ such that
$[\varepsilon_{11}]_{14}=[\pm2\nu_{13}]\nu_{16}$.
%We have ${p_{14}}_*([\pm2\nu_{13}])= \pm 2\nu_{13}$.
By \cite[IV-(10.4),(10.5)]{WG78}, we have
\[
{p_{14}}_*(\varDelta(\nu_{14}))
 ={p_{14}}_*(\varDelta(\iota_{14}))\circ\nu_{13}
 = 2\iota_{13}\circ\nu_{13}=2\nu_{13}.
\]
So, by using the exact sequence (${\mathcal{R}}^{13}_{16}$)
and $R^{13}_{16}=\{[\iota_7]_{13}\mu_7\}$ \cite[Proposition 4.1]{K},
we have
\[
\varDelta(\nu_{14})\equiv \mp[\pm2\nu_{13}]
\bmod {i_{14}}_*(R^{13}_{16})=\{[\iota_7]_{14}\mu_7\},
\]
where the signs are taken in the same order.
From this relation, by equations $\sharp\nu^2_{14}=2$ \eqref{2nu5nu} and 
$\mu_7\nu_{16}=0$ \eqref{mu7nu}, we have
$\varDelta(\nu^2_{14})= [\pm2\nu_{13}]\nu_{16}$.
Hence, the relation $[\varepsilon_{11}]_{14}=\varDelta(\nu^2_{14})$ holds.
By  using \eqref{Ji} and \eqref{Jd}
and the relations $[\varepsilon_{11}]_{14}=\varDelta(\nu^2_{14})$ and
$P(\nu_{29})=[\iota_{14},\nu_{14}]=\pm2\omega_{14}$
\cite[p.~159]{T},
we have
\[
\begin{split}
\Sigma^2 J([\varepsilon_{11}])=J([\varepsilon_{11}]_{14})
=J(\varDelta(\nu^2_{14}))=[\iota_{14},\nu^2_{14}]
=[\iota_{14},\nu_{14}]\nu_{30}
=2\omega_{14}\nu_{30}.
\end{split}
\]
By Theorem \ref{Ji0}, \cite[Proposition 2.3]{T}, \eqref{HJ}, 
%and relations 
%$\varepsilon_{23}=\{\eta_{23}\ext(\eta_{24}),\coe(\eta_{25}),\nu_{27}\}_{18}$
\eqref{ep_GM} and 
$\varepsilon_{23}
\in\{\eta_{23},\nu^2_{24},2\iota_{30}\}_{2}$
 \cite[(6.1)]{T}, we have
\[
\begin{split}
HJ([\varepsilon_{11}])
&\in HJ\{[\eta_{11}]\ext(\eta_{12}),\coe(\eta_{13}),\nu_{15}\}_6
\subset 
\{HJ([\eta_{11}])\circ\ext(\eta_{24}),\coe(\eta_{25}),\nu_{27}\}_{18}\\
&=\{\eta_{23}\ext(\eta_{24}),\coe(\eta_{25}),\nu_{27}\}_{18}
=\varepsilon_{23}
\in\{\eta_{23},\nu^2_{24},2\iota_{30}\}_{2}.
\end{split}
\]
Thus the element $J([\varepsilon_{11}])\in\pi^{12}_{31}$
satisfies the following:
\[
\Sigma^2 J([\varepsilon_{11}])=[\iota_{14},\nu^2_{14}]=2\omega_{14}\nu_{30}.
\mand
HJ([\varepsilon_{11}])
=\varepsilon_{23}\in\{\eta_{23},\nu^2_{24},2\iota_{30}\}_{2}.
\]
Now, we recall the element $\omega'\in\pi^{12}_{31}$ \cite[Proposition 11.14, Lemma 12.21]{T}
which satisfies the following:
\[
\Sigma^2\omega'=2\omega_{14}\nu_{30}
\mand
H(\omega')\in\{\eta_{23},\nu^2_{24},2\iota_{30}\}_{2}.
\]
Note that the relation
$H(\omega')\equiv \varepsilon_{23}
\bmod\ \bar{\nu}_{23}+\varepsilon_{23}$   is derived from the fact 
$\{\eta_{23},\nu^2_{24},2\iota_{30}\}_{2}
=\varepsilon_{23}+\{\bar{\nu}_{23}+\varepsilon_{23}\}$.
So, we can define $\omega'$ as
\begin{equation}\label{omega'=Jep11}
\omega'= J([\varepsilon_{11}])
\end{equation}
and it satisfies
\begin{equation}\label{rel-omega'}
\Sigma^2\omega'=2\omega_{14}\nu_{30}=[\iota_{14},\nu^2_{14}]
\mand
H(\omega')=\varepsilon_{23}.
\end{equation}
%%%%%%%%%%%%%%%%%%%%%%%%%%%%%%%%%%%%%%%%%%%%%%%%%%%%%%%%%%%%%%

We show 
\begin{lem}\label{omg'sg}% 26
$\omega'\sigma_{31}=0$.
\end{lem}
\begin{proof}By the equation
$\omega'= J([\varepsilon_{11}])$ \eqref{omega'=Jep11},
we have 
\[
\omega'\sigma_{31}=J([\varepsilon_{11}])\sigma_{31}
=J([\varepsilon_{11}]\sigma_{19}).
\]
So, it is sufficient to show the equation $[\varepsilon_{11}]\sigma_{19}=0$.
By using \cite[Proposition 1.4]{T} and
the relation
$[\varepsilon_{11}]
\in\{[\eta_{11}]\ext(\eta_{12}),\coe(\eta_{13}),\nu_{15}\}_6
\subset R^{12}_{19}$
\eqref{[ep11]},
we have
\[
\begin{split}
[\varepsilon_{11}]\sigma_{19}
&\in \{[\eta_{11}]\ext(\eta_{12}),\coe(\eta_{13}),\nu_{15}\}_6\circ\sigma_{19}\\
&\subset  \{[\eta_{11}]\ext(\eta_{12}),\coe(\eta_{13}),\nu_{15}\}_1\circ\sigma_{19}
=[\eta_{11}]\ext(\eta_{12})
\circ\Sigma\{\coe(\eta_{12}),\nu_{14},\sigma_{17}\}.
\end{split}
\]
We will examine $\Sigma\{\coe(\eta_{12}),\nu_{14},\sigma_{17}\}\subset\pi_{26}(M^{14})$ 
by using the exact sequence:
\[
\pi_{27}(M^{14},S^{13})
\rarrow{\partial}\pi^{13}_{26}
\rarrow{{i''_{14}}_*}\pi_{26}(M^{14})
\rarrow{j_*}\pi_{26}(M^{14},S^{13})
\rarrow{\partial}\pi^{13}_{25}.
\]
We consider the exact sequence \cite[Theorem 2.1]{Ja1}:
\[
\pi^{14}_{27}\rarrow{H_{\alpha}}
\pi^{13}_{13}\rarrow{Q}\pi_{26}(M^{14},S^{13})
\rarrow{\chi}
\pi^{14}_{26},
\]
where $\alpha=2\iota_{13}$.
By the fact $\pi^{14}_{27}=\{P(\iota_{29})\}\cong\Z$
\cite[Theorem 7.7]{T}
and the relation $HP(\iota_{29})=\pm 2\iota_{27}$ \cite[Proposition 2.7]{T},
we have
\[
\begin{split}
H_{2\iota_{13}}(P(\iota_{29}))
&={2\iota_{13}}_*(\Sigma^{-14}(HP(\iota_{29})))
={2\iota_{13}}_*(\Sigma^{-14}(\pm 2\iota_{27}))
={2\iota_{13}}_*(\pm 2\iota_{13})=\pm 4\iota_{13}.
\end{split}
\]
So, by the fact $\pi^{14}_{26}=0$ \cite[Theorem 7.6]{T},
we have $\pi_{26}(M^{14},S^{13})\cong\Z_4$.
Since there exists an element $\lambda_7$ of $\pi_{26}(M^{14})$ which has
order $8$ 
%\cite[P.~26]{GM2014},
\cite[Lemma 2.7]{MoM},
by the fact $\pi^{13}_{26}=\{(\Sigma\theta)\eta_{25}\}\cong\Z_2$
\cite[Theorem 7.7]{T}, we obtain
\[
\pi_{26}(M^{14})=\{\lambda_7\}\cong \Z_8
\mand
4\lambda_7= i''_{14}(\Sigma\theta)\eta_{25}.
\]
By using \cite[Propositions 1.2 iv), 1.3]{T} and relations
$2\iota''_{14} = i''_{14}\eta_{13} p''_{14}$ \eqref{2iota_M} and
$\{\eta_{13},\nu_{14},\sigma_{17}\}=0$ (Lemma \ref{etnusg} (2)),
we have
\[
\begin{split}
2(\Sigma\{\coe(\eta_{12}),\nu_{14},\sigma_{17}\})
&=2\iota''_{14}\circ \Sigma\{\coe(\eta_{12}),\nu_{14},\sigma_{17}\}
=i''_{14}\eta_{13} \circ\Sigma(p''_8 \circ \{\coe(\eta_{12}),\nu_{14},\sigma_{17}\})\\
&\subset i''_{14}\eta_{13} \circ \Sigma\{p''_8 \circ\coe(\eta_{12}),\nu_{14},\sigma_{17}\}
=i''_{14}\eta_{13} \circ  \Sigma\{\eta_{13},\nu_{14},\sigma_{17}\}=0.
\end{split}
\]
Hence, we have $\Sigma\{\coe(\eta_{12}),\nu_{14},\sigma_{17}\}\subset\{4\lambda_7\}$.
Since $SO(12)$ is a Lie group, by the fact $\sharp[\eta_{11}]=2$,
we obtain
\[
\begin{split}
[\varepsilon_{11}]\sigma_{19}
&\subset [\eta_{11}]\ext(\eta_{12})
\circ\Sigma\{\coe(\eta_{12}),\nu_{14},\sigma_{17}\}\subset\{[\eta_{11}]\ext(\eta_{12})\circ 4\lambda_7\}\\
&=\{4([\eta_{11}]\ext(\eta_{12})\lambda_7)\}
=\{4\iota'\circ[\eta_{11}]\circ\ext(\eta_{12})\lambda_7\}
=\{4[\eta_{11}]\circ\ext(\eta_{12})\lambda_7\}=0,
\end{split}
\]
where $\iota'$ is the identity class of $SO(12)$.
This completes the proof.
\end{proof}

%%%%%%%%%%%%%%%%%%%%%%%%%%%%%%%%%%%%%%%%%%%%%%%%%%%%%%%%%%%%%%
We recall  from \cite[(4.1)]{HKM} the relation:
\begin{equation}\label{Dsm9}
\varDelta(\sigma_9)=[\iota_7]_9(\bar{\nu}_7+\varepsilon_7).
\end{equation}
By $\bar{\nu}_7\sigma_{15}=\varepsilon_7\sigma_{15}=0$ \eqref{bnu6sg}, 
we obtain $\varDelta(\sigma^2_9)=[\iota_7]_9(\bar{\nu}_7+\varepsilon_7)\sigma_{15}=0$ 
and so, there exists a lift $[\sigma^2_9]\in R^{10}_{23}$\ of $\sigma^2_9$. 
We take 
\begin{equation}\label{[ss9]}
[\sigma^2_9]\in\{[\iota_7]_{10},\bar{\nu}_7+\varepsilon_7, \sigma_{15}\}
\end{equation}
 as follows.
By the relation
\[
[\iota_7]_{10}(\bar{\nu}_7+\varepsilon_7)=i_{10}\circ\varDelta(\sigma_9)
={i_{10}}_*\circ\Delta(\sigma_9)=0
\] from \eqref{Dsm9},
we can define the Toda bracket
$
\{[\iota_7]_{10},\bar{\nu}_7+\varepsilon_7, \sigma_{15}\}\subset R^{10}_{23}.
$
By \cite[Proposition 1.4]{T}, 
\eqref{Dsm9} and Lemma \ref{id},
we have 
\[
\begin{split}
{p_{10}}_*\{[\iota_7]_{10},\bar{\nu}_7+\varepsilon_7, \sigma_{15}\}
&\subset {p_{10}}_*\{i_{10},[\iota_7]_9(\bar{\nu}_7+\varepsilon_7), \sigma_{15}\}\\
&={p_{10}}_*\{i_{10},\varDelta(\sigma_9), \sigma_{15}\}
\subset{p_{10}}_*\{i_{10},\varDelta(\iota_9), \sigma^2_{8}\}\\
&=p_{10}\circ\{i_{10},\varDelta(\iota_9), \sigma^2_{8}\}
=-\{p_{10}, i_{10}, \varDelta(\iota_9)\}\circ\sigma^2_9
\ni(2k+1)\sigma^2_9\\
%%\ \bmod\ {p_{10}}_*R^{10}_{16}\circ\sigma_{16}=\{2\sigma^2_9\}.
\end{split}
\]
for some integer $k$.
By the facts ${p_{10}}_*(R^{10}_{16})=\{2\sigma_9\}$ \cite[Proposition 4.1]{K}
and
${p_{10}}_*(R^{10}_{9})=2\iota_9$ \cite[Table 2]{K},
we have
\begin{align*}
\mathrm{Ind}({p_{10}}_*\{[\iota_7]_{10},\bar{\nu}_7+\varepsilon_7, \sigma_{15}\})
&={p_{10}}_*([\iota_7]_{10}\circ \pi^7_{23}+R^{10}_{16}\circ\sigma_{16})\\
&={p_{10}}_*\circ{i_{10}}_*([\iota_7]_{9}\circ \pi^7_{23})
+{p_{10}}_*(R^{10}_{16})\circ\sigma_{16}=\{2\sigma^2_9\}\\
\shortintertext{and}
\mathrm{Ind}({p_{10}}_*\{i_{10},\varDelta(\iota_9), \sigma^2_{8}\})
&={p_{10}}_*(i_{10}\circ R^{9}_{23}+R^{10}_{9} \circ \sigma^2_9)\\
&={p_{10}}_*\circ{i_{10}}_*(R^{9}_{23})+ {p_{10}}_*(R^{10}_{9}) \circ \sigma^2_9 =\{2\sigma^2_9\}
\end{align*}
and hence, we obtain
\[
{p_{10}}_*\{[\iota_7]_{10},\bar{\nu}_7+\varepsilon_7, \sigma_{15}\}
=\sigma^2_9+\{2\sigma^2_9\}.
\]
Therefore, we can take 
$[\sigma^2_9]\in\{[\iota_7]_{10},\bar{\nu}_7+\varepsilon_7, \sigma_{15}\}$.

Let us recall from \cite[the proof of (4.27)]{MMO} the definition of $\psi_{10}\in\pi^{10}_{33}$:
\[
\psi_{10}\in\{\sigma_{10},\bar{\nu}_{17}+\varepsilon_{17},\sigma_{25}\}_{4}.
\] 
By Theorem \ref{Ji0} and relations \eqref{[ss9]}, 
$J([\iota_7])=\sigma_8$ \cite[(2.2)]{HKM}
 and \eqref{Ji},
we have
\[
\begin{split}
J([\sigma^2_9])&\in J\{[\iota_7]_{10},\bar{\nu}_7+\varepsilon_7, \sigma_{15}\}
\subset \{J([\iota_7]_{10}),\bar{\nu}_{17}+\varepsilon_{17}, \sigma_{25}\}_{10}
\subset \{\sigma_{10},\bar{\nu}_{17}+\varepsilon_{17}, \sigma_{25}\}_{4}.
\end{split}
\]
Therefore, 
we can define $\psi_{10}$ as
\begin{equation}\label{Jim}
\psi_{10}=J([\sigma^2_9]). 
\end{equation}
%%%%%%%%%%%%%%%%%%%%%%%%%%%%%%%%%%%%%%%%%%%%%%%%%%%%%%%%%%%%%%

We recall from \cite[Theorem 3]{Ke} the fact:
\[
\varDelta(\nu_{8n+5})\ne 0\ \text{\ for\ $n\ge 1$}. 
\]
In particular, we show 
\begin{lem}\label{wnu21}
$\varDelta(\nu_{21})=[\sigma^2_9]_{21}$ \ and \  $[\iota_{21},\nu_{21}]=\psi_{21}$.
\end{lem}
\begin{proof}
Using the exact sequence $({\mathcal{R}}^{n}_{23})$ for $n\leq 11$,
since $\pi^k_{23}$ and $\pi^k_{24}$ are finite groups for $k\leq 10$
\cite[Theorem 1.1]{MMO}, we have $\sharp [\sigma^2_9]<\infty$.
Moreover, by
the relation $\eqref{Jim}$ and 
the fact $\sharp\psi_{21}=2$ \cite[Theorem 1.1(b)]{MMO}, 
we observe that
\[
J([\sigma^2_9]_{21})=(-1)^{11}\Sigma^{11}J([\sigma^2_9])
 =\psi_{21}\neq 0.
\]
So, we obtain $2 \leq \sharp [\sigma^2_9]_{21}<\infty$.

We consider the exact sequence $({\mathcal{R}}^{21}_{23})$: 
\[
\pi^{21}_{24}\rarrow{\varDelta}R^{21}_{23}
\rarrow{{i_{21}}_*} R^{22}_{23}
\rarrow{{p_{22}}_*}\pi^{21}_{23}
\rarrow{\varDelta}R^{21}_{22}. 
\]
By \cite[Propositions 5.3, 5.6]{T} and \cite[Table, p. 161]{Ke}, % and \cite{Bo}, 
we have
\[
\pi^{21}_{24}=\{\nu_{21}\}\cong\Z_8,\ 
R^{21}_{23}\cong\Z\oplus\Z_2,\  
R^{22}_{23}\cong\Z 
\mand
\pi^{21}_{23}=\{\eta^2_{21}\}\cong\Z_2
\]
Since
$J(\varDelta(\eta^2_{21}))=[\iota_{21},\eta^2_{21}]=4\sigma^*_{21}\ne 0$
from \eqref{Jd} and  \cite[Lemma 8.3]{Mi}, 
we have $\varDelta(\eta^2_{21})\neq 0$.
So, 
${p_{22}}_*$ is trivial and ${i_{21}}_*$ is a split epimorphism. 
Hence, by the fact $2 \leq \sharp [\sigma^2_9]_{21}<\infty$, 
the direct summand $\Z_2$ in $R^{21}_{23}$ is generated by $[\sigma^2_9]_{21}$
and $\varDelta(\nu_{21})=[\sigma^2_9]_{21}$.  

By relations $J(\varDelta(\nu_{21}))=[\iota_{21},\nu_{21}]$ from \eqref{Jd} and 
$J([\sigma^2_9]_{21})=\psi_{21}$, we have
\[
[\iota_{21},\nu_{21}]
=J(\varDelta(\nu_{21}))=J([\sigma^2_9]_{21})=\psi_{21}.
\] 
This completes the proof.
\end{proof}

%The second equation in Lemma \ref{wnu21}
%is an excluded case in \cite[Theorem 3.6 (9)]{Th}. This theorem ensures the following.
%\begin{conj}\label{conj1}
%There exists a lift $[\sigma^2_{16k-7}]\in R^{16k-6}_{16k+7}$ of $\sigma^2_{16k-7}$ such that $\varDelta(\nu_{16k+5})=[\sigma^2_{16k-7}]_{16k+5}$ for $k\ge 2$.
%\end{conj}

From Lemma \ref{wnu21}, we obtain the following relations.
%%%%%%%%%%%%%%%%  23
\begin{prop}\label{etsg*}% 23 16-39
\begin{enumerate}
\item
$\sigma_{14}\omega_{21}+\omega_{14}\sigma_{30} = \psi_{14}$.
\item
$(\Sigma{\eta^*}')\sigma_{32}=[\iota_{16},\eta_{16}\sigma_{17}]$.
\item
$\eta_{16}\sigma^*_{17}
%=\sigma_{16}\omega_{23}+\omega_{16}\sigma_{32}+[\iota_{16},\eta_{16}\sigma_{17}]
=\psi_{17}+[\iota_{16},\eta_{16}\sigma_{17}]$.
\end{enumerate}
\end{prop}
\begin{proof}
%By the fact $\pi^S_{23}=\{\bar{\rho},\nu\bar{\kappa},\phi\}
%\cong\Z_{16}\oplus\Z_8\oplus(\Z_2)^2$ \cite[Theorem 1.1 (a)]{MMO}
%and the relation $\sigma\omega\equiv \phi\bmod\ 4\nu\bar{\kappa}, \bar{\rho}$
%\cite[(6.3)]{MMO}, we have $\sharp \sigma\omega=2$.
%
%\noindent\hrulefill
%
%Moreover, $\Ker\{\Sigma^\infty:\pi^{14}_{37}\to\pi^S_{23}\}\cong \Z_2$
%\cite[Theorem 1.1 (a)]{MMO}, we obtain
%$\sharp\sigma_{n}\omega_{n+7}=\sharp\omega_n\sigma_{n+14}=2$
%for $n\ge 14$.???
%
%\noindent\hrulefill
By relations $P(\nu_{43})=\omega_{21}\sigma_{37}-\sigma_{21}\omega_{28}$
\cite[III-Proposition 2.1 (3)]{Od}, $\sharp\omega_{21}=2$ \cite[Theorem 12.16]{T} 
and $[\iota_{21},\nu_{21}]=\psi_{21}$ (Lemma \ref{wnu21}),
we have
$\omega_{21}\sigma_{37}+\sigma_{21}\omega_{28}
 = [\iota_{21},\nu_{21}] = \psi_{21}$.
%By using \cite[Proposition]{Ba} and
%relations 
%$H(\omega_{14})=\nu_{27}$ \eqref{Homega},
%$H(\sigma_8)=\iota_{15}$ \cite[Lemma 5.14]{T},
%$\eta_{22}\sigma^*_{22}=[\iota_{21},\nu_{21}]$ 
%\cite[Example 2.3 (3)]{MM} 
%and
%$[\iota_{21},\nu_{21}]=\psi_{21}$ (Lemma \ref{wnu21}),
%we have
%\[
%\begin{split}
%\sigma_{21}\omega_{28}+\omega_{21}\sigma_{37}
%&=[\iota_{21},\iota_{21}]\circ\Sigma^{14}H(\omega_{14})
%\circ\Sigma^{29}H(\sigma_{8})\\
%&=[\iota_{21},\iota_{21}]\circ \nu_{41}
%=[\iota_{21},\nu_{21}]
%=\eta_{21}\sigma^*_{22}
%=\psi_{21}.
%\end{split}
%\]
Since $\Sigma^7:\pi^{14}_{37}\to\pi^{21}_{44}$ 
is an isomorphism \cite[Theorem 1.1 (a)]{MMO},
we have (1).

By relations 
$\Sigma{\eta^*}'\equiv[\iota_{16},\eta_{16}]\bmod\ \sigma_{16}\mu_{23}$
\cite[Lemma 2.10]{Og} and  
$\sigma_{16}\mu_{23}\sigma_{32}=\mu_{16}\sigma^2_{25}=0$ \eqref{mu3s^2},
we have  (2).

Recall from  
\cite[III-Proposition 2.5 (5)]{Od} the relation
$\eta_{15}\sigma^*_{16}
 ={\eta^*}'\sigma_{31}+\sigma_{15}\omega_{22}+\omega_{15}\sigma_{31}$,
by (1) and (2), we have 
\[
\eta_{16}\sigma^*_{17}
=\sigma_{16}\omega_{23}+\omega_{16}\sigma_{32}+(\Sigma{\eta^*}')\sigma_{32}
=\psi_{17}+[\iota_{16},\eta_{16}\sigma_{17}].
\]
This completes the proof.
\end{proof}

%%%%%%%%%%%%%%%%%%%%%%%%%%%
\section{Proof of Lemma \ref{HPxl13}}
%%%%%%%%%%%%%%%%%%%%%%%%%%%

We show the following lemma overlapping with \cite[Lemma 2.18]{Og}:
\begin{lem}\label{usgnep}% 18
\begin{enumerate}
\item
$\{2\sigma_{11},\nu_{18},\sigma_{21}\}=\xi'+ 2x\lambda'+4y\xi'$
for some integers $x$ and $y$.
\item
$\{2\sigma_{11},\nu_{18},\varepsilon_{21}\}=\lambda'\eta_{29}$.
\item
$\xi_{12}\eta_{30}\equiv\theta\sigma_{24}\ \bmod\ [\iota_{12},\eta_{12}\sigma_{13}]$. 
\item
$\{\sigma_{12},\nu_{19},\varepsilon_{22}\}_1=\{\sigma_{12},\nu_{19},\varepsilon_{22}\}\ni\omega'%+a\xi_{12}\eta_{30}
+\alpha$\ \ for some $\alpha\in\{\Sigma(\lambda'\eta_{29}), \Sigma(\xi'\eta_{29}) \}$.
\end{enumerate}%description
\end{lem}
\begin{proof}
By the fact that $\pi^{18}_{30}=0$, 
$\pi^{11}_{22}=\{\zeta_{11}\}$ \cite[Theorem 7.4, 7.6]{T},
$\zeta_{11}\sigma_{22}=0$ \cite[Proposition 3.5(7)]{Od} and 
$\zeta_{11}\varepsilon_{22}=0$ \eqref{ze5bnu}, we have 
\begin{gather*}
\mathrm{Ind}\{2\sigma_{11},\nu_{18},\sigma_{21}\}
=2\sigma_{11}\circ\pi^{18}_{30}+\pi^{11}_{22}\circ\sigma_{22}=0\\
\shortintertext{and}
\mathrm{Ind}\{2\sigma_{11},\nu_{18},\varepsilon_{21}\}
 =2\sigma_{11}\circ \Sigma\pi^{17}_{29}+\pi^{11}_{22}\circ\varepsilon_{22}=0.
\end{gather*}
Thus
$\{2\sigma_{11},\nu_{18},\sigma_{21}\}=\{2\sigma_{11},\nu_{18},\sigma_{21}\}_1$
and 
$\{2\sigma_{11},\nu_{18},\varepsilon_{21}\}=\{2\sigma_{11},\nu_{18},\varepsilon_{21}\}_1$
consist of a single element respectively.

By using \cite[Proposition 2.6]{T} and relations
$P(\eta_{21})=2\sigma_{10}\nu_{17}$ \eqref{nu10sg} and
$\pi^{21}_{22}=\{\eta_{21}\}\cong\Z_2$ \cite[Proposition 5.1]{T}, we obtain
\[
H\{2\sigma_{11},\nu_{18},\sigma_{21}\}_1
 =-P^{-1}(2\sigma_{10}\nu_{17})\circ\sigma_{22}=\eta_{21}\sigma_{22}.
\]
This means $H\{2\sigma_{11},\nu_{18},\sigma_{21}\}_1=H\xi'$ from \eqref{Hxi}
and hence, by the EHP-sequence, we have
\[
\{2\sigma_{11},\nu_{18},\sigma_{21}\}_1\equiv\xi'\ \bmod\ \Sigma\pi^{10}_{28}.
\]
By the fact $\pi^{10}_{28}=\{\lambda'',\xi'',\eta_{10}\bar{\mu}_{11}\}$
\cite[Theorem 12.22]{T} and relations $\Sigma\xi''=2\xi'$ and $\Sigma\lambda''=2\lambda'$
\cite[Lemma 12.19]{T}, we have
$\Sigma\pi^{10}_{28}=\{2\lambda',2\xi',\eta_{11}\bar{\mu}_{12}\}$ and
\[
\{2\sigma_{11},\nu_{18},\sigma_{21}\}_1
= \xi'+ 2x\lambda'+2y\xi'+z\eta_{11}\bar{\mu}_{12}
\]
for some integers $x$, $y$ and $z$.
Applying $\Sigma^\infty$, by
$\langle 2\sigma,\nu,\sigma\rangle=2\langle \sigma,\nu,\sigma\rangle=2\xi$
(Lemma \ref{sgnep} (1)),
$\Sigma^\infty\xi'=2\xi$ \eqref{2xi} and
$\Sigma^\infty\lambda'=4\nu^*$ \eqref{2lambda}, we have
\[
2\xi=2\xi+8x\nu^*+4y\xi+z\eta\bar{\mu}\in\pi^S_{18}.
\]
By the relation $\xi=-\nu^*$ \cite[Lemma 12.24]{T}
and the fact $\pi^S_{18}=\{\nu^*,\eta\bar{\mu}\}\cong\Z_8\oplus\Z_2$
\cite[Theorem 12.22]{T}, we have
$0=-4y\nu^*+z\eta\bar{\mu}$. Hence $y$ and $z$ are even.
This leads to (1). 

By using \cite[Proposition 2.6]{T} and 
the facts 
$P(\eta_{21})=2\sigma_{10}\nu_{17}$ \eqref{nu10sg}
and $\pi^{21}_{22}=\{\eta_{21}\}\cong\Z_2$ \cite[Proposition 5.1]{T}, 
we obtain
\[
H\{2\sigma_{11},\nu_{18},\varepsilon_{21}\}_1=-P^{-1}(2\sigma_{10}\nu_{17})\circ\varepsilon_{22}
=\eta_{21}\varepsilon_{22}.
\]
Since $H(\lambda'\eta_{29})=\varepsilon_{21}\eta_{29}=\eta_{21}\varepsilon_{22}$
from \eqref{Hlam}, 
by using the EHP-sequence,
we have 
\[
\{2\sigma_{11},\nu_{18},\varepsilon_{21}\}_1
\equiv\lambda'\eta_{29}\ \bmod\ \Sigma\pi^{10}_{29}.
\]
Applying $\Sigma^\infty$, by $\Sigma\pi^{10}_{29}=\{\bar{\sigma}_{11},\bar{\zeta}_{11}\}\cong\pi^S_{19}
=\{\bar{\sigma},\bar{\zeta}\}$ \cite[Lemma 12.23]{T},
$\langle 2\sigma_,\nu,\varepsilon\rangle=0$
(Lemma \ref{sgnep} (2)) and
$\Sigma^\infty(\lambda'\eta_{29})=0$ \eqref{Elameta}, we obtain (2).

By relations
$\xi_{12}\in\{\sigma_{12},\nu_{19},\sigma_{22}\}_1$ \eqref{xi12}
and 
$\theta\in\{\sigma_{12},\nu_{19},\eta_{22}\}_1$ \cite[Lemma 7.5]{T}, we have
\[\begin{split}
\xi_{12}\eta_{30}
&\in\{\sigma_{12},\nu_{19},\sigma_{22}\}_1\circ\eta_{30}
\subset\{\sigma_{12},\nu_{19},\sigma_{22}\eta_{29}\}_1
=\{\sigma_{12},\nu_{19},\eta_{22}\sigma_{23}\}_1\\
&\supset\{\sigma_{12},\nu_{19},\eta_{22}\}_1\circ\sigma_{24}
\ni\theta\sigma_{24}
\end{split}\]
with the indeterminacy 
$\sigma_{12}\circ \Sigma\pi^{18}_{30}+\pi^{12}_{23}\circ\eta_{23}\sigma_{24}$.
By $\Sigma\pi^{18}_{30}=0$ \cite[Lemma 7.6]{T}, 
$\pi^{12}_{23}=\{\zeta_{12}, [\iota_{12},\iota_{12}]\}$
\cite[Theorem 7.4]{T}
and 
$\zeta_{12}\eta_{23}\sigma_{24}=0$ \eqref{ze7etsg},
 we get that 
\[
\sigma_{12}\circ \Sigma\pi^{18}_{30}+\pi^{12}_{23}\circ\eta_{23}\sigma_{24}
=\{\zeta_{12}, [\iota_{12},\iota_{12}]\}\circ\eta_{23}\sigma_{24}
=\{[\iota_{12},\eta_{12}\sigma_{12}]\}.
\]
This leads to (3).

By using \cite[Proposition 2.6]{T} and relations
$P(\iota_{23})=\sigma_{11}\nu_{18}$, 
$2\sigma_{11}\nu_{18}=0$ \eqref{sgm11n}
and $2\varepsilon_{23}=0$ \cite[Theorem 7.1]{T}, we obtain
\begin{equation}\label{Hsg12}
H\{\sigma_{12},\nu_{19},\varepsilon_{22}\}_1
=-P^{-1}(\sigma_{11}\nu_{18})\circ\varepsilon_{23}
=(\iota_{23}+\{2\iota_{23}\})\circ\varepsilon_{23}
=\varepsilon_{23}.
\end{equation}
%By the relations $H(\xi_{12}\eta_{30})=H(\xi_{12})\eta_{30}=\sigma_{23}\eta_{30}$
%\cite[Lemma 12.14]{T}, 
%$\sigma_{23}\eta_{30}=\bar{\nu}_{23}+\varepsilon_{23}$ \eqref{et9sg}
%and 
%$H(\omega')\equiv\varepsilon_{23}\ 
%\bmod\ \varepsilon_{23}+\bar{\nu}_{23}$ \eqref{e2omg}, 
%we have
%$H(\omega'+a\xi_{12}\eta_{30})=\varepsilon_{23}$
%for some $a\in\{0,1\}$.
Hence, by using the EHP-sequence, the relation 
$H(\omega')=\varepsilon_{23}$ \eqref{rel-omega'}
and the fact $\Sigma\pi^{11}_{30}
=\{\Sigma(\lambda'\eta_{29}),  \Sigma(\xi'\eta_{29}), \bar{\zeta}_{12}, \bar{\sigma}_{12}\}$
\cite[Theorem 12.23]{T},
for any representative $\beta$ in $\{\sigma_{12},\nu_{19},\varepsilon_{22}\}_1$,
we have
\[
\beta \equiv \omega'
\ \bmod\ \Sigma\pi^{11}_{33}
=\{\Sigma(\lambda'\eta_{29}),  \Sigma(\xi'\eta_{29}), \bar{\zeta}_{12}, \bar{\sigma}_{12}\}.
\]
By relations $\Sigma^\infty\beta\in \langle\sigma, \nu, \varepsilon\rangle=0$ (Lemma \ref{sgnep}),
$\Sigma^\infty\omega'=0$ \eqref{rel-omega'},
%$\xi\eta=0$ \eqref{x13et},
$\Sigma^\infty(\lambda'\eta_{29})=\Sigma^\infty(\xi'\eta_{29})=0$ \eqref{Elameta},
and $\{\bar{\zeta}_{12}, \bar{\sigma}_{12}\}\cong\pi^S_{19}$
\cite[Theorem 12.23]{T}, we obtain
\[
\beta \equiv \omega' \ \bmod\  \{\Sigma(\lambda'\eta_{29}), \Sigma(\xi'\eta_{29})\}.
\]
Hence, there exists $\alpha\in\{\Sigma(\lambda'\eta_{29}), \Sigma(\xi'\eta_{29})\}$ such that
\[
\omega' + \alpha = \beta \in \{\sigma_{12},\nu_{19},\varepsilon_{22}\}_1.
\]
Moreover, by the fact 
$\Sigma\pi^{18}_{30}=\pi^{19}_{31}=0$ \cite[Theorem 7.6]{T}, 
we obtain
$\mathrm{Ind}\{\sigma_{12},\nu_{19},\varepsilon_{22}\}_1
=\mathrm{Ind}\{\sigma_{12},\nu_{19},\varepsilon_{22}\}$
and (4).
This completes the proof.
\end{proof}

By $\xi_{12}\eta_{30}\equiv\theta\sigma_{24}\ \bmod\ [\iota_{12},\eta_{12}\sigma_{13}]$
(Lemma \ref{usgnep} (3)), 
$\eta_{12}\sigma^2_{13}=0$ \eqref{et9sg2},
$\theta\sigma^2_{24}=\sigma_{12}\bar{\sigma}_{19}$ \cite[I-Proposition 3.5 (6)]{Od},
$[\iota_{13},\sigma_{13}]=(\Sigma\theta)\sigma_{25}$ \eqref{x13et} and 
$\sigma_{13}\bar{\sigma}_{20}\ne 0$ \cite[Theorem 1(b)]{Od},
we have
\begin{equation}\label{smbsm}% 26
\xi_{12}\eta_{30}\sigma_{31}=\theta\sigma^2_{24}=\sigma_{12}\bar{\sigma}_{19}
\text{\ \ and\ \ }
[\iota_{13},\sigma^2_{13}]=\sigma_{13}\bar{\sigma}_{20}\ne 0.
\end{equation}

%%%%%%%%%%%%%%%%%%%%% 19
In \cite{Og}, {\^O}guchi defined a generator 
$\lambda$ of $\pi^{13}_{31}$ \cite[Lemma 2.18 (2)]{Og}
which is different from Toda \cite[Lemma 12.18]{T}.
So, we denote his generator by $\lambda^o$.
$\lambda^o$ is defined as follows.

Recall from \cite[p.16, 17]{T} notations
$(S^{m})_\infty$, $(S^{m})_t$, $i:(S^{m})_\infty \to \Omega S^{m+1}$,
$\Omega_1:\pi^{k+1}_{m+1}\to \pi_{k}((S^{m})_\infty)$,
$h'_{m}: ((S^{m})_2, S^{m}) \to (S^{2m}, e_0)$
and
$h_{m}: ((S^{m})_\infty, S^{m}) \to ((S^{m})_\infty, e_0)$.
We denote by $j^m_{k,\ell}:(S^{m})_k \to (S^{m})_\ell$
and $j^m_{k}:(S^{m})_k \to (S^{m})_\infty$ the inclusions.
We consider the Toda bracket 
$\{j^{12}_{1,2}, [\iota_{12},\iota_{12}], \nu^2_{23}\}
\subset \pi_{30}((S^{12})_2:2)$.
By the fact %$(S^{12})_1 =S^{12}$ and
$(S^{12})_2 =  S^{12}\cup_{[\iota_{12},\iota_{12}]} e^{24}$, 
%and the sequence 
%$S^{23} \xrightarrow{[\iota_{12},\iota_{12}]} S^{12}
%\xrightarrow{j_{1,2}} (S^{12})_2$ is a cofibration,
we have 
$j^{12}_{1,2} \circ [\iota_{12},\iota_{12}]=0$.
So, by the equation
$[\iota_{12},\iota_{12}]\circ \nu^2_{23}
= P(\nu^2_{25})=P(H(\lambda))=0$
from $H(\lambda)=\nu^2_{25}$ \cite[Lemma 12.18]{T},
the Toda bracket $\{j^{12}_{1,2}, [\iota_{12},\iota_{12}], \nu^2_{23}\}$
is well-defined.
By the definition 
$H=\Omega_1^{-1}\circ {h_{12}}_*\circ \Omega_1
:\pi^{13}_{31}\to\pi^{25}_{31}$
\cite[(2.7)]{T} and
the equation 
${h_{12}}\circ {j^{12}_2} ={j^{24}_1}\circ{h'_{12}}$,
we have
\[
\begin{split}
H(\Omega_1^{-1}(j^{12}_{2}\circ\{j^{12}_{1,2}, [\iota_{12},\iota_{12}], \nu^2_{23}\}))
&=(\Omega_1^{-1}\circ {h_{12}}_*)
(j^{12}_{2}\circ\{j^{12}_{1,2}, [\iota_{12},\iota_{12}], \nu^2_{23}\})\\
&=(\Omega_1^{-1}\circ {j^{24}_1}_*)
(h'_{12}\circ\{j^{12}_{1,2}, [\iota_{12},\iota_{12}], \nu^2_{23}\}).
\end{split}
\]
Since $h'_{12}:(S^{12})_2 \to (S^{12})_2/(S^{12})_1 =S^{24}$ is
the collapsing map, 
the sequence 
$S^{23} \xrightarrow{[\iota_{12},\iota_{12}]} 
S^{12} \xrightarrow{j^{12}_{1,2}} 
(S^{12})_2 \xrightarrow{h'_{12}} S^{24}$ is a cofibration.
Hence, by using \cite[Proposition 1.4]{T}, \cite[(3.4)]{Sp62} and 
$\sharp\nu^2_{24}=2$ \eqref{2nu5nu}, we have
\[
h'_{12}\circ\{j^{12}_{1,2}, [\iota_{12},\iota_{12}], \nu^2_{23}\}
= \{h'_{12}, j^{12}_{1,2}, [\iota_{12},\iota_{12}]\}\circ \nu^2_{24}
\ni \nu^2_{24}.
\]
Thus, by the fact 
$\Omega_1^{-1}\circ {j^{24}_1}_*=\Sigma:\pi^{24}_{30}\to\pi^{25}_{31}$
\cite[(2.3)]{T}, we have
\[
H(\Omega_1^{-1}(j^{12}_{2}\circ\{j^{12}_{1,2}, [\iota_{12},\iota_{12}], \nu^2_{23}\}))
\ni \Sigma \nu^2_{24} =\nu^2_{25}
\]
and there exists an element
\begin{equation*}\label{lmbda^o}
\lambda^o
\in 
\Omega_1^{-1}(j^{12}_{2}\circ\{j^{12}_{1,2}, [\iota_{12},\iota_{12}], \nu^2_{23}\})
\subset \pi^{13}_{31}
\end{equation*}
such that $H(\lambda^o)=\nu^2_{25}$.

%
%\[
%\begin{split}
%\lambda^o\eta_{31}
%&\in \Omega_1^{-1}(j^{12}_{2}\circ\{j^{12}_{1,2},
%[\iota_{12},\iota_{12}], \nu^2_{23}\})\circ \eta_{31}
%=\Omega_1^{-1}(j^{12}_{2}\circ\{j^{12}_{1,2}, 
%[\iota_{12},\iota_{12}], \nu^2_{23}\}\circ \eta_{30})\\
%&=\Omega_1^{-1}(j^{12}_{2}\circ j^{12}_{1,2}
%\circ \{[\iota_{12},\iota_{12}], \nu^2_{23},\eta_{29}\}
%\ni \Omega_1^{-1}({j^{12}_{1}}_*(\omega'))
%=\Sigma\omega'
%\end{split}
%\]

Then, the following relation holds:
\begin{equation}\label{lam^oeta}
\lambda^o\eta_{31}\equiv \Sigma\omega'\bmod\ \xi_{13}\eta_{31}
\quad\text{\cite[Proposition 2.20 (2)]{Og}}.
\end{equation}

By using the EHP-sequence and relations 
$H(\lambda)=\nu^2_{25}$
and
$ H(\lambda^o)=\nu^2_{25}$,
we have
$\lambda\equiv \lambda^o \bmod\ \Sigma\pi^{12}_{30}$.
By the fact $\Sigma\pi^{12}_{30}
=\{\eta_{13}\bar{\mu}_{14},\xi_{13},\Sigma^2\lambda'\}$
\cite[p.~167]{T}
and relations $\Sigma^2\lambda'=2\lambda$ \eqref{2lambda},
$\eta^2_{13}\bar{\mu}_{15}=4\bar{\zeta}_{13}$ \eqref{4bze5}
and $2\iota_{31}\circ\eta_{31}=2\eta_{31}=0$ \eqref{2eta},
we have
\[
\lambda\eta_{31}\equiv \lambda^o\eta_{31}
\bmod\ 4\bar{\zeta}_{13},\ \xi_{13}\eta_{31}.
\]
Moreover, by relations $\Sigma^4(\lambda\eta_{31})=2\nu^*_{17}\eta_{35}=0$
\eqref{2lambda},
$4\bar{\zeta}\neq 0$ \cite[Theorem 12.23]{T},
$\xi_{14}\eta_{32}=\Sigma P(\sigma_{27})=0$ \cite[p.~166]{T}
and
$\Sigma^2(\lambda^o\eta_{31})
\equiv \Sigma^3\omega'=\Sigma [\iota_{14},\nu^2_{14}]=0 \bmod \xi_{15}\eta_{33}=0$
from 
\eqref{lam^oeta} and \eqref{rel-omega'}, we have
$\lambda\eta_{31}\equiv \lambda^o\eta_{31}\bmod\ \xi_{13}\eta_{31}$.
In addition, by the relation \eqref{lam^oeta}, we obtain
\begin{equation}\label{lameta}% 19 13-32
\lambda\eta_{31}\equiv \Sigma\omega'\bmod\ \xi_{13}\eta_{31}.
\end{equation}
%%%%%%%%%%%%%

We show 
\begin{lem}\label{omg'}
\begin{enumerate}
\item
$\lambda\eta_{31}=\Sigma\omega'=\{\sigma_{13},\nu_{20},\varepsilon_{23}\}_n$ 
\ for \ $n\le 13$.
\item
$\nu_{13}\eta^*_{16}=\Sigma\omega'+\xi_{13}\eta_{31}$.
\end{enumerate}
\end{lem}
\begin{proof}
By relations \eqref{lameta}, $\lambda\sigma_{31}=0$ \eqref{lmdsg}, 
Lemma \ref{omg'sg} and 
$\xi_{13}\eta_{31}\sigma_{31} \ne 0$ \eqref{smbsm}, 
we have
\[
0=\lambda\sigma_{31}\eta_{38}=\lambda\eta_{31}\sigma_{32}
\equiv (\Sigma\omega')\sigma_{32}
=\Sigma(\omega'\sigma_{31})=0\ \bmod\ \xi_{13}\eta_{31}\sigma_{31}\ne 0.
\]
This leads to the first equality of (1). 
By the facts $\pi^{13}_{24}=\{\zeta_{13}\}$ and $\pi^{20}_{32}=0$ \cite[Theorems 7.4, 7.6]{T} and 
the equation $\zeta_{13}\varepsilon_{24}=0$ \eqref{ze5bnu}, for $n\le 13$, 
we have
\[
\mathrm{Ind}\{\sigma_{13},\nu_{20},\varepsilon_{23}\}_n
=\sigma_{13}\circ\Sigma^n\pi^{20-n}_{32-n}+\pi^{13}_{24}\circ\varepsilon_{24}=0.
\]
Hence the second equality of (1) follows from Lemma \ref{usgnep} (4)
and $\Sigma^2(\lambda'\eta_{29})=\Sigma^2(\xi'\eta_{29})=0$ \eqref{Elameta}. 
 
By using \cite[Proposition 1.2 i)]{T} and relations 
$\eta^*_{16}\in\{\sigma_{16},2\sigma_{23},\eta_{30}\}$ \eqref{eta*16}
and
$\sigma^*_{16}\in\{\sigma_{16},2\sigma_{23},\sigma_{30}\}$  \cite[p.~187]{Mi}, 
we obtain 
\[
\eta^*_{16}\sigma_{32}
\in\{\sigma_{16},2\sigma_{23},\eta_{30}\}\circ\sigma_{32}
\subset\{\sigma_{16},2\sigma_{23},\eta_{30}\sigma_{31}\}
\supset\{\sigma_{16},2\sigma_{23},\sigma_{30}\}\circ\eta_{38}
\ni\sigma^*_{16}\eta_{38}
\]
with the indeterminacy $\sigma_{16}\circ\pi^{23}_{39}+\pi^{16}_{31}\circ\eta_{31}\sigma_{32}$.
By the facts
$\pi^{23}_{39}=\{\omega_{23},\sigma_{23}\mu_{30}\}$ \cite[Theorem 12.16]{T}
and
$\pi^{16}_{31}=\{\rho_{16},\bar{\varepsilon}_{16},[\iota_{16},\iota_{16}]\}$ \cite[Theorem 10.10]{T},
we have
\[
\sigma_{16}\circ\pi^{23}_{39}+\pi^{16}_{31}\circ\eta_{31}\sigma_{32}
=\{\sigma_{16}\omega_{23}, \sigma^2_{16}\mu_{30}, 
\rho_{16}\eta_{31}\sigma_{32},\bar{\varepsilon}_{16}\eta_{31}\sigma_{32},
[\iota_{16},\eta_{16}\sigma_{17}]\}.
\]
By relations
$\omega_{23}\equiv \eta^*_{23} \bmod\ \sigma_{23}\mu_{30}$ \cite[Proposition 12.20 ii)]{T},
$\sigma^2_{16}\mu_{30}=0$ \eqref{mu3s^2},
$\rho_{16}\eta_{31}\sigma_{32}=0$ \eqref{r13etas}
and
$\bar{\varepsilon}_{16}\eta_{31}\sigma_{32}=0$ \eqref{et9be},
we have
\[
\sigma_{16}\circ\pi^{23}_{39}+\pi^{16}_{31}\circ\eta_{31}\sigma_{32}
=\{\sigma_{16}\eta^*_{23},[\iota_{16},\eta_{16}\sigma_{17}]\}
\]
and there exists integers $a$ and $b$ such that
\[
\eta^*_{16}\sigma_{32}= \sigma^*_{16}\eta_{38}+ a\sigma_{16}\eta^*_{23}+ b [\iota_{16},\eta_{16}\sigma_{17}].
\]
%Applying $E^\infty$, by relations $\sigma^*=0$ from \cite[Lemma 8.3]{Mi}
%and $\eta^*\sigma=\sigma\eta^*$, we obtain $a=1$.
By the relations 
$\nu_{13}\sigma_{16}=0$ \eqref{nu10sg} and
$\nu_{13}\eta_{16}=0$ \eqref{n3}, we have
\[
\begin{split}
\nu_{13}\eta^*_{16}\sigma_{32}
&=\nu_{13}(\sigma^*_{16}\eta_{38}+a\sigma_{16}\eta^*_{23}
+ b [\iota_{16},\eta_{16}\sigma_{17}])\\
&=\nu_{13}\sigma^*_{16}\eta_{38}+ b [\nu_{13},\nu_{13}\eta_{16}\sigma_{17}]
=\nu_{13}\sigma^*_{16}\eta_{38}.
\end{split}
\]
Moreover, by using \cite[Proposition 1.4]{T} and relations 
$\{\nu_{13},\sigma_{16},2\sigma_{23}\}=\xi_{13}+x(\lambda+2\xi_{13})$
for some odd $x$ (Lemma \ref{a0} (3)),
$\lambda\sigma_{31}=0$ \eqref{lmdsg} and 
$\xi_{13}\eta_{31}\sigma_{32}\neq 0$ \eqref{smbsm}, we see that
\[
\begin{split}
\nu_{13}\eta^*_{16}\sigma_{32}
&=\nu_{13}\sigma^*_{16}\eta_{38}
\in\nu_{13}\circ\{\sigma_{16},2\sigma_{23},\sigma_{30}\}\circ\eta_{38}\\
&=\{\nu_{13},\sigma_{16},2\sigma_{23}\}\circ\eta_{31}\sigma_{32}
=\xi_{13}\eta_{31}\sigma_{32}\neq 0.
\end{split}
\]
Thus, by  relations
$\nu_{13}\eta^*_{16}\equiv \Sigma\omega'\ \bmod\ \xi_{13}\eta_{31}$
\cite[Proposition 2.20 (8)]{Og}
and 
$\omega'\sigma_{31}=0$ (Lemma \ref{omg'sg}), we have the equation (2) and completes the proof.
\end{proof}

Here, we need the following property of the generalized $P$-homomorphism \cite{T1, IM}.
\begin{lem}\label{Miy}
Let $k\ge 3$ and $X$, $Y$, $Z$ and $W$ be CW-complexes and 
$\alpha \in [\Sigma^{k}Y, \Sigma X\wedge X]$,
$\beta \in [Z, Y]$ and $\gamma \in [W,Z]$.
Suppose that the generalized P-homomorphisms 
$P:[\Sigma^{k}A, \Sigma X\wedge X] \to [\Sigma^{k-2}A, X]$ for $A=Y,$ $\Sigma W$ and $Y\cup_{\beta} CZ$
are well-defined and  $\alpha\circ \Sigma^{k}\beta=0$ and $\beta\circ\gamma=0$.
Then the Toda bracket $\{P(\alpha), \Sigma^{k-2}\beta, \Sigma^{k-2}\gamma\}_{k-2}$ 
is well-defined and
\[
P \{\alpha, \Sigma^k\beta, \Sigma^k\gamma\}_{k}
\subset \{P(\alpha), \Sigma^{k-2}\beta, \Sigma^{k-2}\gamma\}_{k-2}.
\]
\end{lem}
\begin{proof}
By \cite[Proposition 2.5]{IM}, we have 
$P(\alpha)\circ \Sigma^{k-2}\beta= P(\alpha\circ \Sigma^{k}\beta)=0$.
Hence the Toda bracket $\{P(\alpha), \Sigma^{k-2}\beta, \Sigma^{k-2}\gamma\}_{k-2}$ is well-defined.
We denote $Y\cup_{\beta} CZ$ by $C_{\beta}$ and the inclusion map $Y\to C_{\beta}$
 by $i_{\beta}$.
It is well-known that
$C_{\Sigma^{k}\beta} = \Sigma^{k}C_{\beta}$.
By \cite[Proposition 1.7]{T},
any element of $\{\alpha, \Sigma^{k}\beta, \Sigma^{k}\gamma \}_{k}$ is represented as
%denote ? represented as ? represented by ? 
$(-1)^{k}\bar{\alpha}\circ \Sigma^{k}\tilde{\gamma}$, where
$\bar{\alpha}\in[\Sigma^{k}C_{\beta}, \Sigma X\wedge X]$ is an extension of $\alpha$ and
$\tilde{\gamma}\in[\Sigma W, \Sigma^{k}C_{\beta}]$ is a coextension of $\gamma$.
By \cite[Proposition 2.5]{IM}, we have 
$P((-1)^{k}\bar{\alpha}\circ\Sigma^{k}\tilde{\gamma})=(-1)^{k}P(\bar{\alpha})\circ \Sigma^{k-2}\tilde{\gamma}$.
Since $P(\bar{\alpha})\circ \Sigma^{k-2}i_\beta
 = P(\bar{\alpha}\circ\Sigma^{k}i_\beta)=P(\alpha)$, the element
$P(\bar{\alpha})\in [\Sigma^{k-2}C_{\beta}, X]$ is an extension of $P(\alpha)$.
Hence we obtain 
\[
P((-1)^{k}\bar{\alpha}\circ\Sigma^{k}\tilde{\gamma})
 =(-1)^{k-2}P(\bar{\alpha})\circ\Sigma^{k-2}\tilde{\gamma}
 \in \{P(\alpha), \Sigma^{k-2}\beta, \Sigma^{k-2}\gamma\}_{k-2}
\]
 and
$
P \{\alpha, \Sigma^k\beta, \Sigma^k\gamma\}_{k}
\subset \{P(\alpha), \Sigma^{k-2}\beta, \Sigma^{k-2}\gamma\}_{k-2}.
$
This completes the proof.
\end{proof}

Now we show Lemma \ref{HPxl13}.

\begin{lem}\label{HPx13}
\begin{enumerate}
\item
$H(P(\xi_{13}))\equiv\xi'\ \bmod\ 2\lambda', 2\xi'$\ \ and\ \   
$H(P(\lambda))\equiv\lambda' \bmod 2\lambda',2\xi'$. 
\item
$H(P(\xi_{13}\eta_{31}))=\xi'\eta_{29} $\ \ and\ \   
$H(P(\lambda\eta_{31}))=\lambda'\eta_{29}$.
\item
$H(P(\lambda+\xi_{13})\eta_{29})=\nu_{11}\omega_{14}$.
\end{enumerate}
\end{lem}
\begin{proof}
By the relation
$\xi_{12}\in\{\sigma_{12},\nu_{19},\sigma_{22}\}_1$ \eqref{xi12} 
and the fact that 
$\sigma_{12}\circ \Sigma\pi^{18}_{29}=\sigma_{12}\circ \Sigma^3\pi^{16}_{27}$
\cite[Theorem 7.4]{T}, we have
\[
\xi_{12}\in\{\sigma_{12},\nu_{19},\sigma_{22}\}_1=
\{\sigma_{12},\nu_{19},\sigma_{22}\}_3.
\]
Using Lemma \ref{Miy} and \cite[Proposition 2.3]{T}, by the relation
$H(P(\sigma_{13}))=H(P(\iota_{13}))\circ\sigma_{11}=\pm 2\sigma_{11}$
from \cite[Proposition 2.7]{T}, we see that 
%$
%P\xi_{13}\in P\{\sigma_{13},\nu_{20},\sigma_{23}\}_3\subset
%\{P\sigma_{13}, \nu_{18},\sigma_{21}\}_1
%$
%and 
\[
\begin{split}
H(P(\xi_{13}))&\in H(P\{\sigma_{13},\nu_{20},\sigma_{23}\}_3)\subset
H\{P\sigma_{13}, \nu_{18},\sigma_{21}\}_1
\subset
\{\pm 2\sigma_{11}, \nu_{18},\sigma_{21}\}.%\equiv\xi'\ \bmod\ 2\lambda', 2\xi'.
\end{split}
\]
By the relation 
\[
\{\pm 2\sigma_{11}, \nu_{18},\sigma_{21}\}
\subset \{2\sigma_{11}, \pm \nu_{18},\sigma_{21}\}
\supset \{2\sigma_{11}, \nu_{18},\pm \sigma_{21}\}
\supset \pm\{2\sigma_{11}, \nu_{18},\sigma_{21}\}
\]
from \cite[Proposition 1.2]{T} and the above Toda brackets have the same indeterminacy,
we have $\{\pm 2\sigma_{11}, \nu_{18},\sigma_{21}\}=\pm\{2\sigma_{11}, \nu_{18},\sigma_{21}\}$.
Hence, by Lemma \ref{usgnep} (1), for some integers $x$ and $y$, we obtain
\begin{align*}
H(P(\xi_{13}))&= \xi'+ 2x\lambda'+4y\xi'
\shortintertext{or}
H(P(\xi_{13}))&= -(\xi'+ 2x\lambda'+4y\xi')=\xi'+2(-x)\lambda'+2(-2y-1)\xi'.
\end{align*}
This leads to the first half of (1). 

By Lemmas \ref{omg'} (1) and \ref{Miy} and 
$H(P(\sigma_{13}))=\pm 2\sigma_{11}$, we obtain 
\[
\begin{split}
H(P(\lambda))\eta_{29}&=
H(P(\lambda\eta_{31}))=H(P(\Sigma\omega'))
\in H(P\{\sigma_{13},\nu_{20},\varepsilon_{23}\}_3) \\
&\subset H\{P(\sigma_{13}),\nu_{18},\varepsilon_{21}\}_1
\subset \{\pm 2\sigma_{11},\nu_{18},\varepsilon_{21}\}.
\end{split}
\]
Similarly to  the proof of $\{\pm 2\sigma_{11}, \nu_{18},\sigma_{21}\}=\pm\{2\sigma_{11}, \nu_{18},\sigma_{21}\}$,
by $2\varepsilon_{21}=0$ \cite[Theorem 7.1]{T},
we obtain the equation 
$\{\pm 2\sigma_{11},\nu_{18},\varepsilon_{21}\}
=\{2\sigma_{11},\nu_{18},\varepsilon_{21}\}$.
So, by Lemma \ref{usgnep} (2), we have
$H(P(\lambda))\eta_{29}=\lambda'\eta_{29}$.
Hence by relations
$
H(P\lambda)\equiv\pm\lambda'\ \bmod\ 2\lambda',2\xi',\eta_{11}\bar{\mu}_{12}
$
\cite[(3.3)]{MMO} and $\eta_{11}\bar{\mu}_{12}\eta_{29}=\eta^2_{12}\bar{\mu}_{13}=4\bar{\zeta}_{12}\ne 0$ \eqref{4bze5},
we obtain the second half of (1).

By relations (1) and $2\eta_{29}=0$ \eqref{2eta}, we obtain (2).
Moreover, by the equation 
$\nu_{11}\omega_{14}=(\lambda'+\xi')\eta_{29}$ \eqref{n11om},
we obtain
\[
H(P(\lambda+\xi_{13})\eta_{29})
=H(P(\lambda\eta_{31})+H(P(\xi_{13}\eta_{31})=(\lambda'+\xi')\eta_{29}
=\nu_{11}\omega_{14}.
\]
This completes the proof. 
\end{proof}

%%%%%%%%%%%%%%%%%%%%%%%%%%%%%%%%%%%%%%%%%%%%%%%%%%
\section{Proof of Theorem \ref{main}}
%%%%%%%%%%%%%%%%%%%%%%%%%%%%%%%%%%%%%%%%%%%%%%%%%

By relations $\bar{\nu}_6\sigma_{14}=0$ 
and
$\sigma_{11}\varepsilon_{18}=0$ \eqref{bnu6sg},
$\{\bar{\nu}_{n},\sigma_{n+8},\varepsilon_{n+15}\}$
is well-defined for $n\ge 6$.
We show 
\begin{lem}\label{bnusgep}
\begin{enumerate}
\item
$\{\bar{\nu}_{11},\sigma_{19},\varepsilon_{26}\}= x[\iota_{11},\kappa_{11}]$
for some integer $x$.
%%%%%%%%%%%%%%%%%%%%%% (2) は不要？
\item
$\Sigma\{\bar{\nu}_{11},\sigma_{19},\varepsilon_{26}\}=\{\bar{\nu}_{12},\sigma_{20},\varepsilon_{27}\}=0$.
\end{enumerate}
\end{lem}
\begin{proof}
By \cite[Theorem 12.16]{T}, we have 
\[
\mathrm{Ind}\{\bar{\nu}_{11},\sigma_{19},\varepsilon_{26}\}=
\bar{\nu}_{11}\circ \pi^{19}_{35}+\pi^{11}_{27}\circ\varepsilon_{27}=
\bar{\nu}_{11}\circ\{\omega_{19},\sigma_{19}\mu_{26}\}+\{\sigma_{11}\mu_{18}\}\circ\varepsilon_{27}.  
\]
By relations
$\bar{\nu}_{11}\omega_{19}=0$ \eqref{bnu9omg},
$\bar{\nu}_{11}\sigma_{19}=0$ \eqref{bnu6sg}
and
$\sigma_{11}\mu_{18}\varepsilon_{27}=\sigma_{11}\varepsilon_{18}\mu_{26}=0$
from \eqref{bnu6sg}, the indeterminacy is trivial.  
Similarly, we obtain
$
\mathrm{Ind}\{\bar{\nu}_{12},\sigma_{20},\varepsilon_{27}\}=0.
$ 
Therefore, (2) follows directly from (1). 

By using \cite[Propositions 1.2 0), 1.5]{T} and relations
$\{\nu_{11},\eta_{14},\nu_{15}\}=\bar{\nu}_{11}$ \cite[Lemma 6.2]{T},
$\{\eta_{14},\nu_{15},\sigma_{18}\}=0$ (Lemma \ref{etnusg} (2))
and
$\{\nu_{15},\sigma_{18},\varepsilon_{25}\}=0$ (Lemma \ref{n15sgep} (2)),
we have
\[
\begin{split}
0&\in\{\bar{\nu}_{11},\sigma_{19},\varepsilon_{26}\}
+\{\nu_{11},0,\varepsilon_{26}\}
+\{\nu_{11},\eta_{14},0\}\\
&=\{\bar{\nu}_{11},\sigma_{19},\varepsilon_{26}\}
+\nu_{11}\circ\pi^{14}_{35}+\pi^{11}_{27}\circ\varepsilon_{27}
=\{\bar{\nu}_{11},\sigma_{19},\varepsilon_{26}\}+\nu_{11}\circ\pi^{14}_{35}.
\end{split}
\]
By the fact $\pi^{14}_{35}
=\{\eta_{14}\bar{\kappa}_{15},\sigma^3_{14},\Sigma(\lambda\nu_{31})\}$
\cite[Theorem A]{Mi} and
relations $\nu_{11}\eta_{14}=0$ \eqref{n3},
$\nu_{11}\sigma_{14}=0$ \eqref{nu10sg}
and
$\Sigma(\nu_{10}\lambda\nu_{31})=[\iota_{11},\kappa_{11}]$
\eqref{nulmn},
we have 
$\nu_{11}\circ\pi^{14}_{35}
%=\{\Sigma(\nu_{10}\lambda\nu_{31})\}
%=\{\sigma_{11}\nu_{18}\kappa_{21}\}
%=\{[\iota_{11},\iota_{11}]\circ\kappa_{21}\}
=\{[\iota_{11},\kappa_{11}]\}$.
Thus, we obtain (1) and this completes the proof.
\end{proof}

Here we recall the definitions of $\delta_3\in\pi^3_{27}$
and
$\bar{\sigma}'_6$, $\bar{\bar{\sigma}}'_6\in\pi^6_{30}$ \cite[p. 13, 15]{MMO}: 
\begin{equation}\label{def-del}
\delta_3\in\{\varepsilon_3,\varepsilon_{11}+\bar{\nu}_{11},\sigma_{19}\}_1, 
\ %
\bar{\sigma}'_6\in\{\bar{\nu}_6,\varepsilon_{14}+\bar{\nu}_{14},\sigma_{22}\}_1
\mand
\bar{\bar{\sigma}}'_6\in\{\nu_6,\eta_9,\bar{\sigma}_{10}\}_3.
\end{equation}

We show
\begin{lem}\label{3Ind}
\begin{enumerate}
\item
$\mathrm{Ind}\{\varepsilon_3,\varepsilon_{11}+\bar{\nu}_{11},\sigma_{19}\}_1=\{\bar{\mu}_3\sigma_{20}\}$.
\item
$\mathrm{Ind}\{\bar{\nu}_6,\varepsilon_{14}+\bar{\nu}_{14},\sigma_{22}\}_1
=\{P(\xi_{13}\eta_{31}), \bar{\mu}_6\sigma_{23}\}$.
\item
$\bar{\bar{\sigma}}'_6=\{\nu_6,\eta_9,\bar{\sigma}_{10}\}_3$
\ \ and\ \ \ 
$\bar{\bar{\sigma}}'
=\langle\nu,\eta,\bar{\sigma}\rangle$.
\end{enumerate}
\end{lem}
\begin{proof}
By \cite[Theorems 12.16, 12.17]{T},
we have
\[
\mathrm{Ind}\{\varepsilon_3,\varepsilon_{11}+\bar{\nu}_{11},\sigma_{19}\}_1
=\varepsilon_3\circ \Sigma\pi^{10}_{26}+\pi^3_{20}\circ\sigma_{20}
=\{\varepsilon_3\sigma_{11}\mu_{18},
\bar{\varepsilon}'\sigma_{20}, \bar{\mu}_3\sigma_{20},
\eta_3\mu_4\sigma^2_{13}\}.
\]
By relations
$\varepsilon_3\sigma_{11}=0$ \eqref{bnu6sg},
$\bar{\varepsilon}'\sigma_{20}=0$ \eqref{bep's}
and
$\mu_4\sigma^2_{13}=0$ \eqref{mu3s^2},
we have  (1).

By \cite[Theorems 12.7, 12.16]{T},
we have
\[
\begin{split}
\mathrm{Ind}\{\bar{\nu}_6,\varepsilon_{14}+\bar{\nu}_{14},\sigma_{22}\}_1
&= \bar{\nu}_6\circ \Sigma\pi^{13}_{29}+\pi^6_{23}\circ\sigma_{23}\\
&=\{\bar{\nu}_6\sigma_{14}\mu_{21},
P((\Sigma\theta)\sigma_{25}), \nu_6\kappa_9\sigma_{23},
\bar{\mu}_6\sigma_{23}, \eta_6\mu_7\sigma^2_{16}\}.
\end{split}
\]
By relations
$\bar{\nu}_6\sigma_{14}=0$ \eqref{bnu6sg},
$(\Sigma\theta)\sigma_{25}=\xi_{13}\eta_{31}$ \eqref{x13et},
$\kappa_9\sigma_{23}=0$ \eqref{k7s}
and
$\mu_7\sigma^2_{16}=0$ \eqref{mu3s^2},
we have (2).

By \cite[Theorem A]{Mi} and
\cite[Proposition 5.9]{T},
we have 
\begin{gather*}
\mathrm{Ind}\{\nu_6,\eta_9,\bar{\sigma}_{10}\}_3 
= \nu_6\circ \Sigma^3\pi^6_{27}+\pi^6_{11}\circ\bar{\sigma}_{11}
=\{\nu_6\eta_9\bar{\kappa}_{10}, P(\bar{\sigma}_{13})\}
\shortintertext{and}
\mathrm{Ind}\langle\nu,\eta,\bar{\sigma}\rangle
=\{\nu\eta\bar{\kappa},\nu\sigma^3\}.
\end{gather*}
By relations
$\nu_6\eta_9=0$ \eqref{n3},
$P(\bar{\sigma}_{13})=0$ \cite[(5.7)]{MMO}
and $\nu\sigma=0$ \eqref{nu10sg},
we have 
$\mathrm{Ind}\{\nu_6,\eta_9,\bar{\sigma}_{10}\}_3 =0$
and $\mathrm{Ind}\langle\nu,\eta,\bar{\sigma}\rangle=0$.
Hence we obtain (3).
This completes the proof.
\end{proof}

Since $\varepsilon_n+\bar{\nu}_n=\sigma_n\eta_{n+7}$ for $n\ge 10$ \eqref{et9sg}, 
$\varepsilon_3\sigma_{11}=0$ and $\bar{\nu}_6\sigma_{14}=0$ \eqref{bnu6sg}, 
we have 
\[
\{\varepsilon_3,\sigma_{11},\eta_{18}\sigma_{19}\}_1\subset\{\varepsilon_3,\varepsilon_{11}+\bar{\nu}_{11},\sigma_{19}\}_1
\quad\text{and}\quad
\{\bar{\nu}_6,\sigma_{14},\eta_{21}\sigma_{22}\}_1\subset\{\bar{\nu}_6,\varepsilon_{14}+\bar{\nu}_{14},\sigma_{22}\}_1.
\]
By \cite[Theorem 12.6 and 12.16]{T}, we have
\begin{align*}
\mathrm{Ind}\{\varepsilon_3,\sigma_{11},\eta_{18}\sigma_{19}\}_1
&= \varepsilon_3\circ\pi^{11}_{27}+\pi^3_{19}\circ\eta_{19}\sigma_{20}
=\{\varepsilon_3\sigma_{11}\mu_{18},\mu_3\sigma_{12}\eta_{19}\sigma_{20},
\eta_3\bar{\varepsilon}_4\eta_{19}\sigma_{20}\}
\shortintertext{and}
\mathrm{Ind}\{\bar{\nu}_6,\sigma_{14},\eta_{21}\sigma_{22}\}_1
&= \bar{\nu}_6\circ \Sigma\pi^{13}_{29}+\pi^6_{22}\circ\eta_{22}\sigma_{23}
=\{\bar{\nu}_6\sigma_{14}\mu_{21},\zeta'\eta_{22}\sigma_{23}\}.
\end{align*}
By relations 
$\varepsilon_3\sigma_{11}\mu_{18}=0$ \eqref{bnu6sg},
$\mu_3\sigma_{12}\eta_{19}\sigma_{20}=\mu_3\sigma^2_{12}\eta_{26}=0$ \eqref{mu3s^2},
$\eta_3\bar{\varepsilon}_4\eta_{19}\sigma_{20}=0$ \eqref{bep3etas},
$\bar{\nu}_6\sigma_{14}\nu_{21}=0$ \eqref{bnu6sg}
and
$\zeta'\eta_{22}\sigma_{23}=0$ \eqref{z'etas},
we have 
\[
\mathrm{Ind}\{\varepsilon_3,\sigma_{11},\eta_{18}\sigma_{19}\}_1=0
\mand
\mathrm{Ind}\{\bar{\nu}_6,\sigma_{14},\eta_{21}\sigma_{22}\}_1=0.
\]
%We change the definitions of $\delta_3$ and $\bar{\sigma}'_6$ as follows:
We define $\delta'_3\in\pi^3_{27}$ and $\bar{\sigma}''_6\in\pi^6_{30}$:
\begin{gather}\label{chd}
\delta'_3=\{\varepsilon_3,\sigma_{11},\eta_{18}\sigma_{19}\}_1
\subset\{\varepsilon_3,\varepsilon_{11}+\bar{\nu}_{11},\sigma_{19}\}_1, \\
\label{ch}
\bar{\sigma}''_6=\{\bar{\nu}_6,\sigma_{14},\eta_{21}\sigma_{22}\}_1
\subset\{\bar{\nu}_6,\varepsilon_{14}+\bar{\nu}_{14},\sigma_{22}\}_1.
\end{gather}
By the facts
$\delta'_3\in\{\varepsilon_3,\varepsilon_{11}+\bar{\nu}_{11},\sigma_{19}\}_1$
and
$\bar{\sigma}''_6\in\{\bar{\nu}_6,\varepsilon_{14}+\bar{\nu}_{14},\sigma_{22}\}_1$, the following relations, 
analogous to those for $\delta_3$ and $\bar{\sigma}'_6$
\cite[Theorem 1.1 (b), (3.5), (3.8)]{MMO}, hold:
\begin{equation}\label{d'bs''}
H(\delta'_3)\equiv \nu_5\bar{\sigma}_8 \bmod\ \nu_5\bar{\zeta}_8, \
\sharp\delta'_3=2,\ 
H(\bar{\sigma}''_6) \equiv \bar{\sigma}_{11} \bmod\ \xi'\eta_{29}
\mand
\sharp\bar{\sigma}''_6=2.
\end{equation}
The following relations, which are refinements of the above relations, hold.
%It is easy to check that this changing gives no influences  
%to the computations in \cite{MMO}. 
\begin{prop}\label{hbsgm''}
\begin{enumerate}
\item
$H(\delta'_3)=\nu_5\bar{\sigma}_8$.
\item
$H(\bar{\sigma}''_6)=\bar{\sigma}_{11}+\xi'\eta_{29}$.
\end{enumerate}
\end{prop}
\begin{proof}
Using \cite[Propositions 2.3, 1.2 iv)]{T},
by relations \eqref{chd}, 
$H(\varepsilon_3)=\nu^2_5$ \cite[Lemma 6.1]{T} 
and \cite[Theorems 12.26, 10.3]{T},
we have
\begin{gather*}
\begin{split}
H(\delta'_3)&=H\{\varepsilon_3,\sigma_{11},\eta_{18}\sigma_{19}\}_1
\subset\{H(\varepsilon_3),\sigma_{11},\eta_{18}\sigma_{19}\}_1
=\{\nu^2_5,\sigma_{11},\eta_{18}\sigma_{19}\}_1\\
&\supset \nu_5\circ\{\nu_8,\sigma_{11},\eta_{18}\sigma_{19}\}_1
\ni \nu_5\bar{\sigma}_8
\end{split}
\shortintertext{with the indeterminacy}
\mathrm{Ind}\{\nu^2_5,\sigma_{11},\eta_{18}\sigma_{19}\}_1
=\nu^2_5\circ\Sigma\pi^{10}_{26}+\pi^5_{19}\circ\eta_{19}\sigma_{20}
=\{\nu^2_5\sigma_{11}\mu_{18},
\nu_5\zeta_8\eta_{19}\sigma_{20},\nu_5\bar{\nu}_8\nu_{16}\eta_{19}\sigma_{20}\}.
\end{gather*}
By relations $\nu_8\sigma_{11}\mu_{18}=0$ \cite[(2.4)]{MMO},
$\zeta_8\eta_{19}\sigma_{20}=0$ \eqref{ze7etsg}
and $\nu_{16}\eta_{19}=0$ \eqref{n3},
this indeterminacy is trivial.
So, we obtain (1).

We recall  the relation
$H(\bar{\sigma}''_6) \equiv \bar{\sigma}_{11} 
\bmod\ \xi'\eta_{29}$ \eqref{d'bs''}.
Applying $\Sigma^\infty$, 
by the relation  $\Sigma^\infty(\zeta'\eta_{29})=0$ \eqref{Elameta},
we have 
\[
\Sigma^\infty(H(\bar{\sigma}''_6))=\bar{\sigma}.
\]
On the other hand, by relations  \eqref{ch} and 
$H(\bar{\nu}_6)=(2k+1)\nu_{11}$ for some integer $k$
\cite[Lemma 6.2]{T},
we have
\begin{gather*}
\begin{split}
H(\bar{\sigma}''_6)&=H\{\bar{\nu}_6,\sigma_{14},\eta_{21}\sigma_{22}\}_1
\subset\{H(\bar{\nu}_6),\sigma_{14},\eta_{21}\sigma_{22}\}_1
=\{(2k+1)\nu_{11},\sigma_{14},\eta_{21}\sigma_{22}\}_1\\
&\subset (2k+1)\iota_{11}\circ\{\nu_{11},\sigma_{14},\eta_{21}\sigma_{22}\}_1
\end{split}
\end{gather*}
By using \cite[Proposition 2.6]{T} and the facts
$P(\eta_{21})=\nu_{10}\sigma_{13}$ \eqref{nu10sg} and
$\pi^{21}_{22}=\{\eta_{21}\}\cong\Z_2$ \cite[Proposition 5.1]{T}, we have
\[H\{\nu_{11},\sigma_{14},\eta_{21}\sigma_{22}\}_1
 =-P^{-1}(\nu_{10}\sigma_{13})\circ\eta_{22}\sigma_{23}=\eta^2_{21}\sigma_{23}.
\]
Hence, we have
\[
H(H(\bar{\sigma}''_6))=\eta^2_{21}\sigma_{23}.
\]
By using the EHP-sequence and the facts
$H(\xi')=\eta_{21}\sigma_{22}$ \cite[Lemma  12.19]{T} 
and $\Sigma\pi^{10}_{29}=\{\bar{\sigma}_{11},\bar{\zeta}_{11}\}\cong\Z_2\oplus\Z_8$ 
\cite[Theorem 12.23]{T}, we can put
\[
H(\bar{\sigma}''_6)=\xi'\eta_{29}\ +
a\bar{\sigma}_{11}+ b\bar{\zeta}_{11}\quad\text{for some}\ a\in\{0,1\}
\mand b\in\{0,1,2,\dots,7\}.
\]
Applying $\Sigma^\infty$, we have
\[
\Sigma^\infty(H(\bar{\sigma}''_6))= a \bar{\sigma}+ b\bar{\zeta}.
\]
Hence, by the relation $\Sigma^\infty(H(\bar{\sigma}''_6))=\bar{\sigma}$ and 
the fact $\pi^S_{19}=\{\bar{\sigma},\zeta\}\cong\Z_2\oplus\Z_8$ \cite[Theorem 12.23]{T}, 
we have $a=1$ and $b=0$.
This leads to (2) and completes the proof. 
\end{proof}

We set 
\[
\delta'_n=\Sigma^{n-3}\delta'_3\ (n\ge 3),\ \delta'=\Sigma^\infty\delta'_3
\qquad
\text{and}
\qquad
\bar{\sigma}''_n=\Sigma^{n-6}\bar{\sigma}''_6\ (n\ge 6),\ \bar{\sigma}''=\Sigma^\infty\bar{\sigma}''_6.
\]

%So we have 
%$$ \bar{\bar{\sigma}}'_6\equiv\bar{\sigma}'_6\ \bmod\ E\pi^5_{29}=\{\delta_6,\bar{\mu}_6\sigma_{23},\nu_6\sigma_9\kappa_{16}\}$$and hence, by the relation $\nu_{11}\sigma_{14}=0$ (\ref{nu10sg}) and %$\pi^S_{24}=\{\delta,\bar{\mu}\sigma\}\cong(\Z_2)^2$, we see that 
%\begin{equation}\label{sgm"} \bar{\bar{\sigma}}'_{11}=\bar{\sigma}'_{11}. \end{equation}

Now we show the following lemmas.
\begin{lem}\label{bnu12}
$\bar{\sigma}''_{12}=\{\bar{\nu}_{12},\sigma_{20},\bar{\nu}_{27}\}=\{\nu_{12},\eta_{15},\bar{\sigma}_{16}\}=\bar{\bar{\sigma}}'_{12}$.
\end{lem}
\begin{proof}
%Notice that the indeterminacies of the brackets
We will show $\bar{\sigma}''_{12} = \Sigma\{\bar{\nu}_{11},\sigma_{19},\bar{\nu}_{26}\}
=\{\bar{\nu}_{12},\sigma_{20},\bar{\nu}_{27}\}$.
By $\bar{\nu}_{26}=\varepsilon_{26}+\eta_{26}\sigma_{27}$ \eqref{et9sg}
and \cite[Proposition 1.6]{T}, we have
\[
\{\bar{\nu}_{11},\sigma_{19},\bar{\nu}_{26}\}
\subset\{\bar{\nu}_{11},\sigma_{19},\varepsilon_{26}\}
 + \{\bar{\nu}_{11},\sigma_{19},\eta_{26}\sigma_{27}\}.
\]
For $n=11$ and $12$, by \cite[Theorem 12.16]{T}, we have
\[
\mathrm{Ind}\{\bar{\nu}_{n},\sigma_{n+8},\bar{\nu}_{n+15}\}
=\bar{\nu}_{n}\circ\pi^{n+8}_{n+24}+\pi^{n}_{n+16}\circ\bar{\nu}_{n+16}
=\{\bar{\nu}_{n}\omega_{n+8},
\bar{\nu}_{n}\sigma_{n+8}\mu_{n+15},
\sigma_{n}\mu_{n+7}\bar{\nu}_{n+16}\}.
\]
By relations $\bar{\nu}_{n}\omega_{n+8}=0$ \eqref{bnu9omg},
$\bar{\nu}_{n}\sigma_{n+8}=0$ \eqref{bnu6sg}
and
$\mu_{n+7}\bar{\nu}_{n+16}=0$ \eqref{mepbn2},
we have 
$\mathrm{Ind}\{\bar{\nu}_{n},\sigma_{n+8},\bar{\nu}_{n+15}\}=0$.
By $\sigma_{11}\mu_{18}\eta_{27}\sigma_{28}
=\eta_{11}\mu_{12}\sigma^2_{21}=0$ \eqref{mu3s^2},
we also have 
\[
\mathrm{Ind}\{\bar{\nu}_{11},\sigma_{19},\eta_{26}\sigma_{27}\}
=\bar{\nu}_{11}\circ\pi^{19}_{35}+\pi^{11}_{27}\circ\eta_{27}\sigma_{28}
=\{\sigma_{11}\mu_{18}\eta_{27}\sigma_{28}\}=0.
\]
Using \cite[Proposition 1.3]{T}, by
$\bar{\sigma}''_{6}
\in\{\bar{\nu}_{6},\sigma_{14},\eta_{21}\sigma_{22}\}$
\eqref{ch}
and $\sharp\bar{\sigma}''_{6}=2$ \eqref{d'bs''}, we have
$\bar{\sigma}''_{11}
=\{\bar{\nu}_{11},\sigma_{19},\eta_{26}\sigma_{27}\}$.
So, by Lemma \ref{bnusgep} (1), for some $x\in\{0,1\}$, we have
\[
\{\bar{\nu}_{11},\sigma_{19},\bar{\nu}_{26}\}
=x[\iota_{11},\kappa_{11}]+\bar{\sigma}''_{11}
\]
and
$\bar{\sigma}''_{12} 
=\Sigma\{\bar{\nu}_{11},\sigma_{19},\bar{\nu}_{26}\}
=\{\bar{\nu}_{12},\sigma_{20},\bar{\nu}_{27}\}$.

By 
$\bar{\bar{\sigma}}'_6=\{\nu_6,\eta_9,\bar{\sigma}_{10}\}_3$
(Lemma \ref{3Ind} (3)) and 
$2\bar{\bar{\sigma}}'_6=0$ \cite[(3.8)]{MMO},
we have
$\bar{\bar{\sigma}}'_{12}\in\{\nu_{12},\eta_{15},\bar{\sigma}_{16}\}$.
By \cite[Proposition 5.9]{T} and \cite[Theorem A]{Mi}, we have
\[
\mathrm{Ind}\{\nu_{12},\eta_{15},\bar{\sigma}_{16}\}
=\nu_{12}\circ\pi^{15}_{36}+\pi^{12}_{17}\circ\bar{\sigma}_{17}
=\{\nu_{12}\eta_{15}\bar{\kappa}_{16},
\nu_{12}\sigma^3_{15},\Sigma^2(\nu_{10}\lambda\nu_{31})\}.
\]
So, by relations $\nu_{12}\eta_{15}=0$ \eqref{n3},
$\nu_{12}\sigma_{15}=0$ \eqref{nu10sg}
and 
$\Sigma^2(\nu_{10}\lambda\nu_{31})=0$ \eqref{nulmn},we have
$\mathrm{Ind}\{\nu_{12},\eta_{15},\bar{\sigma}_{16}\}=0$
and
$\bar{\bar{\sigma}}'_{12}=\{\nu_{12},\eta_{15},\bar{\sigma}_{16}\}$.

By using \cite[Proposition 1.5]{T} and the facts
$\{\nu_{12},\eta_{15},\nu_{16}\}=\bar{\nu}_{12}$ \cite[Lemma 6.2]{T},
$\{\eta_{15},\nu_{16},\sigma_{19}\}=0$ (Lemma \ref{etnusg} (2))
and
$\{\nu_{16},\sigma_{19},\bar{\nu}_{26}\}=\bar{\sigma}_{16}$ \eqref{bs15},
we have
\[
0\in\{\bar{\nu}_{12},\sigma_{20},\bar{\nu}_{27}\}
+\{\nu_{12},0,\bar{\nu}_{27}\}
+\{\nu_{12},\eta_{15},\bar{\sigma}_{16}\}
=\{\bar{\nu}_{12},\sigma_{20},\bar{\nu}_{27}\}
+\{\nu_{12},\eta_{15},\bar{\sigma}_{16}\}.
\]
Hence, by equations
$\bar{\sigma}''_{12}
=\{\bar{\nu}_{12},\sigma_{20},\bar{\nu}_{27}\}$
and
$\bar{\bar{\sigma}}'_{12}
=\{\nu_{12},\eta_{15},\bar{\sigma}_{16}\}$, 
we obtain $0=\bar{\sigma}''_{12}+\bar{\bar{\sigma}}'_{12}$
and $\bar{\sigma}''_{12}=\bar{\bar{\sigma}}'_{12}$.
This completes the proof.
\end{proof}
\begin{lem}\label{dl'dl}
\begin{enumerate}
\item
$\delta'_3\equiv\delta_3\ \bmod\ \bar{\mu}_3\sigma_{20}$.
\item 
$\bar{\sigma}''_6\equiv\bar{\sigma}'_6\ \bmod\ 
P(\xi_{13})\eta_{29},\ \bar{\mu}_6\sigma_{23}$.
\item  
$\bar{\sigma}''_6\equiv\bar{\bar{\sigma}}'_6+P(\xi_{13})\eta_{29}\ \bmod\ \nu_6\sigma_9\kappa_{16}$. 
\end{enumerate}
\end{lem}
\begin{proof}
By the definitions of $\delta_3$, $\bar{\sigma}'_6$ \eqref{def-del},
$\delta'_3$ \eqref{chd} and $\bar{\sigma}''_6$ \eqref{ch}, 
and Lemma \ref{3Ind} (1), (2), 
we obtain (1) and (2). 

By the relations  
$H(\bar{\sigma}''_6)=\bar{\sigma}_{11}+\xi'\eta_{29}$
(Proposition \ref{hbsgm''} (2)),
$H(\bar{\bar{\sigma}}'_6)=\bar{\sigma}_{11}$ \cite[(3.8)]{MMO}
and 
$H(P(\xi_{13}\eta_{31}))=\xi'\eta_{29} $
(Lemma \ref{HPx13} (2)), we have
\[
H(\bar{\sigma}''_6)
=H(\bar{\bar{\sigma}}'_6)+H(P(\xi_{13}\eta_{31}))
=H(\bar{\bar{\sigma}}'_6+P(\xi_{13})\eta_{29}).
\]
So, by the EHP-sequence, we have
\[
\bar{\sigma}''_6\equiv\bar{\bar{\sigma}}'_6+P(\xi_{13})\eta_{29}\ \bmod\ \Sigma\pi^5_{29}.
\]

Hence, by the fact that
$\Sigma\pi^5_{29}=\{\nu_6\sigma_9\kappa_{16},
\bar{\mu}_6\sigma_{23},\delta_6\}$
and $\pi^S_{24}=\{\delta,\bar{\mu}\sigma\}$ \cite[p.~31, Theorem 1.1(b)]{MMO},
and relations
$\bar{\sigma}''=\bar{\bar{\sigma}}'$ (Lemma \ref{bnu12})
and $\nu\sigma=0$ \eqref{nu10sg}, we obtain (3).
This completes the proof. 
\end{proof}

\begin{rem}
By Lemma \ref{dl'dl} (2), we have $\bar{\sigma}''_7\equiv\bar{\sigma}'_7\ \bmod\ 
\bar{\mu}_6\sigma_{23}$.
By Lemma \ref{dl'dl} (3)
and the equation
$P(\nu_{17}\kappa_{20})=\nu_8\sigma_{11}\kappa_{18}$ \cite[(5.11)]{MMO},
we also have $\bar{\sigma}''_9=\bar{\bar{\sigma}}'_9$.
Hence, the following relation holds:
\[
\bar{\sigma}'_9 \equiv \bar{\bar{\sigma}}'_9
\bmod\  \bar{\mu}_9\sigma_{26}.
\]
\end{rem}

We set $\pi^S_k(M^2)=\lim\limits_{n\to\infty}\pi_{n+k}(M^{n+2})$,
$i''=\Sigma^\infty i''_2$, $p''=\Sigma^\infty p''_2$.
$\ext({\eta})=\Sigma^\infty\ext({\eta}_3)\in\lim\limits_{n\to\infty}[M^{n+2},S^n]$ 
and $\coe({\eta})=\Sigma^\infty\coe({\eta}_3)\in\pi^S_3(M^2)$.
%By the facts from \cite[Theorems 3.1, 3.2]{Mu66} that:
%\[
%\lim\limits_{n\to\infty}[M^{n+2},S^n]=\{\ext({\eta})\}\cong\Z_4
%\mand \pi^S_3(M^2)=\{\coe({\eta})\}\cong\Z_4,
%\]
%the  extension of $\eta$ are only $\pm\ext({\eta})$ and 
%the coextension of $\eta$ are only $\pm\coe({\eta})$.
%So, by \cite[Proposition 1.7]{T}, 
%$\Sigma^\infty\nu'\in \langle\eta,2\iota,\eta\rangle$ \cite[(5.4)]{T} 
%and 
%$2\nu_5=\Sigma^2\nu'$ \cite[(5.5)]{T},
%we have
% \begin{equation}\label{2nu}
%2\nu=\pm\ext({\eta})\coe({\eta}).
%\end{equation}

Let $\coe({\sigma^2})\in\pi^S_{16}(M^2)$ be a coextension of $\sigma^2\in\pi^S_{14}$.
By using the exact sequence  induced from  the cofibration 
$S^1 \rarrow{2\iota_1} S^1 \rarrow{i''_2} M^2$:
\[
\pi^S_{22}\rarrow{2\iota_*}\pi^S_{22}
\rarrow{i''_*}\pi^S_{23}(M^2)\rarrow{p''_*}
\pi^S_{21}\rarrow{2\iota_*}\pi^S_{21},
\]
the facts
$\pi^S_{21}=\{\eta\bar{\kappa},\sigma^3\}=(\Z_2)^2$ and
$\pi^S_{22}=\{\varepsilon\kappa,\nu\bar{\sigma}\}\cong(\Z_2)^2$
\cite[Theorems A, B]{Mi}
and relations $2\coe(\eta)=i''\eta^2$, 
$\sharp\coe({\sigma^2})=2$ \cite[Theorem 3.2]{Mu66}
and $\varepsilon\kappa=\eta^2\bar{\kappa}$ \cite[I-Proposition 3.1 (2)]{T},
we have
\begin{equation}\label{piS23M}
\pi^S_{23}(M^2)=\{\coe({\eta})\bar{\kappa},
\coe({\sigma^2})\sigma, i''\nu\bar{\sigma}\}
 \cong\Z_4\oplus(\Z_2)^2
\mand 2\coe({\eta})\bar{\kappa}=i''\varepsilon\kappa.
\end{equation}

We show 
\begin{lem}\label{dl''}
\begin{enumerate}
\item
$\delta\equiv\eta\eta^*\sigma\ \bmod\ \bar{\mu}\sigma$.
\item
$\delta'=\eta\eta^*\sigma$.  
\end{enumerate}
\end{lem}
\begin{proof}
By relations
\begin{gather*}
\phi_9\equiv \sigma_9\eta^*_{16}
\bmod\ \sigma^2_9\mu_{23},\ 4\nu_{9}\bar{\kappa}_{12}
\quad\text{\cite[I-Propositions 3.4 (7)]{Od}}
\shortintertext{and}
\phi_5\eta_{28}\equiv \delta_5
\bmod\  \bar{\mu}_5\sigma_{22},\ \nu_5\eta_8\bar{\kappa}_9
\quad\text{\cite[I-Propositions 3.5 (9)]{Od}},
\end{gather*}
we have 
\[
\eta_8\sigma_9\eta^*_{16}\equiv \delta_8
\bmod\ \eta_8\sigma^2_9\mu_{23},\ 
\bar{\mu}_8\sigma_{25},\ \nu_8\eta_{11}\bar{\kappa}_{12}.
\]
Hence, by relations 
$\eta_9\sigma^2_{10}\mu_{24}=0$ \eqref{et9sg2}
and 
$\nu_9\eta_{12}=0$ \eqref{n3}, we have
\begin{equation}\label{etsget*}
\eta_9\sigma_{10}\eta^*_{17}\equiv\delta_9\ \bmod\ \bar{\mu}_9\sigma_{26}
\end{equation}
and (1).

By using \cite[(3.8)]{T} and relations 
$\delta'\in\langle\varepsilon,\sigma,\eta\sigma\rangle$ \eqref{chd},
$\eta\sigma=\bar{\nu}+\varepsilon$ \eqref{et9sg}
and
$\langle\varepsilon,\sigma,\bar{\nu}\rangle
=\{\varepsilon\omega\}$
\cite[(6.1)]{MMO}, we have 
\[
\delta'\in\langle\varepsilon,\sigma,\eta\sigma\rangle
\subset\langle\varepsilon,\sigma,\bar{\nu}\rangle
+\langle\varepsilon,\sigma,\varepsilon\rangle
=\{\varepsilon\omega\}+\langle\varepsilon,\sigma,\varepsilon\rangle.
\]
Since $\mathrm{Ind}\langle\varepsilon,\sigma,\varepsilon\rangle
=\varepsilon\circ\pi^S_{16}=\{\varepsilon\omega\}$
 by \cite[Theorem 12.16]{T} and 
$\varepsilon\sigma=0$ \eqref{bnu6sg},
we have
\[
\langle\varepsilon,\sigma,\varepsilon\rangle
=\delta'+\{\varepsilon\omega\}.
\]

By using \cite[(3.7)]{T} and the facts 
$\langle\eta\ext({\eta}),\coe({\eta}),\nu\rangle=\varepsilon$ \eqref{ep_GM},
$\langle\coe({\eta}),\nu,\sigma\rangle\subset\pi^S_{14}(M^2)=0$
\cite[Theorem 3.2]{Mu66},
$\langle\nu,\sigma,\varepsilon\rangle=0$ (Lemma \ref{sgnep} (2))
and \eqref{piS23M},
we have
\[\begin{split}
0&\in\langle\varepsilon,\sigma,\varepsilon\rangle
+\langle\eta\ext({\eta}),0,\varepsilon\rangle
+\langle\eta\ext({\eta}),\coe({\eta}),0\rangle
=\langle\varepsilon,\sigma,\varepsilon\rangle
+\eta\ext({\eta})\circ\pi^S_{23}(M^2)\\
&=\delta'+\{\varepsilon\omega\}
+\eta\ext({\eta})\circ
\{\coe({\eta})\bar{\kappa},
\coe({\sigma^2})\sigma, i''\nu\bar{\sigma}\}.
\end{split}
\]
By relations $2\nu=\pm\ext({\eta})\coe({\eta})$ \eqref{2nu_n}
and $\eta\nu=0$ \eqref{n2}, we have
\[
\eta\ext({\eta})\circ\coe({\eta})\bar{\kappa}
=\eta(\pm2\nu)\bar{\kappa}=0
\mand
\eta\ext({\eta})\circ i''\nu\bar{\sigma}
=\eta^2\nu\bar{\sigma}=0.
\]
By \cite[Proposition 1.7, (3.5) ii)]{T},
$\omega\in\langle\eta,2\sigma,\sigma\rangle$
\cite[I-Proposition 3.3 (6)]{Od}
and
$\eta\sigma\omega=\varepsilon\omega$
\cite[I-Proposition 3.3 (5)]{Od}, we have
\[
\eta\ext({\eta})\circ\coe({\sigma^2})\sigma
\in\eta\circ\langle\eta,2\iota,\sigma^2\rangle\circ\sigma
\subset\eta\circ\langle\eta,2\sigma,\sigma\rangle\circ\sigma
\ni\eta\omega\sigma=\varepsilon\omega
\]
with the indeterminacy
\[
\mathrm{Ind}(\eta\circ\langle\eta,2\sigma,\sigma\rangle\circ\sigma)
=\eta^2\circ\pi^S_{15}\circ\sigma+\eta\circ\pi^S_{9}\circ\sigma^2
=\eta^2\sigma\circ\pi^S_{15}+\eta\sigma^2\circ\pi^S_{9}.
\]
By relations $\sigma\zeta=0$ \eqref{n7} and 
$\eta^2\sigma=0$ \eqref{et9sg2},
this indeterminacy is trivial and 
$\eta\ext({\eta})\circ
\{\coe({\eta})\bar{\kappa},
\coe({\sigma^2})\sigma, i''\nu\bar{\sigma}\}=\{\varepsilon\omega\}$.
Hence, we have $0\in \delta'+\{\varepsilon\omega\}$.
So, by relations
$\varepsilon\omega=\eta\phi$ \cite[I-Proposition 3.3 (5)]{Od}
and 
$\eta\phi= \eta\eta^*\sigma$ \eqref{eta9phi},
we have
\[
\delta'\equiv 0\ \bmod\ \eta\eta^*\sigma.
\]
Therefore, by the fact 
$\pi^S_{26}=\{\bar{\mu}\sigma,\eta\eta^*\sigma\}$
\cite[Theorem 1.2]{MMO},
(2) follows from (1) and  
$\delta'\equiv \delta \bmod\ \bar{\mu}\sigma$
(Lemma \ref{dl'dl} (1)). This completes the proof. 
\end{proof}

Now we show 
\begin{thm}\label{bsg'dt}
$\delta'_9=\eta_9\sigma_{10}\eta^*_{17}$ \ and \ 
$\bar{\sigma}''_{19}+\delta'_{19}=[\iota_{19},\nu^2_{19}]$.
\end{thm}
\begin{proof}
By 
$\delta'_9\equiv\delta_9\ \bmod\ \bar{\mu}_9\sigma_{26}$
(Lemma \ref{dl'dl} (1)) 
and 
$\eta_9\sigma_{10}\eta^*_{17}\equiv\delta_9\ 
\bmod\ \bar{\mu}_9\sigma_{26}$ \eqref{etsget*}, 
we have
\[
\delta'_9\equiv\eta_9\sigma_{10}\eta^*_{17}\ \bmod\ \bar{\mu}_9\sigma_{16}.
\]
So, relations
$\delta'=\eta\sigma\eta^*$
(Lemma \ref{dl''}) 
and $\bar{\mu}\sigma\neq 0$ \cite[Theorem 1.1 (b)]{MMO}
lead to $\delta'_9=\eta_9\sigma_{10}\eta^*_{17}$.

Recall from \cite[II-Proposition 2.1 (10)]{Od} the relation:
\[
\eta_9\psi_{10}\equiv\ \bar{\sigma}'_9+\delta_9\ \bmod\ \bar{\mu}_9\sigma_{26}, \sigma^2_9\eta_{23}\mu_{24}.
\]
By Lemma \ref{dl'dl} (1) and (2), we obtain 
\[
\eta_9\psi_{10}\equiv\ \bar{\sigma}''_9+\delta'_9\ \bmod\ \bar{\mu}_9\sigma_{26}, \sigma^2_9\eta_{23}\mu_{24}.
\]
By relations $\psi_{21}=[\iota_{21},\nu_{20}]$ (Lemma \ref{wnu21}) 
and $\eta_{20}\nu_{21}=0$ \eqref{n2},
we have
$\eta_{20}\psi_{21}=[\eta_{20},\eta_{20}\nu_{21}]=0$.
So, by $\sigma^2_{20}\eta_{34}=\eta_{20}\sigma^2_{21}=0$ 
\eqref{et9sg2},
we obtain $\bar{\sigma}''_{20}\equiv \delta'_{20}\ \bmod\ \bar{\mu}_{20}\sigma_{37}$
and
\begin{equation*}%\label{bsg'20} 
\bar{\sigma}''_{20}= \delta'_{20}\ + a \bar{\mu}_{20}\sigma_{37}
\text{\ \ for some\ } a\in\{0,1\}.
\end{equation*}
By the relation
$\delta'_{20}=\eta_{20}\eta^*_{21}\sigma_{37}$
from
$\delta'=\eta\eta^*\sigma$ (Lemma \ref{dl''})
and 
$\pi^{20}_{44}\cong\pi^S_{24}$ \cite[Theorem 1.1 (b)]{MMO},
we have 
%\begin{equation}\label{det20}
%\delta'_{20}=\eta_{20}\eta^*_{21}\sigma_{37}
%\mand
%\eta_{20}\eta^*_{21}\sigma_{37}=\eta_{20}\sigma_{21}\eta^*_{28}.
%\end{equation}
%From (\ref{bsg'20}) and (\ref{det20}), we have
\[
\bar{\sigma}''_{20}=\eta_{20}\eta^*_{21}\sigma_{37}+a\bar{\mu}_{20}\sigma_{37}
\mand
\bar{\sigma}''=\eta\eta^*\sigma+a\bar{\mu}\sigma.
\]
By using \cite[(3.6)]{T} and relations 
$\nu^2\bar{\sigma}=0$ (Lemma \ref{nu^2bs} (3)),
$\nu^2=\langle\nu,\eta,\nu\rangle$ \cite[Lemma 5.12]{T},
$\langle\nu,\eta,\bar{\sigma}\rangle=\bar{\sigma}''$ (Lemma \ref{bnu12}),
$\eta^2\eta^*=4\nu^*$ \cite[Theorem 14.1 i)]{T},
$\nu^*\sigma=0$ \cite[I-Proposition 3.5 (3)]{Od},
and $\eta\bar{\mu}\sigma \neq 0$ \cite[Theorem 1 (a)]{Od},
we have
\[
\begin{split}
0&=\nu^2\bar{\sigma}=\langle\eta,\nu,\eta\rangle\circ\bar{\sigma}=\eta\circ\langle\nu,\eta,\bar{\sigma}\rangle=\eta\bar{\sigma}''=\eta(\eta\eta^*\sigma+a\bar{\mu}\sigma)=4\nu^*\sigma+a\eta\bar{\mu}\sigma=a\eta\bar{\mu}\sigma
\end{split}
\]
and $a=0$.  Hence we obtain $\bar{\sigma}''_{20}= \delta'_{20}$.
So, by the fact 
$\pi^{39}_{45}=\{\nu^2_{39}\}$ \cite[Proposition 5.11]{T}, 
we have
$\bar{\sigma}''_{19}+ \delta'_{19}\in P\pi^{39}_{45}=\{[\iota_{19},\nu^2_{19}]\}$.
Since $\bar{\sigma}''_{19}+ \delta'_{19}\equiv \bar{\sigma}'_{19}+\delta_{19} \bmod \bar{\mu}_{19}\sigma_{36}$
from Lemma \ref{dl'dl} (1), (2) and 
$\pi^{19}_{43}=\{\delta_{19},\bar{\mu}_{19}\sigma_{36}, 
\bar{\sigma}'_{19}\}=(\Z_2)^3$ \cite[Theorem 1.1 (b)]{MMO},
we obtain $\bar{\sigma}''_{19}+ \delta'_{19}\neq 0$. This implies 
$\bar{\sigma}''_{19}+\delta'_{19}=[\iota_{19},\nu^2_{19}]$
 and completes the proof.
%The third one follows from \cite[(5.38-40)]{MMO} and Lemma \ref{dl'dl}.  
\end{proof}

For $n\ge 20$, by relations 
 $\bar{\sigma}''_{n}=\delta'_n=\eta_n\sigma_{n+1}\eta^*_{n+8}$,
 (Theorem \ref{bsg'dt}),
$\bar{\sigma}''_{n}\in\{\bar{\nu}_{n},\sigma_{n+8},\bar{\nu}_{n+15}\}$
and
$\bar{\sigma}''_{n}\in\{\nu_{n},\eta_{n+3},\bar{\sigma}_{n+4}\}$
(Lemma \ref{bnu12}), we have
\[
\{\bar{\nu}_{n},\sigma_{n+8},\bar{\nu}_{n+15}\}
\ni \eta_n\sigma_{n+1}\eta^*_{n+8}
\mand
\{\nu_{n},\eta_{n+3},\bar{\sigma}_{n+4}\}
\ni\eta_n\sigma_{n+1}\eta^*_{n+8}.
\]
By \cite[Theorem 12.16]{T}, \cite[Theorem B]{Mi}
and \cite[Proposition 5.9]{T},
we have
\begin{align*}
&\mathrm{Ind}\{\bar{\nu}_{n},\sigma_{n+8},\bar{\nu}_{n+15}\}
=\bar{\nu}_{n}\circ\pi^{n+8}_{n+24}+\pi^{n}_{n+16}\circ\bar{\nu}_{n+16}
=\{\bar{\nu}_{n}\omega_{n+8}, \bar{\nu}_{n}\sigma_{n+8}\mu_{n+15}\}\\
\shortintertext{and}
&\mathrm{Ind}\{\nu_{n},\eta_{n+3},\bar{\sigma}_{n+4}\}
=\nu_{n}\circ\pi^{n+3}_{n+24}+\pi^{n}_{n+5}\circ\bar{\sigma}_{n+5}
=\{\nu_{n}\eta_{n+3}\bar{\kappa}_{n+4},\nu_{n}\sigma^3_{n+3}\}.
\end{align*}
So, by relations
$\bar{\nu}_{n}\omega_{n+8}=0$ \eqref{bnu9omg},
$\bar{\nu}_{n}\sigma_{n+8}=0$ \eqref{bnu6sg}
and $\nu_{n}\eta_{n+3}=0$ \eqref{n3},
we obtain
\[
\mathrm{Ind}\{\bar{\nu}_{n},\sigma_{n+8},\bar{\nu}_{n+15}\}
=\mathrm{Ind}\{\nu_{n},\eta_{n+3},\bar{\sigma}_{n+4}\}=0
\]
and 
\[
\{\bar{\nu}_{n},\sigma_{n+8},\bar{\nu}_{n+15}\}
=\{\nu_{n},\eta_{n+3},\bar{\sigma}_{n+4}\}
=\eta_n\sigma_{n+1}\eta^*_{n+8}
\text{\ \ for\ \ } n\ge 20.
\]
Moreover, by the fact $\pi^{20}_{44}\cong\pi^S_{24}$
\cite[Theorem 1.1 (b)]{MMO},
we have 
$\eta_{20}\sigma_{21}\eta^*_{28}=\eta_{20}\eta^*_{21}\sigma_{37}$.
This completes the proof of Theorem \ref{main}.

\section{Proof of Theorem \ref{main2}}

%We recall the element \cite[p. 187]{Mi}
%\begin{equation}\label{sg*}
%\sigma^*_{16}\in\{\sigma_{16},2\sigma_{23},\sigma_{30}\}_1. 
%\end{equation}

%%%%%%%%%%%%%%%%%%%%%%%%%%%%%% 22 不要？
%We recall \cite[II-Proposition 2.1 (4)]{Od}:
%\begin{equation}\label{sg14ro}
%4\sigma_{14}\rho_{21}=2\sigma_{15}\rho_{22}=\sigma_{17}\rho_{24}=0.
%\end{equation}

We show 
\begin{lem}\label{a1}
$\bar{\nu}_7\eta^{*\prime}=0$. 
\end{lem}
\begin{proof}
Since $\Sigma:\pi_{31}^{7}\to\pi_{32}^{8}$ is a monomorphism \cite[Theorem 1.1(a)]{MMO},
it suffices to show $\bar{\nu}_8\Sigma\eta^{*\prime}=0$.
By relations
$\Sigma\eta^{*\prime}\equiv [\iota_{16},\eta_{16}]
\bmod\ \sigma_{16}\mu_{23}$
\cite[Lemma 2.10]{Og},
$\bar{\nu}_{23}\eta_{31}=\nu^3_{23}$ \eqref{etbn},
$\bar{\nu}_{15}\nu_{23}=0$ \cite[(7.22)]{T}
and $\bar{\nu}_8\sigma_{16}=0$ \eqref{bnu6sg},
we have
\[
\begin{split}
\bar{\nu}_8\Sigma\eta^{*\prime}
&\equiv \bar{\nu}_8\circ [\iota_{16},\eta_{16}]
= [\bar{\nu}_8,\bar{\nu}_8]\circ \eta_{31}
= [\iota_{8},\iota_{8}]\circ \Sigma(\bar{\nu}_7\wedge \bar{\nu}_7)\circ \eta_{31}\\
&= [\iota_{8},\iota_{8}]\circ \bar{\nu}_{15}^2\eta_{31}
= [\iota_{8},\iota_{8}]\circ \bar{\nu}_{15}\nu^3_{23} = 0
\ \bmod\ \bar{\nu}_8\sigma_{16}\mu_{23} = 0.
\end{split}
\]
This completes the proof. 
\end{proof}

Next, we show
\begin{lem}\label{bnuom7}
$\bar{\nu}_6\omega_{14}\equiv P(\lambda+\xi_{13})\eta_{29}\ \bmod\ \nu_6\sigma_9\kappa_{16},4\bar{\zeta}'_6$. 
\end{lem}
\begin{proof}
Apply \cite[Corollary 5.10]{Mi} to the case $\alpha=\bar{\nu}_6,\ \beta=\omega_{14},\ n=5,\ p=13,\ i=29$. Then, by the relation
$H(\omega_{14})=\nu_{27}$ \eqref{Homega}
and 
we have 
\[
H(\bar{\nu}_6\omega_{14})\equiv H(\bar{\nu}_6)\omega_{14}+\bar{\nu}^2_{11}\nu_{27}
\ \bmod\ %G,
\sum^6_{k=3}{f_k}_*\pi^{5k+1}_{30}+\Ker\{ \Sigma: \pi^{10}_{29}\to\pi^{11}_{30}\}
\]
where $f_k:S^{5k+1}\to S^{11}$.
By the relation $H(\bar{\nu}_6)\equiv \nu_{11}\bmod 2\nu_{11}$ \cite[Lemma 6.2]{T},
we put $H(\bar{\nu}_6)=\nu_{11}+2a\nu_{11}$ for an integer $a$.
Then, by using \cite[Theorem 5.15]{WG} and relations 
$H(\omega_{14})=\nu_{27}$,
$(2\nu_{11})\omega_{14}=0$ \eqref{2nu11ome}
and $2\nu^2_{11}=0$ \eqref{2nu5nu}, we have
\[
\begin{split}
H(\bar{\nu}_6)\omega_{14}
&=(\nu_{11}+2a\nu_{11})\omega_{14}
=\nu_{11}\omega_{14}+(2a\nu_{11})\omega_{14}+[\nu_{11},2a\nu_{11}]\circ H(\omega_{14})\\
&=\nu_{11}\omega_{14}+a\iota_{11}\circ(2\nu_{11})\omega_{14}
+[\nu_{11},2a\nu_{11}]\circ\nu_{27}
=\nu_{11}\omega_{14}+[\nu_{11},2a\nu^2_{11}]=\nu_{11}\omega_{14}.
\end{split}
\]
We have $\bar{\nu}_{19}\nu_{27}=0$ \cite[(7.22)]{T} 
and 
$\Ker\{ \Sigma: \pi^{10}_{29}\to\pi^{11}_{30}\}=0$ \cite[Theorem 12.23]{T}. 
We also have $f_3\in\pi^{11}_{16}=0$ \cite[Proposition  5.9]{T}
and 
$\pi^{5k+1}_{30}=0$ for $k=5,6$ \cite[Proposition 5.8]{T}. 
Hence, we obtain
\[
H(\bar{\nu}_6\omega_{14})\equiv\nu_{11}\omega_{14}\ \bmod\ {f_4}_*\pi^{21}_{30}.
\]
By \cite[Theorems 7.2, 7.3, 7.6]{T}, we have
\[
\begin{split}
{f_4}_*\pi^{21}_{30}
&\subset\pi^{11}_{21}\circ\pi^{21}_{30}
=\{\sigma_{11}\nu_{18},\eta_{11}\mu_{12}\}
\circ\{\nu^3_{21},\mu_{21},\eta_{21}\varepsilon_{22}\}\\
&\subset\sigma_{11}\circ\pi^{18}_{30}
+\eta_{11}\mu_{12}\circ\{\nu^3_{21},\mu_{21},\eta_{21}\varepsilon_{22}\}
=\{\eta_{11}\nu^3_{12}\mu_{21},
\eta_{11}\mu^2_{12},\eta^2_{11}\mu_{13}\varepsilon_{22}\}.
\end{split}
\]
By relations
$\eta_{11}\nu_{12}=0$ \eqref{n2},
$\eta_{11}\mu^2_{12}=4\bar{\zeta}_{11}$ \eqref{etmumu}
and
$\eta^2_{11}\mu_{13}\varepsilon_{22}=0$ \eqref{n5sn^3},
we have ${f_4}_*\pi^{21}_{30}\subset\{4\bar{\zeta}_{11}\}$ and
\[
H(\bar{\nu}_6\omega_{14})\equiv\nu_{11}\omega_{14}
\ \bmod\ 4\bar{\zeta}_{11}.
\]

By relations 
$H(P(\lambda+\xi_{13})\eta_{29})=\nu_{11}\omega_{14}$
(Lemma \ref{HPx13} (3)), 
$H(\bar{\zeta}'_6)\equiv \bar{\zeta}_{11}\ \bmod\ 2\bar{\zeta}_{11}$ \cite[(3.8)]{MMO}
and $8\bar{\zeta}_{11}=0$ \cite[Theorem 12.23]{T},
we have
\[
H(\bar{\nu}_6\omega_{14})\equiv
H(P(\lambda+\xi_{13})\eta_{29}) \ \bmod H(4\bar{\zeta}'_6).
\]
Therefore, by the EHP-sequence and 
$8\bar{\zeta}'_6=0$ \cite[Theorem 1.1 (b)]{MMO}, we have
\[
\bar{\nu}_6\omega_{14}\equiv 
P(\lambda+\xi_{13})\eta_{29}+4a\bar{\zeta}'_6 \ \bmod\ \Sigma\pi^5_{29}
\text{\ \ for some $a\in\{0,1\}$.}
\]
So, by the fact 
$\Sigma\pi^5_{29}=\{\delta_6,\bar{\mu}_6\sigma_{23},\nu_6\sigma_9\kappa_{16}\}\cong(\Z_2)^3$
\cite[p.~31]{MMO}, we can put
\[
\bar{\nu}_6\omega_{14}=
P(\lambda+\xi_{13})\eta_{29}+4a\bar{\zeta}'_6
+b\delta_6+c\bar{\mu}_6\sigma_{23}+d\nu_6\sigma_9\kappa_{16}
\text{\ \ for some $b$, $c$, $d\in\{0,1\}$.}
\]
Applying $\Sigma^5$, by relations
$\bar{\nu}_{11}\omega_{19}=0$ \eqref{bnu9omg},
$\Sigma(P(\lambda+\xi_{13})\eta_{29})=0$,
$2\bar{\zeta}'_7=0$, 
$\Sigma^6\pi^5_{29}=\{\delta_{11},\bar{\mu}_{11}\sigma_{28}\}
\cong(\Z_2)^2$ \cite[Theorem 1.1 (b)]{MMO}
and 
$\nu_{11}\sigma_{14}\kappa_{21}=0$ \eqref{nu9sgka},
we have $0=b\delta_{11}+c\bar{\mu}_{11}\sigma_{28}$ and $b=c=0$.
Hence, we obtain
\[
\bar{\nu}_6\omega_{14}=
P(\lambda+\xi_{13})\eta_{29}+4a\bar{\zeta}'_6+d\nu_6\sigma_9\kappa_{16}.
\]
This completes the proof. 
\end{proof}

By using \cite[Proposition 1.2 iv)]{T} and
relations 
$\{\eta_{15},2\sigma_{16},\sigma_{23}\}=\omega_{15}+\eta^{*\prime}+\{\sigma_{15}\mu_{22}\}$
(Lemma \ref{a0} (1)),
$\bar{\nu}_7\eta^{*\prime}=0$ (Lemma \ref{a1}),
$\bar{\nu}_7\sigma_{15}=0$ \eqref{bnu6sg}, 
$\bar{\nu}_7\eta_{15}=\nu^3_7$ \eqref{etbn}
and 
$\{\nu^2_{10},2\sigma_{16},\sigma_{23}\}
\ni\sigma_{10}\kappa_{17}$
(Lemma \ref{a0} (4)), we have
\[
\bar{\nu}_7\omega_{15}
\in\bar{\nu}_7\circ\{\eta_{15},2\sigma_{16},\sigma_{23}\}
\subset\{\nu^3_7,2\sigma_{16},\sigma_{23}\}
\supset\nu_7\circ\{\nu^2_{10},2\sigma_{16},\sigma_{23}\}
\ni \nu_7\sigma_{10}\kappa_{17}.
\]
By $\nu_{13}\circ\pi^{16}_{31}=0$ (Lemma \ref{a0} (2))
and \cite[Theorem 12.7]{T}, we have
\[
\mathrm{Ind}\{\nu^3_7,2\sigma_{16},\sigma_{23}\}
=\nu^3_7\circ\pi^{16}_{31}+\pi^7_{24}\circ\sigma_{24}
=\{\sigma'\eta_{14}\mu_{15},\nu_7\kappa_{10},
\bar{\mu}_7,\eta_7\mu_8\sigma_{17}\}\circ\sigma_{24}.
\]
So, by relations
$\sigma'\eta_{14}\mu_{15}\sigma_{24}=\bar{\zeta}'_7$ \eqref{bz'7},
$\kappa_{10}\sigma_{24}=0$ \eqref{k7s}
and $\mu_8\sigma^2_{17}=0$ \eqref{mu3s^2}, we have
$\mathrm{Ind}\{\nu^3_7,2\sigma_{16},\sigma_{23}\}
=\{\bar{\zeta}'_7,\bar{\mu}_7\sigma_{24}\}$ and 
\[
\bar{\nu}_7\omega_{15}\equiv \nu_7\sigma_{10}\kappa_{17}
\ \bmod\ \bar{\zeta}'_7,\bar{\mu}_7\sigma_{24}.
\]
The first $3$ elements become trivial and the last survives in the stable range,
by facts $\bar{\nu}_9\omega_{17}=0$ \eqref{bnu9omg},
$\nu_{11}\sigma_{14}=0$ \eqref{nu10sg},
$\bar{\zeta}'_9=\Sigma P(\sigma_{17}\eta_{24}\mu_{25})=0$ \cite[(5.11)]{MMO}
and $\bar{\mu}\sigma\neq 0$ \cite[Theorem 1.1 (b)]{MMO}. This induces the relation
\[
\bar{\nu}_7\omega_{15}\equiv\nu_7\sigma_{10}\kappa_{17}\ \bmod\ \bar{\zeta}'_7.
\]
On the other hand, by Lemma \ref{bnuom7} and 
$2\bar{\zeta}'_7=0$ \cite[Theorem 1.1 (b)]{T}, we have
\[
\bar{\nu}_7\omega_{15}\equiv 0 \ \bmod\ \nu_7\sigma_{10}\kappa_{17}.
\]
Hence we obtain the equation $\bar{\nu}_7\omega_{15}=\nu_7\sigma_{10}\kappa_{17}$.
This completes the proof of Theorem \ref{main2}.
%By Lemmas \ref{bnuom7}, \ref{b1}, we get Theorem \ref{main2}(2). 


\begin{thebibliography}{99}
%\bibitem{Ba}
%{M.~G.~Barratt},
%{Note on a formula due to Toda},
%J. London Math. Soc. \textbf{36}(1961), 95--96.
%\bibitem{BH}
%{M. G. Barratt} and {J. P. Hilton}: On join operations in homotopy groups,
% Proc. London Math. Soc. (3) \textbf{3}(1953), 430--445. 
%\bibitem{Bo}
%{R.~Bott},
%The stable homotopy of the classical groups,
%Ann.\ of Math.\ \textbf{70} (1959), 313--337.
\bibitem{GM}
{M.~Golasi\'nski} and {J.~Mukai},
Gottlieb groups of spheres, 
Topology \textbf{47} (2008), 399--430.
%\bibitem{GM2014} 
%M.~Golasi\'nski and J.~Mukai, 
%\textit{Gottlieb and Whitehead Center Groups of Spheres, Projective  and Moore Spaces},
%Springer Cham Heidelberg New York Dordrecht London, 2014.
%\bibitem{H}
%{P.~J.~Hilton},
%A note on the $P$-homomorphism in homotopy groups of spheres,
%Proc.\ Camb.\ Phil.\ Soc.\ \textbf{59} (1955) 230--233.
\bibitem{HKM}
{Y.~Hirato}, {H.~Kachi} and {J.~Mukai},
$21$-st and $22$-nd  homotopy groups of the $n$-th rotation group,
\ J. Fac. Sci. Shinshu Univ.\ \textbf{41} (2006), 1--28.
\bibitem{IMM}
{T.~Inoue}, {T.~Miyauchi} and {J.~Mukai},
The $2$-components of the $31$-stem homotopy groups of the $9$ and $10$-spheres,
J. Fac. Sci. Shinshu Univ. \textbf{46} (2015), 1--19.
\bibitem{IM}
{T.~Inoue} and {J.~Mukai},
A note on the Hopf homomorphism of a Toda bracket and its application,
Hiroshima Math. J. \textbf{33} (2003), 379--389.
\bibitem{Ja1}
I.~M.~James,
On the homotopy groups of certain pairs and triads,
Quart. J. Math. Oxford \textbf{3} (1954), 260--270.
%\bibitem{J}
%I. M.~James: On the suspension triad, Annals of Math., 
%{\bf 63}-2 (1956), 191--247.
\bibitem{K}
{H.~Kachi},
On the homotopy groups of rotation groups $R_n$,
J. Fac. Sci. Shinshu Univ.\ \textbf{3} (1968), 13--33.
%\bibitem{HM}
%{Y. Hirato} and {J. Mukai}: Some Toda bracket in $\pi^S_{26}(S^0)$, \ Math. J. Okayama Univ.\ \textbf{42} (2000), 83--88.
%\bibitem{Ka} {H. Kachi}: On the homotopy groups of rotation groups $R_{n}$, J. Fac. Sci., Shinshu Univ. {\bf 3} (1968), 13--33.
%\bibitem{KM}
%{H. Kachi} and {J. Mukai}: Some homotopy groups of the rotation group, \ J. Math. Hiroshima Univ.\ \textbf{41} (1999), 327--345.
%\bibitem{KM2}
%H.~Kachi and J.~Mukai,
%The $19$ and $20$-th homotopy groups of the rotation groups $R_n$,
%Math. J. Okayama Univ. \textbf{42} (2000), 89--113.
\bibitem{Ke}
{M.~A.~Kervaire},
Some nonstable homotopy groups of Lie groups,
Illinois J. Math.\ \textbf{4} (1960), 161--169.
\bibitem{Mi}
{M.~Mimura},
On the generalized Hopf homomorphism and the higher composition. Part I; II,
 $\pi_{n+i}(S^n)$ for $i = 21$ and $22$,
J. Math. Kyoto Univ. \textbf{4} (1964), 171--190; \textbf{4} (1965), 301--326.
%\bibitem{Mi1} {M. Mimura}: The homotopy groups of Lie groups of low rank, J.\ Math.\ Kyoto Univ.\ {\bf 6}-2 (1967), 131--176.
\bibitem{MMO}
{M.~Mimura}, {M.~Mori} and {N.~Oda},
Determinations of $2$-components of the $23$- and $24$-stems in homotopy groups of spheres,
Mem. Fac. Sci. Kyushu Univ. \textbf{29} (1975), 1--42.
%\bibitem{MO}
%{M.~Mimura} and {N.~Oda},
%Periodic families in homotopy groups,
%J. Math. Kyoto Univ. \textbf{21} (1981), 171--187.
\bibitem{MT}
{M.~Mimura} and {H.~Toda},
The $(n+20)$-th homotopy groups of $n$-spheres,
J. Math. Kyoto Univ. \textbf{3}-1 (1963), 37--58.
\bibitem{MT2}
M.~Mimura and H.~Toda:
Homotopy groups of $SU(3), SU(4)$ and $Sp(2)$,
J. Math. Kyoto Univ. \textbf{3} (1964), 217--250.
\bibitem{MiM2} 
{T. Miyauchi} and {J. Mukai}, 
Determination of the $2$-primary components of the $32$-th homotopy groups of\ $S^n$,
Bol. Soc. Mat. Mex. (3) \textbf{23} (2017), 319--387.
%\bibitem{MM}
%{K.~Morisugi} and {J.~Mukai},
%Whitehead square of a lift of the Hopf map to a mod $2$ Moore space,
%J. Math. Kyoto Univ. \textbf{42}-2 (2002), 331--336.
\bibitem{MoM}
K.~Morisugi and J.~Mukai,
 Lifting to a mod $2$ Moore space, 
J. Math. Soc. Japan. \textbf{52}-3 (2000), 515--533.
\bibitem{Mu66}
J.~Mukai,
Stable homotopy of some elementary complexes,
Mem. Fac. Sci. Kyusyu Univ. Ser. A \textbf{3} (1966), 266--282.
%\bibitem{Mu1} {J. Mukai}: On the stable homotopy of a $Z_2$-Moore space, Osaka J. Math.
%\textbf{6}(1969), 63--91.
%\bibitem{Mu93} 
%J. Mukai, On stable homotopy of the complex projective space, 
%Japan. J. Math. (N.S.) {\bf 19} (1993), no.~1, 191--216.
%\bibitem{Mu00}
%{J.~Mukai}, 
%Homotopy from the real $(n-1)$-projective space to the $n$th rotation group,
%Kyushu J. Math. {\bf 54} (2000), no.~2, 423--428.
\bibitem{Mu2}
{J.~Mukai},
Determination of the $P$-image by Toda brackets, Geometry and Topology Monographs 
\textbf{13} (2008), 355--383. 
\bibitem{Od}
{N.~Oda},
Unstable homotopy groups of spheres, Bull. Inst. Adv. Res. Fukuoka Univ. \textbf{44} (1979), 49--152.
\bibitem{Od1}
{N.~Oda},
On the orders of the generators in the $18$-stem of the homotopy groups of spheres, Adv. Stud. Pure Math. \textbf{9} (1986), 231--236.
\bibitem{Og}
{K.~{\^O}guchi},
Generators of $2$-primary components of homotopy groups and symplectic groups,
J. Fac. Sci. Univ. Tokyo  \textbf{11} (1964), 65--111. 
\bibitem{Os82}
H.~{\^O}shima,
Some James numbers of Stiefel manifolds,
Math. Proc. Cambridge Phil. Soc., \textbf{92} (1982), 139--161.
\bibitem{Sp62}
{E.~Spanier},
Secondary operations on mappings and cohomology,
Ann. of Math. \textbf{75} (1962), 260--282.
%\bibitem{Th}
%{S.~Thomeier},
%Whitehead products and homotopy groups of spheres,
%Proc. 13-th Biennial Seminar of Can. Math. Congress, {\bf 2} (1972), 144--155. 
\bibitem{T}
{H.~Toda},
\textit{Composition methods in homotopy groups of spheres}, 
Ann. of Math. Studies, \textbf{49}, Princeton, 1962.
\bibitem{T1}
{H.~Toda}:
A survey of homotopy theory. Advances in Math. \textbf{10} (1973), 417--455. 
\bibitem{WG}
{G.~W.~Whitehead}:
A generalization of the Hopf invariant, Ann. of Math. \textbf{51} (1950), 192--237.
\bibitem{WG78}
{G.~W.~Whitehead}:
\textit{Elements of Homotopy Theory}, 
Graduate Texts in Mathematics \textbf{61}, 
Springer Verlag, Berlin, 1978.
\bibitem{WJ}
J.~H.~C.~Whitehead:
On certain theorems of G.~W.~Whitehead,
Ann. of Math. \textbf{58} (1953), 418--428.
\bibitem{YMW}
J. ~Yang, J.~Mukai and J.~Wu:
On the homotopy groups of the suspended quaternionic projective plane and applications,
Algebr. Geom. Topol. \textbf{25} (2025), 2981--3033.
\end{thebibliography}
\end{document}